\documentclass[oneside,english]{amsart}
\usepackage[T1]{fontenc}
\usepackage[latin9]{inputenc}
\usepackage{geometry}
\usepackage[active]{srcltx}
\usepackage{babel}
\usepackage{longtable}
\usepackage{amstext}
\usepackage{amsthm}
\usepackage[unicode=true,pdfusetitle,
 bookmarks=true,bookmarksnumbered=false,bookmarksopen=false,
 breaklinks=false,pdfborder={0 0 1},backref=false,colorlinks=false]
 {hyperref}

\makeatletter

\providecommand{\tabularnewline}{\\}

\numberwithin{equation}{section}
\numberwithin{figure}{section}
\theoremstyle{plain}
\newtheorem{conjecture}{\protect\conjecturename}[section]
\theoremstyle{plain}
\newtheorem{thm}{\protect\theoremname}[section]
\theoremstyle{remark}
\newtheorem*{rem*}{\protect\remarkname}
\theoremstyle{plain}
\newtheorem{lem}{\protect\lemmaname}[section]
\theoremstyle{plain}
\newtheorem{prop}{\protect\propositionname}[section]

\usepackage{babel}
\usepackage{amssymb}

\makeatother

\providecommand{\conjecturename}{Conjecture}
\providecommand{\lemmaname}{Lemma}
\providecommand{\propositionname}{Proposition}
\providecommand{\remarkname}{Remark}
\providecommand{\theoremname}{Theorem}

\begin{document}
\title[Conjectures of Schlosser and Zhou]{Some conjectures of Schlosser and Zhou on sign patterns of the coefficients
of infinite products}
\author{Bing He}
\address{School of Mathematics and Statistics, HNP-LAMA, Central South University
\\
 Changsha 410083, Hunan, People's Republic of China}
\email{yuhelingyun@foxmail.com; yuhe123456@foxmail.com}
\author{Linpei Li}
\address{School of Mathematics and Statistics, HNP-LAMA, Central South University\\
 Changsha 410083, Hunan, People's Republic of China}
\email{232111035@csu.edu.cn}
\begin{abstract}
Recently, Schlosser and Zhou proposed many conjectures on sign patterns
of the coefficients appearing in the $q$-series expansions of the
infinite Borwein product and other infinite products raised to a real
power. In this paper, we will study several of these conjectures.
Let 
\[
G(q):=\prod_{i=1}^{I}\left(\prod_{k=0}^{\infty}(1-q^{m_{i}+kn_{i}})(1-q^{-m_{i}+(k+1)n_{i}})\right)^{u_{i}}
\]
where $I$ is a positive integer, $1\leq m_{i}<n_{i}$ and $u_{i}\neq0$
for $1\leq i\leq I$ and $|q|<1.$ We will establish an asymptotic
formula for the coefficients of $G(q)^{\delta}$ with $\delta$ being
a positive real number by using the Hardy--Ramanujan--Rademacher
circle method. As applications, we apply the asymptotic formula to
confirm some of the conjectures of Schlosser and Zhou.
\end{abstract}

\maketitle

\section{Introduction}

\subsection{Background and motivations}

Let $\mathcal{H}$ denote the upper half-plane $\{z\in\mathbb{C}:\Im(z)>0\}$
and let $q=e^{2\pi i\tau}$ with $\tau\in\mathcal{H}.$ The Dedekind
eta function $\eta(\tau)$ is defined by 
\[
\eta(\tau):=e^{\frac{\pi i\tau}{12}}\prod_{n=1}^{\infty}\left(1-e^{2\pi i\tau}\right)=q^{1/24}(q;q)_{\infty},
\]
 where 
\[
(a;q)_{\infty}:=\prod_{i=0}^{\infty}(1-aq^{i}).
\]
Let $p(n)$ denote the number of unrestricted partitions. We known
that $p(n)$ has an interesting generating function: 
\[
\sum_{n=0}^{\infty}p(n)q^{n}=\frac{1}{(q;q)_{\infty}}.
\]

Applying the modularity of the Dedekind eta function $\eta(\tau),$
Hardy and Ramanujan \cite{HR} and Rademacher \cite{R} proved the
famous asymptotic formula: 
\[
p(n)=\frac{1}{2\sqrt{2}\pi}\sum_{k=1}^{\infty}A_{k}(n)\sqrt{k}\frac{\mathrm{d}}{\mathrm{d}n}\left(\frac{2}{\sqrt{n-\frac{1}{24}}}\sinh\left(\frac{\pi}{k}\sqrt{\frac{2}{3}\left(n-\frac{1}{24}\right)}\right)\right),
\]
where 
\[
A_{k}(n):=\sum_{\substack{0\leq h<k\\
\gcd(h,k)=1
}
}e^{\pi i\left(s(h,k)-2nh/k\right)}
\]
with $s(h,k)$ being the Dedekind sum defined as 
\[
s(d,c):=\sum_{n(\mathrm{mod}c)}\left(\left(\frac{dn}{c}\right)\right)\left(\left(\frac{n}{c}\right)\right).
\]
Here 
\[
\left(\left(x\right)\right):=\begin{cases}
x-\left\lfloor x\right\rfloor -1/2 & \text{if}\quad x\notin\mathbb{Z},\\
0 & \text{if}\quad x\in\mathbb{Z},
\end{cases}
\]
with $\left\lfloor x\right\rfloor $ being the greatest integer not
exceeding the real number $x.$

In 1993, P. Borwein investigated modular forms and considered the
Fourier coefficients of the infinite product 
\[
G_{p}(q):=\frac{(q;q)_{\infty}}{(q^{p};q^{p})_{\infty}}=:\sum_{n\geq0}c_{p}(n)q^{n}
\]
with $p$ being a prime, see, for example, \cite{A}. The infinite
product $G_{p}(q)$ was called the infinite Borwein product by Schlosser
and Zhou \cite{SZ}. We can see that $(q^{p};q^{p})_{\infty}/(q;q)_{\infty}$
is the generating function for the number of partitions into parts
that are not a multiple of $p.$ In \cite[Theorem 2.1]{A}, Andrews
proved that $c_{p}(n)$ and $c_{p}(n+p)$ have the same sign for $n\geq0$
and for all primes $p.$

In 2019, with the help of computer algebra, Schlosser proposed the
following conjecture \cite[ Conjecture 1]{S} on the Fourier coefficients
of the infinite Borwein product raised to a real power.
\begin{conjecture}
\label{conj:S} Let $\delta$ be a real number satisfying 
\[
0.227998127341\cdots\approx\frac{9-\sqrt{73}}{2}\leq\delta\leq1\quad\text{or}\quad2\leq\delta\leq3.
\]
Then the series $A^{\delta}(q),B^{\delta}(q),C^{\delta}(q)$ appearing
in the dissection 
\[
G_{3}(q)^{\delta}=A^{\delta}(q^{3})-qB^{\delta}(q^{3})-q^{2}C^{\delta}(q^{3})
\]
are power series in $q$ with non-negative real coefficients.
\end{conjecture}
It should be mentioned that the irrational number $\frac{9-\sqrt{73}}{2}$
comes from the coefficient of $q^{3}$ in $G_{3}(q)^{\delta}.$

In \cite{SZ}, Schlosser and Zhou applied the circle method to asymptotically
estimate the coefficients of the infinite Borwein product raised to
a real power. Although it seems that they only partially affirmed
Conjecture \ref{conj:S}, the asymptotic formula in \cite[Theorem 3]{SZ}
can be used to confirm this conjecture and \cite[Conjecture 17]{SZ}
completely and to settle \cite[Conjectures 20 and 23]{SZ} partially.
Some additional results were also derived in that paper. At the end
of their paper, they utilized computer algebra to pose many conjectures
in the appendix. These conjectures predict precise sign patterns of
the coefficients appearing in the the $q$-series expansions of the
infinite Borwein product and other infinite products raised to a real
power. Many of these conjectures concern precise sign patterns of
the coefficients of the infinite product
\[
\frac{(q^{s},q^{k-s};q^{k})_{\infty}^{\delta}}{(q^{t},q^{k-t};q^{k})_{\infty}^{\delta}},
\]
where 
\[
(a_{1},a_{2},\cdots,a_{m};q)_{\infty}:=(a_{1};q)_{\infty}(a_{2};q)_{\infty}\cdots(a_{m};q)_{\infty}.
\]

Numerous studies have concentrated on infinite products of this form.
Richmond and Szekeres \cite{RS} considered the power series expansions
of the infinite products
\begin{align*}
\frac{(q^{3},q^{5};q^{8})_{\infty}}{(q^{1},q^{7};q^{8})_{\infty}} & =:\sum_{n=0}^{\infty}a(n)q^{n},\\
\frac{(q^{1},q^{7};q^{8})_{\infty}}{(q^{3},q^{5};q^{8})_{\infty}} & =:\sum_{n=0}^{\infty}b(n)q^{n}
\end{align*}
and obtained in \cite[Theorem 5.1]{RS} that $a(4m+3)$ and $b(4m+2)$
are always zero for $m\geq0.$ Furthermore, they made conjectures
\cite[p. 367]{RS} on vanishing of the coefficients of the infinite
products
\[
\frac{(q^{5},q^{7};q^{12})_{\infty}}{(q^{1},q^{11};q^{12})_{\infty}}
\]
and
\[
\frac{(q^{1},q^{11};q^{12})_{\infty}}{(q^{5},q^{7};q^{12})_{\infty}}.
\]
The conjectures of Richmond and Szekeres were resolved by Andrew and
Bressound \cite{AB}, who derived a more general result on the vanishing
coefficients in the infinite product
\[
\frac{(q^{r},q^{2k-r};q^{2k})_{\infty}}{(q^{k-r},q^{k+r};q^{2k})_{\infty}}.
\]

 In \cite{RS}, Richmond and Szekeres also derived an asymptotic
formula for the coefficients of the Rogers--Ramanujan continued fraction,
which is defined by 
\[
R(q):=\dfrac{q^{1/5}}{1+\dfrac{q}{1+\dfrac{q^{2}}{1+\dfrac{q^{3}}{1+\cdots}}}}=q^{1/5}\frac{(q,q^{4};q^{5})_{\infty}}{(q^{2},q^{3};q^{5})_{\infty}}.
\]
If 
\[
\frac{(q,q^{4};q^{5})_{\infty}}{(q^{2},q^{3};q^{5})_{\infty}}=:\sum_{n=0}^{\infty}c(n)q^{n},
\]
then Richmond and Szekeres showed that 
\[
c(n)\sim\frac{2^{1/2}}{5^{3/4}}\cos\left(\frac{4\pi}{5}\left(n+\frac{3}{20}\right)\right)n^{-3/4}\exp\left(\frac{4\pi}{5}\sqrt{\frac{n}{5}}\right).
\]
Hence $c(5n),c(5n+2)>0,$ and $c(5n+1),c(5n+3),c(5n+4)<0$ for sufficiently
large $n.$

Motivated by Richmond--Szekeres \cite{RS} and Andrew--Bressound
\cite{AB}, we focus on several conjectures from \cite{SZ}.
\begin{conjecture}
\label{conj:1} Let 
\[
Q_{5}(q):=\frac{(q,q^{4};q^{5})_{\infty}}{(q^{2},q^{3};q^{5})_{\infty}}.
\]
Then the $q$-series coefficients of $Q_{5}(q)^{\delta}$ exhibit
the sign pattern $+-+--$ for 
\[
1\leq\delta\leq\frac{\sqrt{97}-5}{2}\approx2.424428900898\cdots
\]
For 
\[
2.571366313289\cdots\approx\alpha\leq\delta\leq4,
\]
they exhibit the sign pattern $+-+-+.$ Here $\alpha$ is the unique
real root of the polynomial 
\[
x^{7}+35x^{6}+7x^{5}-6055x^{4}-14336x^{3}+104300x^{2}-184752x+282240
\]
that satisfies $2<\alpha<3.$ For $\delta=-1$ $Q_{5}(q)^{\delta}$
exhibit the sign pattern $++---,$ and for $-3\leq\delta\leq-2$ the
sign pattern $+++--.$
\end{conjecture}
It should be mentioned that the irrational number $\frac{\sqrt{97}-5}{2}$
arises from the coefficient of $q^{4}$ in $Q_{5}(q)^{\delta}$ and
the constant $\alpha$ originates from the coefficient of $q^{9}$
in $Q_{5}(q)^{\delta}.$
\begin{conjecture}
\label{conj:2} Let 
\[
Q_{6}(q):=(q,q^{5};q^{6})_{\infty}.
\]
Then the $q$-series coefficients of $Q_{6}(q)^{\delta}$ exhibit
the sign pattern $(+-)^{3}$ for all $\delta\geq3.$
\end{conjecture}
\begin{conjecture}
\label{conj:3} Let 
\[
Q_{8}(q):=\frac{(q,q^{7};q^{8})_{\infty}}{(q^{3},q^{5};q^{8})_{\infty}}.
\]
Then the q-series coefficients of $Q_{8}(q)^{\delta}$ exhibit for
$\delta=2$ the length $16$ sign pattern $+-++-+--+--+-++-.$ For
\[
2.664479110226972\cdots\approx\beta\leq\delta\leq4
\]
they exhibit the sign pattern $+-++-+--.$ Here $\beta$ is the unique
real root of the polynomial 
\[
\ensuremath{\begin{aligned} & x^{12}-90x^{11}+1457x^{10}+30486x^{9}-537081x^{8}\\
 & +1892346x^{7}-3683653x^{6}-837509646x^{5}+774767020x^{4}\\
 & +3333687384x^{3}-40887173664x^{2}+94379731200x+49816166400
\end{aligned}
}
\]
that satisfies $2<\beta<3.$ For 
\[
-1<\delta\leq\frac{7-\sqrt{73}}{2}\approx-0.77200187265877\cdots
\]
they exhibit the sign pattern $+++----+.$ For $\delta=-2$ they exhibit
the length $16$ sign pattern $+++++----+++----.$
\end{conjecture}
It should be mentioned that the irrational number $\frac{7-\sqrt{73}}{2}$
comes from the coefficient of $q^{4}$ in $Q_{8}(q)^{\delta}$ and
the constant $\beta$ arises from the coefficient of $q^{14}$ in
$Q_{8}(q)^{\delta}.$ 
\begin{conjecture}
\label{conj:4} Let 
\[
Q_{10}(q):=\frac{(q,q^{9};q^{10})_{\infty}}{(q^{3},q^{7};q^{10})_{\infty}}.
\]
Then the $q$-series coefficients of $Q_{10}(q)^{\delta}$ exhibit
for $\delta=1$ the sign pattern $+-++--+--+.$ And for $\delta=-1$
the sign pattern $++++-----+.$
\end{conjecture}
\begin{conjecture}
\label{conj:5} Let 
\[
Q_{12}(q):=\frac{(q,q^{11};q^{12})_{\infty}}{(q^{5},q^{7};q^{12})_{\infty}}.
\]
Then the $q$-series coefficients of $Q_{12}(q)^{\delta}$ exhibit
for $\delta=1$ the sign pattern $+-+\,0-+-+-0+-,$ for $2\leq\delta\leq3$
they exhibit the sign pattern $+-+--+-+-++-.$ For $\delta=-1$ they
exhibit the sign pattern $+++++\,0-----0,$ and for $-1<\delta<0$
the sign pattern $+++++------+.$
\end{conjecture}
Generally speaking, for infinite products raised to a real power $\delta,$
every coefficient is a polynomial in $\delta$ and the signs of these
coefficients are uncertain. However, the $q$-series coefficients
of $Q_{5}(q)^{\delta},Q_{6}(q)^{\delta},Q_{8}(q)^{\delta},Q_{10}(q)^{\delta}$
and $Q_{12}(q)^{\delta}$ for $\delta$ within the specified ranges
of real numbers exhibit regular sign patterns. But, polynomiality
of the $q$-series coefficients of $Q_{5}(q)^{\delta},Q_{6}(q)^{\delta},Q_{8}(q)^{\delta},Q_{10}(q)^{\delta}$
and $Q_{12}(q)^{\delta}$ makes proofs of Conjectures \ref{conj:1}--\ref{conj:5}
very difficult.

Up to now, as far as we know, no proofs for these conjectures have
been given.

In order to study these conjectures, we investigate the Fourier coefficients
of a real power of a genegal class of infinite products: 
\[
G(q):=\prod_{i=1}^{I}(q^{m_{i}},q^{n_{i}-m_{i}};q^{n_{i}})_{\infty}^{u_{i}},
\]
where $I$ is a positive integer, $\{m_{i}\}_{i=1}^{I}$ and $\{n_{i}\}_{i=1}^{I}$
are two finite sequence of integers with $1\leq m_{i}<n_{i}$ for
$1\leq i\leq I$ and $\{u_{i}\}_{i=1}^{I}$ is a finite sequence of
nonzero integers. Asymptotic behavior of the coefficients of $G(q)$
was studied by Chern \cite{C}. Some uniform asymptotic formulas for
restricted bipartite partitions can be found in \cite{Z}. In this
paper, using the Hardy--Ramanujan--Rademacher circle method, we
will establish an asymptotic formula for $c_{\delta}(n),$ where 
\[
G(q)^{\delta}=:\sum_{n=0}^{\infty}c_{\delta}(n)q^{n},
\]
with $\delta$ being a positive real number.

\subsection{\label{subsec:1-2} Notations and our main result}

In this subsection, we first state some notations. Several of these
notations are borrowed from Chern \cite{C}. Given a real number $x,$
we define 
\[
\varPhi(x)=\begin{cases}
1, & \text{if}\;x=0,\\
x, & \text{otherwise},
\end{cases}
\]
and 
\[
\varUpsilon(x)=\begin{cases}
0, & \text{if}\;x=0,\\
x, & \text{if}\;0<x\leq1/2,\\
1-x, & \text{if}\;1/2<x<1.
\end{cases}
\]

Let $h,k$ be two integers such that $0\leq h<k$ and $\gcd(h,k)=1,$
and let $m,n$ be two positive integers. We define 
\[
\lambda_{m,n}(h,k):=\left\lceil \frac{mh}{\gcd(n,k)}\right\rceil 
\]
and
\[
\lambda_{m,n}^{*}(h,k):=\lambda_{m,n}(h,k)-\frac{mh}{\gcd(n,k)},
\]
where $\left\lceil x\right\rceil $ denotes the smallest integer not
less than the real number $x.$ For $\lambda_{m,n}^{*}(h,k),$ we
define four subsets of $\{1,2,\cdots,I\}:$
\begin{align*}
\mathcal{I}_{0}(h,k) & :=\{1\leq i\leq I\;|\;\lambda_{m_{i},n_{i}}^{*}(h,k)=0\},\\
\mathcal{I}_{1}(h,k) & :=\{1\leq i\leq I\;|\;\lambda_{m_{i},n_{i}}^{*}(h,k)\neq0\},\\
\mathcal{I}_{0}^{+}(h,k) & :=\{1\leq i\leq I\;|\;\lambda_{m_{i},n_{i}}^{*}(h,k)=0,u_{i}>0\},\\
\mathcal{I}_{0}^{-}(h,k) & :=\{1\leq i\leq I\;|\;\lambda_{m_{i},n_{i}}^{*}(h,k)=0,u_{i}<0\}.
\end{align*}

We also define
\[
q_{i}^{(1)}(h,k):=\exp\left(\frac{2\pi\mathrm{i}\left(m_{i}d_{i}+\lambda_{m_{i},n_{i}}(h,k)n_{i}h_{n_{i}}^{\prime}(h,k)d_{i}\right)}{kn_{i}}+\frac{2\pi\left(m_{i}d_{i}h-\lambda_{m_{i},n_{i}}(h,k)d_{i}^{2}\right)}{n_{i}z}\right),
\]
\[
\hat{q}_{i}^{(1)}(h,k):=\exp\left(\frac{2\pi\left(m_{i}d_{i}h-\lambda_{m_{i},n_{i}}(h,k)d_{i}^{2}\right)}{n_{i}}\right),
\]
\[
q_{i}^{(2)}(h,k):=\exp\left(\frac{2\pi\mathrm{i}h_{n_{i}}^{\prime}(h,k)d_{i}}{k}-\frac{2\pi d_{i}^{2}}{n_{i}z}\right),
\]
\[
\hat{q}_{i}^{(2)}(h,k):=\exp\left(-\frac{2\pi d_{i}^{2}}{n_{i}}\right),
\]
and 
\[
\Pi_{h,k}\text{:=\ensuremath{\prod_{i\in\mathcal{I}_{0}(h,k)}\left(1-q_{i}^{(1)}\right)^{u_{i}\delta}},}
\]
where $d_{i}=\gcd(n_{i},k)$ and $h_{n}^{\prime}(h,k)$ will be defined
in Subsection \ref{subsec:2-2}. Sometimes, for brevity, we write
$q_{i}^{(1)}(h,k),\hat{q}_{i}^{(1)}(h,k),q_{i}^{(2)}(h,k)$ and $\hat{q}_{i}^{(2)}(h,k)$
as $q_{i}^{(1)},\hat{q}_{i}^{(1)},q_{i}^{(2)}$ and $\hat{q}_{i}^{(2)}$
respectively.

We use $\mathrm{i}$ to denote $\sqrt{-1},$ and employ the following
notations: 
\[
\Omega:=\sum_{i=1}^{I}u_{i}\left(\frac{12m_{i}^{2}}{n_{i}}-12m_{i}+2n_{i}\right),
\]
\[
\omega_{h.k}:=\prod_{i=1}^{I}\exp\left(-2u_{i}\pi\mathrm{i}s\left(\frac{n_{i}h}{\gcd(n_{i},k)},\frac{k}{\gcd(n_{i},k)}\right)\right),
\]
\[
\Delta(h,k):=\sum_{i=1}^{I}u_{i}\left(\frac{12\gcd(n_{i},k)^{2}}{n_{i}}(\lambda_{m_{i},n_{i}}^{*}(h,k)-\lambda_{m_{i},n_{i}}^{*2}(h,k))-\frac{2\gcd(n_{i},k)^{2}}{n_{i}}\right),
\]
and
\begin{align*}
\varTheta_{h,k} & :=\prod_{i=1}^{I}\exp\left(u_{i}\pi\mathrm{i}\left(\frac{m_{i}h}{k}-\frac{m_{i}\gcd(n_{i},k)}{kn_{i}}+2\frac{\lambda_{m_{i},n_{i}}^{*}(h,k)m_{i}\gcd(n_{i},k)}{kn_{i}}\right.\right.\\
 & \qquad\qquad\qquad\left.\left.+\frac{(\lambda_{m_{i},n_{i}}^{2}(h,k)-\lambda_{m_{i},n_{i}}(h,k))h_{n_{i}}^{\prime}(h,k)\gcd(n_{i},k)}{k}\right)\right).
\end{align*}

Let $L=\text{lcm}(n_{1},\cdots,n_{I}).$ We define two disjoint sets:
\[
\mathcal{L}_{\leq0}:=\{(\varkappa,\kappa)\;|\;0\leq\varkappa<\kappa,1\leq\kappa\leq L,\Delta(\varkappa,\kappa)\leq0\},
\]
\[
\mathcal{L}_{>0}:=\{(\varkappa,\kappa)\;|\;0\leq\varkappa<\kappa,1\leq\kappa\leq L,\Delta(\varkappa,\kappa)>0\}.
\]

We now present our main result.
\begin{thm}
\label{thm:1} Let $N$ be a positive integer and let $\delta$ be
a positive real number. If the inequality 
\begin{equation}
\min_{1\leq i\leq I}\left(\varUpsilon\left(\lambda_{m_{i},n_{i}}^{*}(h,k)\right)\frac{d_{i}^{2}}{n_{i}}\right)\geq\frac{\delta\Delta(\varkappa,\kappa)}{24}\label{eq:thm1}
\end{equation}
holds for all $1\leq\kappa\leq L$ and $0\leq\varkappa<\kappa$ such
that $(\varkappa,\kappa)\in\mathcal{L}_{>0},$ then for $n>-\delta\Omega/24,$
we have
\begin{align*}
c_{\delta}(n) & =\sum_{(\varkappa,\kappa)\in\mathcal{L}_{>0}}\sum_{\substack{1\leq k\leq N\\
k\equiv\kappa(\mathrm{mod}L)
}
}\sum_{\substack{0\leq h\leq k\\
h\equiv\varkappa(\mathrm{mod}\kappa)
}
}\frac{2\delta\pi}{k}e^{-2\pi inh/k}\mathrm{i}^{\delta\sum_{j=1}^{I}u_{j}}(-1)^{\delta\sum_{j=1}^{I}u_{j}\lambda_{m_{j},n_{j}}(h,k)}\\
 & \;\times\omega_{h.k}^{\delta}\varTheta_{h,k}^{\delta}\Pi_{h,k}\left(\frac{\Delta(\varkappa,\kappa)}{24n\delta+\delta^{2}\Omega}\right)^{1/2}I_{-1}\left(\frac{\pi}{6k}\sqrt{\text{\ensuremath{\Delta(\varkappa,\kappa)}}(24n\delta+\delta^{2}\Omega)}\right)+E_{\delta,N}(n),
\end{align*}
the error term $E_{\delta,N}(n)$ satisfies the bound 
\begin{align*}
\left|E_{\delta,N}(n)\right| & \leq\sum_{(\varkappa,\kappa)\in\mathcal{L}_{\leq0}}\sum_{\substack{1\leq k\leq N\\
k\equiv\kappa(\mathrm{mod}L)
}
}\sum_{\substack{0\leq h\leq k\\
h\equiv\varkappa(\mathrm{mod}\kappa)
}
}\frac{2}{k(N+1)}\exp\left(\frac{24\pi n+\delta\pi\Omega}{6N^{2}}\right)\\
 & \;\times\exp\left(\frac{\delta\pi}{12}\Delta(\varkappa,\kappa)\right)\prod_{i\in\mathcal{I}_{0}^{+}(h,k)}2^{u_{i}\delta}\prod_{i\in\mathcal{I}_{0}^{-}(h,k)}\left(1-e^{2\pi\mathrm{i}/n_{i}}\right)^{u_{i}\delta}\\
 & \;\times\prod_{i\in\mathcal{I}_{0}(h,k)}\left(\hat{q}_{i}^{(1)}\hat{q}_{i}^{(2)},\left(\hat{q}_{i}^{(1)}\right)^{-1}\hat{q}_{i}^{(2)};\hat{q}_{i}^{(2)}\right)^{-|u_{i}|\delta}\\
 & \;\times\prod_{i\in\mathcal{I}_{1}(h,k)}\left(\hat{q}_{i}^{(1)},\left(\hat{q}_{i}^{(1)}\right)^{-1}\hat{q}_{i}^{(2)};\hat{q}_{i}^{(2)}\right)^{-|u_{i}|\delta}\\
 & +\sum_{(\varkappa,\kappa)\in\mathcal{L}_{>0}}\sum_{\substack{1\leq k\leq N\\
k\equiv\kappa(\bmod\,L)
}
}\sum_{\substack{0\leq h\leq k\\
h\equiv\varkappa(\bmod\,\kappa)
}
}\exp\left(\frac{24\pi n+\delta\pi\Omega}{6N^{2}}\right)\\
 & \quad\times\prod_{i\in\mathcal{I}_{0}^{+}(h,k)}2^{u_{i}\delta}\prod_{i\in\mathcal{I}_{0}^{-}(h,k)}\left(1-e^{2\pi\mathrm{i}/n_{i}}\right)^{u_{i}\delta}\exp\left(\frac{\delta\pi}{12}\Delta(\varkappa,\kappa)\right)\\
 & \quad\times\left(\frac{2}{k(N+1)}\prod_{i\in\mathcal{I}_{0}(h,k)}\left(\hat{q}_{i}^{(1)}\hat{q}_{i}^{(2)},\left(\hat{q}_{i}^{(1)}\right)^{-1}\hat{q}_{i}^{(2)};\hat{q}_{i}^{(2)}\right)^{-|u_{i}|\delta}\right.\\
 & \qquad\quad\left.\times\prod_{i\in\mathcal{I}_{1}(h,k)}\left(\hat{q}_{i}^{(1)},\left(\hat{q}_{i}^{(1)}\right)^{-1}\hat{q}_{i}^{(2)};\hat{q}_{i}^{(2)}\right)^{-|u_{i}|\delta}-\frac{2}{k(N+1)}+\frac{\pi\sqrt{2}}{kN}\right).
\end{align*}
\end{thm}
From Theorem \ref{thm:1}, we deduce the following results on Conjectures
\ref{conj:1}--\ref{conj:5}.
\begin{thm}
\label{thm:6} Conjecture \ref{conj:1} is true.
\end{thm}
\begin{thm}
\label{thm:7} Conjecture \ref{conj:2} is true.
\end{thm}
\begin{rem*}
In \cite{SZ}, Schlosser and Zhou believed that it is possible to
supply a proof of Conjecture \ref{conj:2} without using asymptotic
machinery. Although this conjecture can be confimed by using our method,
a proof without resorting to asymptotic machinery remains open. A
similar situation occurs in \cite[Conjecture 17]{SZ}.
\end{rem*}
\begin{thm}
\label{thm:8} The $q$-series coefficients of $Q_{8}(q)^{\delta}$
exhibit for $\delta=2$ the length $16$ sign pattern $+-++-+--+--+-++-.$
For 
\[
-0.99<\delta\leq\frac{7-\sqrt{73}}{2}\approx-0.77200187265877\cdots
\]
they exhibit the sign pattern $+++----+.$ For $\delta=-2$ they exhibit
the length $16$ sign pattern $+++++----+++----.$
\end{thm}
\begin{rem*}
Applying Theorem \ref{thm:1}, we are unable to confirm Conjecture
\ref{conj:3} for $\beta\leq\delta\leq4.$ However, we can obtain
that, for any $\varepsilon>0,$ when $-1<\delta\leq-1+\varepsilon,$
the sign pattern of the $q$-series coefficients of $Q_{8}(q)^{\delta}$
is $+++----\,+$ for sufficiently large $n.$
\end{rem*}
\begin{thm}
\label{thm:10} Conjecture \ref{conj:4} is true.
\end{thm}
\begin{thm}
\label{thm:12} The $q$-series coefficients of $Q_{12}(q)^{\delta}$
exhibit for $\delta=1$ the sign pattern $+-+\,0-+-+-0+-.$ For $\delta=-1$
they exhibit the sign pattern $+++++\,0-----0,$ and for $-0.999<\delta<-0.501,$
when $n\geq1277$ they exhibit the sign pattern $+++++------+,$ for
$-0.499\leq\delta<-0.001,$ when $n\geq1283$ they exhibit they exhibit
the sign pattern $++++------+\,+,$ for $\delta=-\frac{1}{2}$ they
exhibit the length $24$ sign pattern $+++++------+++++------++.$
\end{thm}
\begin{rem*}
Employing Theorem \ref{thm:1}, we can not confirm Conjecture \ref{conj:5}
when $2\leq\delta\leq3.$ However, we can establish that for any $\varepsilon>0,$
when $-1<\delta\leq-1+\varepsilon$ and $-0.5-\varepsilon<\delta\leq-0.5,$
the sign pattern of the $q$-series coefficients of $Q_{12}(q)^{\delta}$
is $+++++------\,+$ for sufficiently large $n.$ Similarly, for $-0.5<\delta\leq-0.5+\varepsilon$
and $-\varepsilon<\delta\leq0,$ the $q$-series coefficients of $Q_{12}(q)^{\delta}$
exhibit the sign pattern $++++------+\,+$ for sufficiently large
$n.$
\end{rem*}
In the next section, we first recall transformation formulas for Dedekind's
eta-function and Jacobi's theta-function and then deduce a modular
transformation for $G(q)^{\delta}.$ Section \ref{sec:3} is devoted
to our proof of Theorem \ref{thm:1}. In Section \ref{sec:4}, as
applications of Theorem \ref{thm:1}, we employ Theorem \ref{thm:1}
to show Theorems \ref{thm:6}--\ref{thm:12}. We first apply Theorem
\ref{thm:1} to give estimates for the coefficients of $Q_{5}(q)^{\delta},Q_{6}(q)^{\delta},Q_{8}(q)^{\delta},Q_{10}(q)^{\delta}$
and $Q_{12}(q)^{\delta}$ for $\delta$ within the specified ranges
of real numbers, thus obtaining information about the sign of the
coefficients when $n\geq n_{0}$ with certain $n_{0}\in\mathbb{N}.$
This reduces the last possible counter-examples to $n<n_{0}.$ Since
each coefficient of $Q_{5}(q)^{\delta},Q_{6}(q)^{\delta},Q_{8}(q)^{\delta},Q_{10}(q)^{\delta}$
and $Q_{12}(q)^{\delta}$ is a polynomial in $\delta,$ we find ranges
in which coefficients are located for $\delta$ within the specified
ranges and for $n<n_{0},$ from which we obtain information about
the sign of these coefficients, and thereby prove Theorems \ref{thm:6}--\ref{thm:12}.

\section{Preliminaries}

\subsection{Dedekind's eta-function and Jacobi's theta-function}

For Dedekind's eta-function $\eta(\tau),$ we have the following transformation
formula.
\begin{lem}
Let $\varGamma=SL_{2}(\mathbb{Z})$. Then for $\gamma=\left(\begin{matrix}a & b\\
c & d
\end{matrix}\right)\in\varGamma$, $c>0$, we have 
\[
\eta(\gamma\tau)=e^{-\pi\mathrm{i}s(d,c)}e^{\pi\mathrm{i}(a+d)/(12c)}\sqrt{-\mathrm{i}(c\tau+d)}\eta(\tau),
\]
where the square root is taken on the principal branch with $\sqrt{z}>0$
for $z>0$.
\end{lem}
Setting $q=e^{2\pi\mathrm{i}\tau}$ with $\mathrm{Im}(\tau)>0$, we
know that the generating function for integer partitions can be written
as 
\[
\sum_{n=0}^{\infty}p(n)q^{n}=\frac{e^{\pi\mathrm{i}\tau/12}}{\eta(\tau)}.
\]

With $q=e^{2\pi\mathrm{i}\tau},$ we let  
\[
\vartheta(\varsigma,\tau):=\sum_{n=-\infty}^{\infty}(-1)^{n-\frac{1}{2}}q^{(n+\frac{1}{2})^{2}}e^{(2n+1)\pi i\varsigma}.
\]
If $\zeta=e^{2\pi\mathrm{i}\varsigma}$, then, by Jacobi's triple
product identity, we have
\[
\vartheta(\varsigma,\tau)=-\mathrm{i}q^{\frac{1}{8}}\zeta^{-\frac{1}{2}}(\zeta,\zeta^{-1}q,q;q)_{\infty}.
\]
It is easy to know that 
\begin{equation}
\vartheta(\varsigma+\alpha\tau+\beta,\tau)=(-1)^{\alpha+\beta}e^{-\pi\mathrm{i}\alpha^{2}\tau}e^{-2\pi\mathrm{i}\alpha\varsigma}\vartheta(\varsigma,\tau).\label{eq:theta2}
\end{equation}
Let $j(\gamma(\tau))=\frac{1}{c\tau+d}.$ Then

\begin{equation}
\eta(\gamma(\tau))=e^{-\pi\mathrm{i}s(d,c)}e^{\pi\mathrm{i}(a+d)/(12c)}e^{-\pi\mathrm{i}/4}j(\gamma(\tau))^{-1/2}\eta(\tau),\label{eq:eta}
\end{equation}
and 
\begin{equation}
\vartheta(\varsigma j(\gamma(\tau)),\gamma(\tau))=e^{-3\pi\mathrm{i}s(d,c)}e^{3\pi\mathrm{i}(a+d)/(12c)}e^{-3\pi\mathrm{i}/4}e^{\pi\mathrm{i}cj(\gamma(\tau))\varsigma^{2}}j(\gamma(\tau))^{-1/2}\vartheta(\varsigma,\tau).\label{eq:theta1}
\end{equation}

\subsection{\label{subsec:2-2} A modular transformation for $G(q)^{\delta}$}

Let $\tau=(h+iz^{\prime})/k$ with $\gcd(h,k)=1.$ We now construct
a matrix in $SL_{2}(\mathbb{Z}).$ Let $d=\gcd(n,k).$ Then $\gcd(n/d,k/d)=1$
and there is an integer $h_{n}^{\prime}(h,k)$ such that $h_{n}^{\prime}(h,k)hn/d\equiv-1\;(\bmod\;k/d).$
Set 
\[
b_{n}(h,k):=\frac{h_{n}^{\prime}(h,k)hn/d+1}{k/d},
\]
and
\[
\gamma_{(n,h,k)}:=\left(\begin{array}{cc}
h_{n}^{\prime}(h,k) & -b_{n}(h,k)\\
\frac{k}{d} & -\frac{nh}{d}
\end{array}\right).
\]
Then we have 
\[
\gamma_{(n,h,k)}(n\tau)=\frac{h_{n}^{\prime}(h,k)(nh+inz^{\prime})/k-b_{n}(h,k)}{\frac{k}{d}(nh+inz^{\prime})/k-\frac{nh}{d}}=\frac{h_{n}^{\prime}(h,k)d}{k}+\mathrm{i}\frac{d^{2}}{nkz^{\prime}}.
\]

Define 
\[
\Lambda(\varsigma,\tau):=(\zeta,\zeta^{-1}q;q)_{\infty}.
\]
Then, by \eqref{eq:eta} and \eqref{eq:theta1},
\[
\Lambda(\varsigma,\tau)=\mathrm{i}q^{-\frac{1}{12}}\zeta^{\frac{1}{2}}\frac{\vartheta(\varsigma,\tau)}{\eta(\tau)},
\]
and 
\begin{equation}
\Lambda(\varsigma,\tau)=e^{-\frac{\pi\mathrm{i}}{6}(\tau-\gamma(\tau))}e^{\pi\mathrm{i}\varsigma(1-j(\gamma(\tau)))}e^{2\pi\mathrm{i}s(d,c)}e^{-2\pi\mathrm{i}(a+d)/(12c)}e^{\pi\mathrm{i}/2}e^{-\pi\mathrm{i}cj(\gamma(\tau))\varsigma^{2}}\Lambda(\varsigma j(\gamma(\tau)),\gamma(\tau)).\label{eq:kappa}
\end{equation}
By \eqref{eq:kappa}, and $-s(d,c)=s(-d,c)$, we have
\begin{align*}
\Lambda(m\tau,n\tau) & =\mathrm{i}\exp\left(-2\pi\mathrm{i}s\left(nh/d,k/d\right)\right)\exp\left(\pi\mathrm{i}\left(\frac{mh}{k}+\frac{md}{kn}-\frac{2m^{2}h}{kn}\right)\right)\\
 & \;\times\exp\left(\frac{\pi}{12k}\left(\left(\frac{12m^{2}}{n}-12m+2n\right)z^{\prime}+\left(\frac{12mhd}{n}-\frac{12m^{2}h^{2}}{n}-\frac{2d^{2}}{n}\right)\frac{1}{z^{\prime}}\right)\right)\\
 & \;\times\Lambda(m\tau j(\gamma_{(n,h,k)}(n\tau)),\gamma_{(n,h,k)}(n\tau)).
\end{align*}

Let $\lambda_{m,n}(h,k)$ and $\lambda_{m,n}^{*}(h,k)$ be defined
as in Subsection \ref{subsec:1-2}. Using \eqref{eq:theta2} we have
\begin{align*}
\Lambda(m\tau,n\tau) & =\mathrm{i}(-1)^{\lambda_{m,n}(h,k)}\exp\left(-2\pi\mathrm{i}s\left(nh/d,k/d\right)\right)\\
 & \;\times\exp\left(\pi\mathrm{i}\left(\frac{mh}{k}-\frac{md}{kn}+2\frac{\lambda_{m,n}^{*}(h,k)md}{kn}+\frac{(\lambda_{m,n}^{2}(h,k)-\lambda_{m,n}(h,k))h_{n}^{\prime}(h,k)d}{k}\right)\right)\\
 & \;\times\exp\left(\frac{\pi}{12k}\left(\frac{12m^{2}}{n}-12m+2n\right)z^{\prime}\right)\\
 & \;\times\exp\left(\frac{\pi}{12k}\left(\frac{12d^{2}}{n}(\lambda_{m,n}^{*}(h,k)-\lambda_{m,n}^{*2}(h,k))-\frac{2d^{2}}{n}\right)\frac{1}{z^{\prime}}\right)\\
 & \;\times\Lambda(m\tau j(\gamma_{(n,h,k)}(n\tau))+\lambda_{m,n}(h,k)\gamma_{(n,h,k)}(n\tau),\gamma_{(n,h,k)}(n\tau)).
\end{align*}
Then 
\begin{align*}
G(e^{2\pi\mathrm{i}\tau}) & =\prod_{i=1}^{I}\Lambda(m_{i}\tau,n_{i}\tau)^{u_{i}}\\
 & =\mathrm{i}^{\sum_{j=1}^{I}u_{j}}(-1)^{\sum_{j=1}^{I}u_{j}\lambda_{m_{j},n_{j}}(h,k)}\omega_{h.k}\varTheta_{h,k}\exp\left(\frac{\pi}{12k}\left(\Omega z^{\prime}+\Delta(h,k)\frac{1}{z^{\prime}}\right)\right)\\
 & \;\times\prod_{i=1}^{I}\Lambda(m_{i}\tau j(\gamma_{(n_{i},h,k)}(n_{i}\tau))+\lambda_{m_{j},n_{j}}(h,k)\gamma_{(n_{i},h,k)}(n_{i}\tau),\gamma_{(n_{i},h,k)}(n_{i}\tau)).
\end{align*}
Let $q_{i}^{(1)}$ and $q_{i}^{(2)}$ be defined as in Subsection
\ref{subsec:1-2}. Then
\[
\begin{aligned}G(e^{\frac{2\pi\mathrm{i}h}{k}-\frac{2\pi z}{k^{2}}}) & =\prod_{i=1}^{I}\Lambda(m_{i}\tau,n_{i}\tau)^{u_{i}}\\
 & =\mathrm{i}^{\sum_{j=1}^{I}u_{j}}(-1)^{\sum_{j=1}^{I}u_{j}\lambda_{m_{j},n_{j}}(h,k)}\omega_{h.k}\varTheta_{h,k}\\
 & \;\times\exp\left(\frac{\pi}{12k}\left(\Omega\frac{z}{k}+\Delta(h,k)\frac{k}{z}\right)\right)\prod_{i=1}^{I}H(q_{i}^{(1)},q_{i}^{(2)})^{-u_{i}},
\end{aligned}
\]
where $H(\zeta,q):=1/(\zeta,\zeta^{-1}q;q).$ From this we deduce
that 
\begin{equation}
\begin{aligned}G(e^{\frac{2\pi\mathrm{i}h}{k}-\frac{2\pi z}{k^{2}}})^{\delta} & =\mathrm{i}^{\delta\sum_{j=1}^{I}u_{j}}(-1)^{\delta\sum_{j=1}^{I}u_{j}\lambda_{m_{j},n_{j}}(h,k)}\omega_{h.k}^{\delta}\varTheta_{h,k}^{\delta}\\
 & \times\exp\left(\frac{\delta\pi}{12k}\left(\Omega\frac{z}{k}+\Delta(h,k)\frac{k}{z}\right)\right)\hat{G}(h,k,z)^{\delta}
\end{aligned}
\label{eq:tran}
\end{equation}
with 
\begin{align*}
\hat{G}(h,k,z) & :=\prod_{i=1}^{I}H(q_{i}^{(1)},q_{i}^{(2)})^{-u_{i}}.
\end{align*}

\section{\label{sec:3} Proof of Theorem \ref{thm:1}}

In this section, we give our proof of Theorem \ref{thm:1}.

\subsection{Rademacher expansion for the coefficients of $G(q)^{\delta}$}

Let $n,N\in\mathbb{N}.$ By Rademacher \cite{R}, we have 
\[
c_{\delta}(n)=\sum_{\substack{0\leq h<k\leq N\\
\gcd(h,k)=1
}
}\frac{\mathrm{i}}{k^{2}}e^{-2\pi\mathrm{i}nh/k}\int_{z_{h,k}^{\prime}}^{z_{h,k}^{\prime\prime}}G(e^{2\pi\mathrm{i}h/k-2\pi z/k^{2}})^{\delta}e^{2\pi nz/k^{2}}dz,
\]
where $z$ runs on the arc of the circle: 
\[
\left|z-\frac{1}{2}\right|=\frac{1}{2}
\]
with the ends $z_{h,k}^{\prime}$ and $z_{h,k}^{\prime\prime}$ being
given by
\[
z_{hk}^{\prime}=\frac{k^{2}}{k^{2}+k_{1}^{2}}+\mathrm{i}\frac{kk_{1}}{k^{2}+k_{1}^{2}},\quad z_{hk}^{\prime\prime}=\frac{k^{2}}{k^{2}+k_{2}^{2}}-\mathrm{i}\frac{k^{2}}{k^{2}+k_{2}^{2}}.
\]
Here $k_{1},k_{2}$ arise from the denominators of adjoint points
of $h/k$ in the Farey sequence of order $N.$

Let $L:=\text{lcm}(n_{1},\cdots n_{I}).$ If $1\leq k\leq N,$ $0\leq h<k$
and $\gcd(h,k)=1,$ then we can find $1\leq\kappa\leq L$ and $0\leq\varkappa<\kappa$
such that $k\equiv\kappa\;(\bmod\;L),$ $h\equiv\varkappa\;(\bmod\;\kappa).$
From this we have, for $1\leq i\leq I,$
\[
\gcd(n_{i},k)=\gcd(n_{i},\kappa),\qquad\lambda_{m_{i},n_{i}}^{*}(h,k)=\lambda_{m_{i},n_{i}}^{*}(\varkappa,\kappa).
\]

Let $\mathcal{L}_{\leq0}$ and $\mathcal{L}_{>0}$ be defined as in
Subsection \ref{subsec:1-2}. Using \eqref{eq:tran} we have 

\begin{align*}
c_{\delta}(n) & =\sum_{1\leq\kappa\leq L}\sum_{0\leq\varkappa<\kappa}\sum_{\substack{1\leq k\leq N\\
k\equiv\kappa\,\text{mod}\,L
}
}\sum_{\substack{0\leq h\leq k\\
h\equiv\varkappa\,\text{mod}\,\kappa
}
}\frac{\mathrm{i}}{k^{2}}e^{-2\pi\mathrm{i}nh/k}\mathrm{i}^{\delta\sum_{j=1}^{I}u_{j}}(-1)^{\delta\sum_{j=1}^{I}u_{j}\lambda_{m_{j},n_{j}}(h,k)}\omega_{h.k}^{\delta}\varTheta_{h,k}^{\delta}\\
 & \;\times\int_{z_{h,k}^{\prime}}^{z_{h,k}^{\prime\prime}}\exp\left(\frac{\delta\pi}{12k}\left(\Omega\frac{z}{k}+\Delta(\varkappa,\kappa)\frac{k}{z}\right)\right)\hat{G}(h,k,z)^{\delta}e^{2\pi nz/k^{2}}dz\\
 & =:\sum_{1\leq\kappa\leq L}\sum_{0\leq\varkappa<\kappa}S_{\varkappa,\kappa}^{\delta}.
\end{align*}
Then
\[
c_{\delta}(n)=\sum_{(\varkappa,\kappa)\in\mathcal{L}_{>0}}S_{\varkappa,\kappa}^{\delta}+\sum_{(\varkappa,\kappa)\in\mathcal{L}_{\leq0}}S_{\varkappa,\kappa}^{\delta}=:I+E.
\]
Let $s_{h,k}$ denote the chord from $z_{h,k}^{\prime}$ to $z_{h,k}^{\prime\prime}.$
Then the path of integration in the inner sum of $E$ can be replaced
by the chord $s_{h,k}.$

We now state some bounds from \cite{R} and \cite{R2}. On the chord
$s_{h,k},$ we have
\begin{equation}
\Re\left(\frac{1}{z}\right)\geq1,\qquad0<\Re\left(z\right)<\frac{2k^{2}}{N^{2}}.\label{eq:Re}
\end{equation}
The endpoints $z_{h,k}^{\prime}$ and $z_{h,k}^{\prime\prime}$ on
the chord $s_{h,k}$ satisfy 
\[
\left|z_{h,k}^{\prime}\right|\leq\frac{\sqrt{2}k}{N},\qquad\left|z_{h,k}^{\prime\prime}\right|\leq\frac{\sqrt{2}k}{N}.
\]
On the circle $\left|z-\frac{1}{2}\right|=\frac{1}{2}$, we have $\Re(\frac{1}{z})=1.$
The length of the chord $s_{h,k}$ is 
\[
\left|s_{h,k}\right|\leq\frac{2k}{N+1}.
\]

\subsection{The minor arcs}

Let $\mathcal{I}_{0}(h,k)$ and $\mathcal{I}_{1}(h,k)$ be defined
as in Subsection \ref{subsec:1-2}. Then $\mathcal{I}_{0}(h,k)$ and
$\mathcal{I}_{1}(h,k)$ are two disjoint subsets of $\{1,2,\cdots,I\}$
and 
\[
\{1,2,\cdots,I\}=\mathcal{I}_{0}(h,k)\cup\mathcal{I}_{1}(h,k).
\]

\begin{prop}
\label{prop:2} Let
\[
\mathcal{T}_{i}^{h,k}(\tau):=m_{i}\tau j(\gamma_{(n_{i},h,k)}(n_{i}\tau))+\lambda_{m_{i},n_{i}}(h,k)\gamma_{(n_{i},h,k)}(n_{i}\tau).
\]
For $h/k\in\mathcal{F}_{N}$ and $i\in\mathcal{I}_{0}(h,k)$, we have
\[
\Im(\mathcal{T}_{i}^{h,k}(\tau))=0,
\]
and $\mathcal{T}_{i}^{h,k}(\tau)$ is not an integer. Futhermore,
we have
\begin{equation}
\left|1-\exp\left(\frac{2\pi\mathrm{i}}{n_{i}}\right)\right|\leq\left|1-\exp(2\pi\mathrm{i}(\mathcal{T}_{i}^{h,k}(\tau)))\right|\leq2.\label{eq:1-e}
\end{equation}
\end{prop}
\begin{proof}
It follows easily that 
\[
\mathcal{T}_{i}^{h,k}(\tau)=\frac{m_{i}d_{i}+\lambda_{m_{i},n_{i}}(h,k)n_{i}h_{n_{i}}^{\prime}(h,k)d_{i}}{kn_{i}}-\mathrm{i}\frac{m_{i}d_{i}h-\lambda_{m_{i},n_{i}}(h,k)d_{i}^{2}}{n_{i}z}.
\]
Since $i\in\mathcal{I}_{0}(h,k),$ we have
\begin{equation}
m_{i}d_{i}h-\lambda_{m_{i},n_{i}}(h,k)d_{i}^{2}=d_{i}^{2}\left(\frac{m_{i}h}{d_{i}}-\lambda_{m_{i},n_{i}}(h,k)\right)=-\lambda_{m_{i},n_{i}}^{*}(h,k)d_{i}^{2}=0.\label{eq:3-1}
\end{equation}
So $\Im(\mathcal{T}_{i}^{h,k}(\tau))=0$ and
\begin{align*}
\mathcal{T}_{i}^{h,k}(\tau) & =\frac{m_{i}d_{i}+\lambda_{m_{i},n_{i}}(h,k)n_{i}h_{n_{i}}^{\prime}(h,k)d_{i}}{kn_{i}}\\
 & =\frac{m_{i}d_{i}+m_{i}hn_{i}h_{n_{i}}^{\prime}(h,k)}{kn_{i}}\\
 & =\frac{m_{i}b_{n_{i}}(h,k)}{n_{i}}=\frac{(m_{i}/d_{i})b_{n_{i}}(h,k)}{n_{i}/d_{i}}.
\end{align*}
It follows from \eqref{eq:3-1} that $d_{i}\;|\;m_{i}h.$ Since $\gcd(h,k)=1$
and $d_{i}|k,$ we have $\gcd(h,d_{i})=1,$ and so $d_{i}\;|\;m_{i}.$
Since $b_{n_{i}}(h,k)k/d_{i}=hn_{i}h_{n_{i}}^{\prime}(h,k)/d_{i}+1,$
we know that $\gcd(n_{i}/d_{i},b_{n_{i}}(h,k))=1.$ If $\frac{(m_{i}/d_{i})b_{n_{i}}(h,k)}{(n_{i}/d_{i})}$
is an integer, then $(n_{i}/d)\;|\;(m_{i}/d),$ so that $n_{i}\;|\;m_{i}.$
This is a contradiction since $1\leq m_{i}\leq n_{i}-1.$ So $\frac{(m_{i}/d)b_{n_{i}}(h,k)}{(n_{i}/d)}$
is not an integer. The inequalities \eqref{eq:1-e} then follow readily.
\end{proof}
For $(\varkappa,\kappa)\in\mathcal{L}_{\leq0},$ when $24\pi n+\delta\pi\Omega>0,$
we have 
\begin{align*}
\left|S_{\varkappa,\kappa}^{\delta}\right| & \leq\sum_{\substack{1\leq k\leq N\\
k\equiv\kappa(\text{mod}\,L)
}
}\sum_{\substack{0\leq h\leq k\\
h\equiv\varkappa(\text{mod}\,\kappa)
}
}\frac{1}{k^{2}}\int_{s_{h,k}}\exp\left(\frac{24\pi n+\delta\pi\Omega}{12k^{2}}\Re(z)\right)\\
 & \;\times\exp\left(\frac{\delta\pi}{12}\Delta(\varkappa,\kappa)\Re\left(\frac{1}{z}\right)\right)\left|\hat{G}(h,k,z)\right|^{\delta}dz\\
 & \leq\sum_{\substack{1\leq k\leq N\\
k\equiv\kappa(\text{mod}\,L)
}
}\sum_{\substack{0\leq h\leq k\\
h\equiv\varkappa(\text{mod}\,\kappa)
}
}\frac{1}{k^{2}}\int_{s_{h,k}}\exp\left(\frac{24\pi n+\delta\pi\Omega}{6N^{2}}\right)\exp\left(\frac{\delta\pi}{12}\Delta(\varkappa,\kappa)\right)\\
 & \;\times\prod_{i=1}^{I}\left|H(q_{i}^{(1)},q_{i}^{(2)})^{-u_{i}}\right|^{\delta}dz\\
 & \leq\sum_{\substack{1\leq k\leq N\\
k\equiv\kappa(\text{mod}\,L)
}
}\sum_{\substack{0\leq h\leq k\\
h\equiv\varkappa(\text{mod}\,\kappa)
}
}\frac{1}{k^{2}}\int_{s_{h,k}}\exp\left(\frac{24\pi n+\delta\pi\Omega}{6N^{2}}\right)\exp\left(\frac{\delta\pi}{12}\Delta(\varkappa,\kappa)\right)\\
 & \;\times\prod_{i\in\mathcal{I}_{0}(h,k)}\left|H(q_{i}^{(1)},q_{i}^{(2)})^{-u_{i}}\right|^{\delta}\prod_{i\in\mathcal{I}_{1}(h,k)}H(\left|q_{i}^{(1)}\right|,\left|q_{i}^{(2)}\right|)^{|u_{i}|\delta}dz\\
 & =\sum_{\substack{1\leq k\leq N\\
k\equiv\kappa(\text{mod}\,L)
}
}\sum_{\substack{0\leq h\leq k\\
h\equiv\varkappa(\text{mod}\,\kappa)
}
}\frac{1}{k^{2}}\int_{s_{h,k}}\exp\left(\frac{24\pi n+\delta\pi\Omega}{6N^{2}}\right)\exp\left(\frac{\delta\pi}{12}\Delta(\varkappa,\kappa)\right)\\
 & \;\times\prod_{i\in\mathcal{I}_{0}(h,k)}\left(1-q_{i}^{(1)}\right)^{u_{i}\delta}\prod_{i\in\mathcal{I}_{0}(h,k)}\left(\left|q_{i}^{(1)}\right|\left|q_{i}^{(2)}\right|,\left|q_{i}^{(1)}\right|^{-1}\left|q_{i}^{(2)}\right|;\left|q_{i}^{(2)}\right|\right)^{-|u_{i}|\delta}\\
 & \;\times\prod_{i\in\mathcal{I}_{1}(h,k)}\left(\left|q_{i}^{(1)}\right|,\left|q_{i}^{(1)}\right|^{-1}\left|q_{i}^{(2)}\right|;\left|q_{i}^{(2)}\right|\right)^{-|u_{i}|\delta}dz
\end{align*}
It follows readily that 
\[
\left|q_{i}^{(1)}\right|=\exp\left(\frac{2\pi\left(m_{i}d_{i}h-\lambda_{m_{i},n_{i}}(h,k)d_{i}^{2}\right)}{n_{i}}\Re\left(\frac{1}{z}\right)\right)\text{,}
\]
and 
\[
\left|q_{i}^{(2)}\right|=\exp\left(-\frac{2\pi d_{i}^{2}}{n_{i}}\Re\left(\frac{1}{z}\right)\right).
\]
For all $i\in\mathcal{I}_{1}(h,k),$ it is easy to see that  
\[
m_{i}d_{i}h-\lambda_{m_{i},n_{i}}(h,k)d_{i}^{2}<0.
\]
Then, by Proposition \ref{prop:2}, we have 
\[
\begin{aligned}\left|S_{\varkappa,\kappa}^{\delta}\right| & \leq\sum_{\substack{1\leq k\leq N\\
k\equiv\kappa\,\left(\text{mod}\,L\right)
}
}\sum_{\substack{0\leq h\leq k\\
h\equiv\varkappa\,\left(\text{mod}\,\kappa\right)
}
}\frac{1}{k^{2}}\int_{s_{h,k}}\exp\left(\frac{24\pi n+\delta\pi\Omega}{6N^{2}}\right)\exp\left(\frac{\delta\pi}{12}\Delta(\varkappa,\kappa)\right)\\
 & \;\times\prod_{i\in\mathcal{I}_{0}^{+}(h,k)}2^{u_{i}\delta}\prod_{i\in\mathcal{I}_{0}^{-}(h,k)}\left|1-e^{2\pi\mathrm{i}/n_{i}}\right|^{u_{i}\delta}\\
 & \;\times\prod_{i\in\mathcal{I}_{0}(h,k)}\left(\hat{q}_{i}^{(1)}\hat{q}_{i}^{(2)},\left(\hat{q}_{i}^{(1)}\right)^{-1}\hat{q}_{i}^{(2)};\hat{q}_{i}^{(2)}\right)^{-|u_{i}|\delta}\\
 & \;\times\prod_{i\in\mathcal{I}_{1}(h,k)}\left(\hat{q}_{i}^{(1)},\left(\hat{q}_{i}^{(1)}\right)^{-1}\hat{q}_{i}^{(2)};\hat{q}_{i}^{(2)}\right)^{-|u_{i}|\delta}dz
\end{aligned}
\]
\[
\begin{aligned} & \leq\sum_{\substack{1\leq k\leq N\\
k\equiv\kappa\,\left(\text{mod}\,L\right)
}
}\sum_{\substack{0\leq h\leq k\\
h\equiv\varkappa\,\left(\text{mod}\,\kappa\right)
}
}\frac{2}{k(N+1)}\exp\left(\frac{24\pi n+\delta\pi\Omega}{6N^{2}}\right)\exp\left(\frac{\delta\pi}{12}\Delta(\varkappa,\kappa)\right)\\
 & \;\times\prod_{i\in\mathcal{I}_{0}^{+}(h,k)}2^{u_{i}\delta}\prod_{i\in\mathcal{I}_{0}^{-}(h,k)}\left|1-e^{2\pi\mathrm{i}/n_{i}}\right|^{u_{i}\delta}\\
 & \;\times\prod_{i\in\mathcal{I}_{0}(h,k)}\left(\hat{q}_{i}^{(1)}\hat{q}_{i}^{(2)},\left(\hat{q}_{i}^{(1)}\right)^{-1}\hat{q}_{i}^{(2)};\hat{q}_{i}^{(2)}\right)^{-|u_{i}|\delta}\\
 & \;\times\prod_{i\in\mathcal{I}_{1}(h,k)}\left(\hat{q}_{i}^{(1)},\left(\hat{q}_{i}^{(1)}\right)^{-1}\hat{q}_{i}^{(2)};\hat{q}_{i}^{(2)}\right)^{-|u_{i}|\delta}
\end{aligned}
\]
where $\hat{q}_{i}^{(1)}$ and $\hat{q}_{i}^{(2)}$ are as defined
in Subsection \ref{subsec:1-2}. From this we deduce that 
\begin{align*}
|E|\leq\sum_{(\varkappa,\kappa)\in\mathcal{L}_{\leq0}}\left|S_{\varkappa,\kappa}^{\delta}\right| & <\sum_{(\varkappa,\kappa)\in\mathcal{L}_{\leq0}}\sum_{\substack{1\leq k\leq N\\
k\equiv\kappa\,\left(\text{mod}\,L\right)
}
}\sum_{\substack{0\leq h\leq k\\
h\equiv\varkappa\,\left(\text{mod}\,\kappa\right)
}
}\frac{2}{k(N+1)}\exp\left(\frac{24\pi n+\delta\pi\Omega}{6N^{2}}\right)\\
 & \;\times\exp\left(\frac{\delta\pi}{12}\Delta(\varkappa,\kappa)\right)\prod_{i\in\mathcal{I}_{0}^{+}(h,k)}2^{u_{i}\delta}\prod_{i\in\mathcal{I}_{0}^{-}(h,k)}\left|1-e^{2\pi\mathrm{i}/n_{i}}\right|^{u_{i}\delta}\\
 & \;\times\prod_{i\in\mathcal{I}_{0}(h,k)}\left(\hat{q}_{i}^{(1)}\hat{q}_{i}^{(2)},\left(\hat{q}_{i}^{(1)}\right)^{-1}\hat{q}_{i}^{(2)};\hat{q}_{i}^{(2)}\right)^{-|u_{i}|\delta}\\
 & \;\times\prod_{i\in\mathcal{I}_{1}(h,k)}\left(\hat{q}_{i}^{(1)},\left(\hat{q}_{i}^{(1)}\right)^{-1}\hat{q}_{i}^{(2)};\hat{q}_{i}^{(2)}\right)^{-|u_{i}|\delta}
\end{align*}
when $24\pi n+\delta\pi\Omega>0.$

\subsection{The major arcs}

Let $\Pi_{h,k}$ be defined as in Subsection \ref{subsec:1-2}. For
the major arcs $I,$ we have
\begin{align*}
I & =\sum_{(\varkappa,\kappa)\in\mathcal{L}_{>0}}\sum_{\substack{1\leq k\leq N\\
k\equiv\kappa(\text{mod}\,L)
}
}\sum_{\substack{0\leq h\leq k\\
h\equiv\varkappa(\text{mod}\,\kappa)
}
}\frac{\mathrm{i}}{k^{2}}e^{-2\pi\mathrm{i}nh/k}\mathrm{i}^{\delta\sum_{j=1}^{I}u_{j}}(-1)^{\delta\sum_{j=1}^{I}u_{j}\lambda_{m_{j},n_{j}}(h,k)}\omega_{h.k}^{\delta}\varTheta_{h,k}^{\delta}\\
 & \;\times\int_{z_{h,k}^{\prime}}^{z_{h,k}^{\prime\prime}}\exp\left(\frac{\delta\pi}{12k}\left(\Omega\frac{z}{k}+\Delta(\varkappa,\kappa)\frac{k}{z}\right)\right)\hat{G}(h,k,z)^{\delta}e^{2\pi nz/k^{2}}dz\\
 & =\sum_{(\varkappa,\kappa)\in\mathcal{L}_{>0}}\sum_{\substack{1\leq k\leq N\\
k\equiv\kappa(\text{mod}\,L)
}
}\sum_{\substack{0\leq h\leq k\\
h\equiv\varkappa(\text{mod}\,\kappa)
}
}\frac{\mathrm{i}}{k^{2}}e^{-2\pi\mathrm{i}nh/k}\mathrm{i}^{\delta\sum_{j=1}^{I}u_{j}}(-1)^{\delta\sum_{j=1}^{I}u_{j}\lambda_{m_{j},n_{j}}(h,k)}\omega_{h.k}^{\delta}\varTheta_{h,k}^{\delta}\Pi_{h,k}\\
 & \;\times\int_{z_{h,k}^{\prime}}^{z_{h,k}^{\prime\prime}}\exp\left(\frac{\delta\pi}{12k}\left(\Omega\frac{z}{k}+\Delta(\varkappa,\kappa)\frac{k}{z}\right)\right)e^{2\pi nz/k^{2}}dz\\
 & \;+\sum_{(\varkappa,\kappa)\in\mathcal{L}_{>0}}\sum_{\substack{1\leq k\leq N\\
k\equiv\kappa(\text{mod}\,L)
}
}\sum_{\substack{0\leq h\leq k\\
h\equiv\varkappa(\text{mod}\,\kappa)
}
}\frac{\mathrm{i}}{k^{2}}e^{-2\pi\mathrm{i}nh/k}\mathrm{i}^{\delta\sum_{j=1}^{I}u_{j}}(-1)^{\delta\sum_{j=1}^{I}u_{j}\lambda_{m_{j},n_{j}}(h,k)}\omega_{h.k}^{\delta}\varTheta_{h,k}^{\delta}\\
 & \;\times\int_{z_{h,k}^{\prime}}^{z_{h,k}^{\prime\prime}}\exp\left(\frac{\delta\pi}{12k}\left(\Omega\frac{z}{k}+\Delta(\varkappa,\kappa)\frac{k}{z}\right)\right)\left(\hat{G}(h,k,z)^{\delta}-\Pi_{h,k}\right)e^{2\pi nz/k^{2}}dz\\
 & =:I_{M}+I_{R}.
\end{align*}

In order to give a bound for $I_{R},$ we first need the following
two lemmas.
\begin{lem}
\label{lem:d} Let $u_{j},\cdots,u_{J}$ be nonzero integers, and
$\delta\in\mathbb{R}_{>0}.$ Then for all $s_{j},t_{j},q_{j}\in\mathbb{C}$
with $\left|s_{j}\right|,\left|t_{j}\right|\text{,\ensuremath{\left|q_{j}\right|}\ensuremath{\ensuremath{\leq}1},}$
$j=1,\cdots,J,$ we have
\[
\left|\left(\prod_{j=1}^{J}(s_{j},t_{j};q_{j})^{u_{j}}\right)^{\delta}-1\right|\leq\prod_{j=1}^{J}(|s_{j}|,|t_{j}|;|q_{j}|)^{-|u_{j}|\delta}-1.
\]
\end{lem}
\begin{proof}
We apply the formula 
\[
(1+x)^{\alpha}=\sum_{k=0}^{\infty}\frac{\alpha(\alpha-1)(\alpha-k+1)}{k!}x^{k}
\]
to easily get 
\[
(1-x)^{-\alpha}=\sum_{k=0}^{\infty}\frac{(\alpha)_{k}}{k!}x^{k},
\]
where $(\alpha)_{j}:=\alpha(\alpha+1)\cdots(\alpha+j-1)$ and $|x|<1.$
It follows that 
\begin{equation}
\begin{aligned}(s_{j},t_{j},q_{j})^{u_{j}\delta} & =\prod_{i\geq1}(1-s_{j}q_{j}^{i-1})^{u_{j}\delta}(1-t_{j}q_{j}^{i-1})^{u_{j}\delta}\\
 & =\prod_{i\geq1}\left(\sum_{k_{ij}=0}^{\infty}\frac{(-u_{j}\delta)_{k_{ij}}}{k_{ij}!}\left(s_{j}q_{j}^{i-1}\right)^{k_{ij}}\right)\left(\sum_{l_{ij}=0}^{\infty}\frac{(-u_{j}\delta)_{l_{ij}}}{l_{ij}!}\left(t_{j}q_{j}^{i-1}\right)^{l_{ij}}\right)
\end{aligned}
\label{eq:ll1}
\end{equation}
and so 
\begin{equation}
(|s_{j}|,|t_{j}|;|q_{j}|)^{-|u_{j}|\delta}=\prod_{i\geq1}\left(\sum_{k_{ij}=0}^{\infty}\frac{(|u_{j}|\delta)_{k_{ij}}}{k_{ij}!}\left(|s_{j}||q_{j}|^{i-1}\right)^{k_{ij}}\right)\left(\sum_{l_{ij}=0}^{\infty}\frac{(|u_{j}|\delta)_{l_{ij}}}{l_{ij}!}\left(|t_{j}||q_{j}|^{i-1}\right)^{l_{ij}}\right).\label{eq:ll2}
\end{equation}
From \eqref{eq:ll1} and \eqref{eq:ll2} we deduce that 
\begin{align*}
 & \left|\left(\prod_{j=1}^{J}(s_{j},t_{j};q_{j})^{u_{j}}\right)^{\delta}-1\right|\\
 & =\left|\prod_{j=1}^{J}\prod_{i\geq1}\left(\sum_{k_{ij}=0}^{\infty}\frac{(-u_{j}\delta)_{k_{ij}}}{k_{ij}!}\left(s_{j}q_{j}^{i-1}\right)^{k_{ij}}\right)\left(\sum_{l_{ij}}^{\infty}\frac{(-u_{j}\delta)_{l_{ij}}}{l_{ij}!}\left(t_{j}q_{j}^{i-1}\right)^{l_{ij}}\right)-1\right|\\
 & =\left|\overset{*}{\sum}\prod_{j=1}^{J}\prod_{i\geq1}\left(\frac{(-u_{j}\delta)_{k_{ij}}}{k_{ij}!}\left(s_{j}q_{j}^{i-1}\right)^{k_{ij}}\right)\left(\frac{(-u_{j}\delta)_{l_{ij}}}{l_{ij}!}\left(t_{j}q_{j}^{i-1}\right)^{l_{ij}}\right)\right|\\
 & \leq\overset{*}{\sum}\prod_{j=1}^{J}\prod_{i\geq1}\left(\frac{(|u_{j}|\delta)_{k_{ij}}}{k_{ij}!}\left(|s_{j}||q_{j}|^{i-1}\right)^{k_{ij}}\right)\left(\frac{(|u_{j}|\delta)_{l_{ij}}}{l_{ij}!}\left(|t_{j}||q_{j}|^{i-1}\right)^{l_{ij}}\right)\\
 & =\prod_{j=1}^{J}\prod_{i\geq1}\left(\sum_{k_{ij}=0}^{\infty}\frac{(|u_{j}|\delta)_{k_{ij}}}{k_{ij}!}\left(|s_{j}||q_{j}|^{i-1}\right)^{k_{ij}}\right)\left(\sum_{l_{ij}=0}^{\infty}\frac{(|u_{j}|\delta)_{l_{ij}}}{l_{ij}!}\left(|t_{j}||q_{j}|^{i-1}\right)^{l_{ij}}\right)-1\\
 & =\prod_{j=1}^{J}(|s_{j}|,|t_{j}|;|q_{j}|)^{-|u_{j}|\delta}-1,
\end{align*}
where we have used the inequality $\left|(\gamma)_{k}\right|\leq(\left|\gamma\right|)_{k}$
and $\overset{*}{\sum}$ denotes a summation where $k_{ij},l_{ij}\geq0$
for $i\geq1,1\leq j\leq n,$ and at least one of the elements in the
set $\{k_{ij}|i\geq1,1\leq j\leq n\}\cup\{l_{ij}|i\geq1,1\leq j\leq n\}$
is non-zero. This concludes the proof.
\end{proof}
\begin{lem}
\label{lem:16} Suppose that $\delta>0,\Delta(\varkappa,\kappa)>0$
and $24\pi n+\delta\pi\Omega>0.$ If the inequality 
\begin{equation}
\min_{1\leq i\leq I}\left(\varUpsilon\left(\lambda_{m_{i},n_{i}}^{*}(h,k)\right)\frac{d_{i}^{2}}{n_{i}}\right)\geq\frac{\delta\Delta(\varkappa,\kappa)}{24}\label{eq:lem5}
\end{equation}
holds, then we have 
\begin{align*}
 & \exp\left(\frac{\delta\pi}{12}\Delta(\varkappa,\kappa)\Re\left(\frac{1}{z}\right)\right)\left|\frac{1}{\Pi_{h,k}}\hat{G}(h,k,z)^{\delta}-1\right|\\
 & \leq\exp\left(\frac{\delta\pi}{12}\Delta(\varkappa,\kappa)\right)\times\left(\prod_{i\in\mathcal{I}_{0}(h,k)}\left(\hat{q}_{i}^{(1)}\hat{q}_{i}^{(2)},\left(\hat{q}_{i}^{(1)}\right)^{-1}\hat{q}_{i}^{(2)};\hat{q}_{i}^{(2)}\right)^{-|u_{i}|\delta}\right.\\
 & \qquad\qquad\left.\times\prod_{i\in\mathcal{I}_{1}(h,k)}\left(\hat{q}_{i}^{(1)},\left(\hat{q}_{i}^{(1)}\right)^{-1}\hat{q}_{i}^{(2)};\hat{q}_{i}^{(2)}\right)^{-|u_{i}|\delta}-1\right).
\end{align*}
\end{lem}
\begin{proof}
For convience, we set
\[
p_{i}=\begin{cases}
q_{i}^{(1)}q_{i}^{(2)}, & i\in\mathcal{I}_{0}(h,k),\\
q_{i}^{(1)}, & i\in\mathcal{I}_{1}(h,k).
\end{cases}
\]
Then 
\begin{align*}
 & \left|\frac{1}{\Pi_{h,k}}\hat{G}(h,k,z)^{\delta}-1\right|=\left|\prod_{i=1}^{I}\left(p_{i},\left(q_{i}^{(1)}\right)^{-1}q_{i}^{(2)};q_{i}^{(2)}\right)^{u_{i}\delta}-1\right|\\
 & \leq\overset{*}{\sum}\prod_{i=1}^{I}\prod_{j\geq1}\left(\frac{(|u_{i}|\delta)_{k_{ij}}}{k_{ij}!}\left(|p_{i}|\left|q_{i}^{(2)}\right|^{j-1}\right)^{k_{ij}}\right)\left(\frac{(|u_{i}|\delta)_{l_{ij}}}{l_{ij}!}\left(\left|\left(q_{i}^{(1)}\right)^{-1}q_{i}^{(2)}\right|\left|q_{i}^{(2)}\right|^{j-1}\right)^{l_{ij}}\right)\\
 & =\overset{*}{\sum}\prod_{i=1}^{I}\prod_{j\geq1}\left(\frac{(|u_{i}|\delta)_{k_{ij}}}{k_{ij}!}\right)\left(\frac{(|u_{i}|\delta)_{l_{ij}}}{l_{ij}!}\right)\\
 & \;\times\exp\left(-2\pi\Re\left(\frac{1}{z}\right)\sum_{i=1}^{I}\sum_{j\geq1}\frac{d_{i}^{2}}{n_{i}}\left(\Phi(\lambda_{m_{i},n_{i}}^{*}(h,k))k_{ij}+(1-\lambda_{m_{i},n_{i}}^{*}(h,k))l_{ij}+2(j-1)(k_{ij}+l_{ij})\right)\right)
\end{align*}

Let 
\[
\varPsi(\varkappa,\kappa):=\sum_{i=1}^{I}\sum_{j\geq1}\frac{d_{i}^{2}}{n_{i}}\left(\Phi(\lambda_{m_{i},n_{i}}^{*}(h,k))k_{ij}+(1-\lambda_{m_{i},n_{i}}^{*}(h,k))l_{ij}+2(j-1)(k_{ij}+l_{ij})\right).
\]
Then
\begin{align*}
 & \exp\left(\frac{\delta\pi}{12}\Delta(\varkappa,\kappa)\Re\left(\frac{1}{z}\right)\right)\left|\frac{1}{\Pi_{h,k}}\hat{G}(h,k,z)^{\delta}-1\right|\\
 & \leq\overset{*}{\sum}\prod_{i=1}^{I}\prod_{j\geq1}\left(\frac{(|u_{i}|\delta)_{k_{ij}}}{k_{ij}!}\right)^{2}\times\exp\left(-2\pi\Re\left(\frac{1}{z}\right)\left(-\frac{\delta\Delta(\varkappa,\kappa)}{24}+\varPsi(\varkappa,\kappa)\right)\right).
\end{align*}
Since at least one of the elements in the set $\{k_{ij}|i\geq1,1\leq j\leq n\}\cup\{l_{ij}|i\geq1,1\leq j\leq n\}$
is non-zero, we use \eqref{eq:lem5} to obtain that
\begin{align*}
 & -\frac{\delta\Delta(\varkappa,\kappa)}{24}+\varPsi(\varkappa,\kappa)\\
 & \geq-\frac{\delta\Delta(\varkappa,\kappa)}{24}+\min_{1\leq i\leq I}\left(\min\left\{ \Phi(\lambda_{m_{i},n_{i}}^{*}(h,k)),1-\lambda_{m_{i},n_{i}}^{*}(h,k))\right\} \frac{d_{i}^{2}}{n_{i}}\right)\\
 & =-\frac{\delta\Delta(\varkappa,\kappa)}{24}+\min_{1\leq i\leq I}\left(\varUpsilon\left(\lambda_{m_{i},n_{i}}^{*}(h,k)\right)\frac{d_{i}^{2}}{n_{i}}\right)\geq0.
\end{align*}
 It follows that 
\[
\exp\left(\frac{\delta\pi}{12}\Delta(\varkappa,\kappa)\Re\left(\frac{1}{z}\right)\right)\left|\frac{1}{\Pi_{h,k}}\hat{G}(h,k,z)^{\delta}-1\right|
\]
is maximized when $\Re\left(\frac{1}{z}\right)=1$ and so
\begin{align*}
 & \exp\left(\frac{\delta\pi}{12}\Delta(\varkappa,\kappa)\Re\left(\frac{1}{z}\right)\right)\left|\frac{1}{\Pi_{h,k}}\hat{G}(h,k,z)^{\delta}-1\right|\\
 & \leq\exp\left(\frac{\delta\pi}{12}\Delta(\varkappa,\kappa)\right)\times\left(\prod_{i\in\mathcal{I}_{0}(h,k)}\left(\hat{q}_{i}^{(1)}\hat{q}_{i}^{(2)},\left(\hat{q}_{i}^{(1)}\right)^{-1}\hat{q}_{i}^{(2)};\hat{q}_{i}^{(2)}\right)^{-|u_{i}|\delta}\right.\\
 & \qquad\qquad\qquad\qquad\qquad\left.\times\prod_{i\in\mathcal{I}_{1}(h,k)}\left(\hat{q}_{i}^{(1)},\left(\hat{q}_{i}^{(1)}\right)^{-1}\hat{q}_{i}^{(2)};\hat{q}_{i}^{(2)}\right)^{-|u_{i}|\delta}-1\right).
\end{align*}
This completes the proof.
\end{proof}
We now turn to bounding $I_{R}.$

For each $h/k\in\mathcal{L}_{>0},$ when $24\pi n+\delta\pi\Omega>0,$
by Lemma \ref{lem:d} we have 
\begin{align*}
|I_{R}| & \leq\sum_{(\varkappa,\kappa)\in\mathcal{L}_{>0}}\sum_{\substack{1\leq k\leq N\\
k\equiv\kappa(\text{mod}\,L)
}
}\sum_{\substack{0\leq h\leq k\\
h\equiv\varkappa(\text{mod}\,\kappa)
}
}\frac{1}{k^{2}}\int_{z_{h,k}^{\prime}}^{z_{h,k}^{\prime\prime}}\exp\left(\frac{24\pi n+\delta\pi\Omega}{12k^{2}}\Re(z)\right)\\
 & \quad\times\exp\left(\frac{\delta\pi}{12}\Delta(\varkappa,\kappa)\Re\left(\frac{1}{z}\right)\right)\Pi_{h,k}\left|\frac{1}{\Pi_{h,k}}\hat{G}(h,k,z)^{\delta}-1\right|dz\\
 & <\sum_{(\varkappa,\kappa)\in\mathcal{L}_{>0}}\sum_{\substack{1\leq k\leq N\\
k\equiv\kappa(\text{mod}\,L)
}
}\sum_{\substack{0\leq h\leq k\\
h\equiv\varkappa(\text{mod}\,\kappa)
}
}\frac{1}{k^{2}}\int_{z_{h,k}^{\prime}}^{z_{h,k}^{\prime\prime}}\exp\left(\frac{24\pi n+\delta\pi\Omega}{6N^{2}}\right)\\
 & \quad\times\exp\left(\frac{\delta\pi}{12}\Delta(\varkappa,\kappa)\Re\left(\frac{1}{z}\right)\right)\Pi_{h,k}\\
 & \quad\times\left(\prod_{i\in\mathcal{I}_{0}(h,k)}\left(\left|q_{i}^{(1)}\right|\left|q_{i}^{(2)}\right|,\left|q_{i}^{(1)}\right|^{-1}\left|q_{i}^{(2)}\right|;\left|q_{i}^{(2)}\right|\right)^{-|u_{i}|\delta}\right.\\
 & \qquad\quad\left.\times\prod_{i\in\mathcal{I}_{1}(h,k)}\left(\left|q_{i}^{(1)}\right|,\left|q_{i}^{(1)}\right|^{-1}\left|q_{i}^{(2)}\right|;\left|q_{i}^{(2)}\right|\right)^{-|u_{i}|\delta}-1\right)\\
 & \leq\sum_{(\varkappa,\kappa)\in\mathcal{L}_{>0}}\sum_{\substack{1\leq k\leq N\\
k\equiv\kappa(\text{mod}\,L)
}
}\sum_{\substack{0\leq h\leq\kappa\\
h\equiv\varkappa(\text{mod}\,\kappa)
}
}\frac{2}{k(N+1)}\exp\left(\frac{24\pi n+\delta\pi\Omega}{6N^{2}}\right)\\
 & \quad\times\prod_{i\in\mathcal{I}_{0}^{+}(h,k)}2^{u_{i}\delta}\prod_{i\in\mathcal{I}_{0}^{-}(h,k)}\left|1-e^{2\pi\mathrm{i}/n_{i}}\right|^{u_{i}\delta}\exp\left(\frac{\delta\pi}{12}\Delta(\varkappa,\kappa)\right)\\
 & \quad\times\left(\prod_{i\in\mathcal{I}_{0}(h,k)}\left(\hat{q}_{i}^{(1)}\hat{q}_{i}^{(2)},\left(\hat{q}_{i}^{(1)}\right)^{-1}\hat{q}_{i}^{(2)};\hat{q}_{i}^{(2)}\right)^{-|u_{i}|\delta}\right.\\
 & \qquad\quad\left.\times\prod_{i\in\mathcal{I}_{1}(h,k)}\left(\hat{q}_{i}^{(1)},\left(\hat{q}_{i}^{(1)}\right)^{-1}\hat{q}_{i}^{(2)};\hat{q}_{i}^{(2)}\right)^{-|u_{i}|\delta}-1\right),
\end{align*}
where, in the last inequality, we have used Lemma \ref{lem:16}.

Next we estimate $I_{M}.$ Let 
\[
I_{h.k}:=\int_{z_{h,k}^{\prime}}^{z_{h,k}^{\prime\prime}}\exp\left(\frac{\delta\pi}{12k}\left(\Omega\frac{z}{k}+\Delta(\varkappa,\kappa)\frac{k}{z}\right)\right)e^{2\pi nz/k^{2}}dz
\]
and let $K^{-}$denote the integration path along the whole circle
\[
\left|z-\frac{1}{2}\right|=\frac{1}{2}
\]
in negative direction. Then we have 
\[
I_{h.k}=\left(\int_{K^{-}}-\int_{0}^{z_{h,k}^{\prime}}-\int_{z_{h,k}^{\prime\prime}}^{0}\right)\exp\left(\frac{\delta\pi}{12k}\left(\Omega\frac{z}{k}+\Delta(\varkappa,\kappa)\frac{k}{z}\right)\right)e^{2\pi nz/k^{2}}dz.
\]
Using \eqref{eq:Re}, we know that the lengths of the arcs from $0$
to $z_{h,k}^{\prime}$ and from $z_{h,k}^{\prime\prime}$ to $0$
are respectively less than
\[
\frac{\pi}{2}\left|z_{h,k}^{\prime}\right|\leq\frac{\pi}{2}\frac{\sqrt{2}k}{N},\quad\text{and\ensuremath{\quad}\ensuremath{\frac{\pi}{2}\left|z_{h,k}^{\prime\prime}\right|\leq\frac{\pi}{2}\frac{\sqrt{2}k}{N}},}
\]
so we have 
\begin{align*}
\left|\left(\int_{0}^{z_{h,k}^{\prime}}+\int_{z_{h,k}^{\prime\prime}}^{0}\right)\exp\left(\frac{\delta\pi}{12k}\left(\Omega\frac{z}{k}+\Delta(\varkappa,\kappa)\frac{k}{z}\right)\right)e^{2\pi nz/k^{2}}dz\right|\\
\quad\leq\frac{\pi\sqrt{2}k}{N}\exp\left(\frac{24\pi n+\delta\pi\Omega}{6N^{2}}+\frac{\delta\pi}{12}\Delta(\varkappa,\kappa)\right)
\end{align*}
when $24\pi n+\delta\pi\Omega>0.$ 

In order to simplify $\int_{K^{-}}\exp\left(\frac{\delta\pi}{12k}\left(\Omega\frac{z}{k}+\Delta(\varkappa,\kappa)\frac{k}{z}\right)\right)e^{2\pi nz/k^{2}}dz$,
we first recall the definition of the modified Bessel function $I_{-1}(z)$
of the first kind given by 
\[
I_{-1}(z):=\sum_{n\geq0}\frac{1}{n!(n+1)!}\left(\frac{z}{2}\right)^{2n+1}.
\]
See, for example, \cite[p.222, Eq.(4.12.2)]{AAR}. From \cite[p.236, Exercise13]{AAR},
we know that 
\[
I_{-1}(z)=\frac{(z/2)}{2\pi i}\int_{1-\mathrm{i}\infty}^{1+\mathrm{i}\infty}\frac{1}{t^{2}}e^{t+z^{2}/4t}dt.
\]
Let $w=\frac{1}{z}$ in the integral $\int_{K^{-}}.$ Then 
\begin{align*}
 & \int_{K^{-}}\exp\left(\frac{\delta\pi}{12k}\left(\Omega\frac{z}{k}+\Delta(\varkappa,\kappa)\frac{k}{z}\right)\right)e^{2\pi nz/k^{2}}dz\\
 & \;=-\int_{1-\mathrm{i}\infty}^{1+\mathrm{i}\infty}\frac{1}{w^{2}}\exp\left(\frac{24\pi n+\delta\pi\Omega}{12k^{2}}\frac{1}{w}+\frac{\delta\pi}{12}\Delta(\varkappa,\kappa)w\right)dw\quad\left(t=\frac{\delta\pi}{12}\Delta(\varkappa,\kappa)w\right)\\
 & \;=-\frac{\delta\pi}{12}\Delta(\varkappa,\kappa)\int_{1-\mathrm{i}\infty}^{1+\mathrm{i}\infty}\frac{1}{t^{2}}\exp\left(t+\frac{24\pi n+\delta\pi\Omega}{12k^{2}}\frac{\delta\pi}{12}\Delta(\varkappa,\kappa)\frac{1}{t}\right)dt\\
 & \;=-\frac{2\delta k\pi\mathrm{i}\sqrt{\Delta(\varkappa,\kappa)}}{\sqrt{(24n\delta+\delta^{2}\Omega)}}I_{-1}\left(\frac{\pi}{6k}\sqrt{\text{\ensuremath{\Delta(\varkappa,\kappa)}}(24n\delta+\delta^{2}\Omega)}\right).
\end{align*}
From this we derive that
\begin{align*}
I_{M} & =\sum_{(\varkappa,\kappa)\in\mathcal{L}_{>0}}\sum_{\substack{1\leq k\leq N\\
k\equiv\kappa\,\left(\text{mod}\,L\right)
}
}\sum_{\substack{0\leq h\leq k\\
h\equiv\varkappa\,\left(\text{mod}\,\kappa\right)
}
}\frac{\mathrm{i}}{k^{2}}e^{-2\pi inh/k}\mathrm{i}^{\delta\sum_{j=1}^{I}u_{j}}(-1)^{\delta\sum_{j=1}^{I}u_{j}\lambda_{m_{j},n_{j}}(h,k)}\\
 & \;\times\omega_{h.k}^{\delta}\varTheta_{h,k}^{\delta}\Pi_{h,k}\int_{z_{h,k}^{\prime}}^{z_{h,k}^{\prime\prime}}\exp\left(\frac{\delta\pi}{12k}\left(\Omega\frac{z}{k}+\Delta(\varkappa,\kappa)\frac{k}{z}\right)\right)e^{2\pi nz/k^{2}}dz\\
 & =I_{MM}+I_{MR},
\end{align*}
where 
\begin{align*}
I_{MM} & :=\sum_{(\varkappa,\kappa)\in\mathcal{L}_{>0}}\sum_{\substack{1\leq k\leq N\\
k\equiv\kappa\,\left(\text{mod}\,L\right)
}
}\sum_{\substack{0\leq h\leq k\\
h\equiv\varkappa\,\left(\text{mod}\,\kappa\right)
}
}\frac{2\delta\pi}{k}e^{-2\pi\mathrm{i}nh/k}\mathrm{i}^{\delta\sum_{j=1}^{I}u_{j}}(-1)^{\delta\sum_{j=1}^{I}u_{j}\lambda_{m_{j},n_{j}}(h,k)}\\
 & \;\times\omega_{h.k}^{\delta}\varTheta_{h,k}^{\delta}\Pi_{h,k}\left(\frac{\Delta(\varkappa,\kappa)}{24n\delta+\delta^{2}\Omega}\right)^{1/2}I_{-1}\left(\frac{\pi}{6k}\sqrt{\text{\ensuremath{\Delta(\varkappa,\kappa)}}(24n\delta+\delta^{2}\Omega)}\right),
\end{align*}
and 
\begin{align*}
\left|I_{MR}\right| & \leq\sum_{(\varkappa,\kappa)\in\mathcal{L}_{>0}}\sum_{\substack{1\leq k\leq N\\
k\equiv\kappa\,\left(\text{mod}\,L\right)
}
}\sum_{\substack{0\leq h\leq k\\
h\equiv\varkappa\,\left(\text{mod}\,\kappa\right)
}
}\frac{\pi\sqrt{2}}{kN}\prod_{i\in\mathcal{I}_{0}^{+}(h,k)}2^{u_{i}\delta}\\
 & \;\times\prod_{i\in\mathcal{I}_{0}^{-}(h,k)}\left|1-e^{2\pi\mathrm{i}/n_{i}}\right|^{u_{i}\delta}\exp\left(\frac{24\pi n+\delta\pi\Omega}{6N^{2}}+\frac{\delta\pi}{12}\Delta(\varkappa,\kappa)\right).
\end{align*}

Combining the above results for $I$ and $E$ we obtain Theorem \ref{thm:1}
readily.

\section{\label{sec:4} Applications}

In this section, we apply Theorem \ref{thm:1} to show Theorems \ref{thm:6}--\ref{thm:12}.

\subsection{Proof of Theorem \ref{thm:6}}

Let $Q_{5}(q)$ be defined as in Conjecture \ref{conj:1}. Then ${\bf m}=\{1,2\},$
${\bf n}=\{5,5\},$ and ${\bf u}=\{1,-1\}.$ Hence $L=5,$ and $\Omega=24/5.$
It is easy to find that 
\[
\mathcal{L}_{>0}=\{(2,5),(3,5)\}.
\]
From this we deduce that
\[
\max\{\sqrt{\Delta(\varkappa,\kappa)}/k|(\varkappa,\kappa)\in\mathcal{L}_{>0},k\equiv\kappa\;\text{(mod\;}L)\}=\frac{2\sqrt{6}}{5\sqrt{5}}
\]
and 
\[
\sqrt{\Delta(\varkappa,\kappa)}/k
\]
attains its maximum when $(\varkappa,\kappa,k)=(2,5,5),\;(3,5,5).$

When $(\varkappa,\kappa)=(2,5),(3,5),$ we easily find that 
\[
\frac{24}{\Delta(\varkappa,\kappa)}\varUpsilon\left(\lambda_{m_{i},n_{i}}^{*}(h,k)\right)\frac{d_{i}^{2}}{n_{i}}\geq\delta
\]
for $i=1,2$ and $0<\delta\leq5.$ We let 
\begin{align*}
Q_{5}^{\delta}(q) & =:\sum_{n=1}^{\infty}c_{\delta}^{(5)}(n)q^{n},\\
N & =\left\lceil \sqrt{4\pi\left(n+\frac{\delta}{5}\right)}\right\rceil ,
\end{align*}
and
\[
\hat{c}_{\delta}^{(5)}(n):=\frac{4\pi\sqrt{\delta}}{5\sqrt{5}}\cos\left(\frac{4\pi}{5}\left(n+\frac{3\delta}{20}\right)\right)\left(n+\frac{\delta}{5}\right)^{-1/2}I_{-1}\left(\frac{4\pi\sqrt{\delta}}{5\sqrt{5}}\sqrt{n+\frac{\delta}{5}}\right).
\]
Then we have 
\begin{align*}
c_{\delta}^{(5)}(n)-\hat{c}_{\delta}^{(5)}(n) & =\sum_{\substack{5<k\leq N\\
k\equiv5\,\left(\text{mod}\,5\right)
}
}\sum_{\substack{0\leq h\leq k\\
h\equiv2\,\left(\text{mod}\,5\right)
}
}\frac{2\pi\delta^{1/2}}{k}e^{-2\pi\mathrm{i}nh/k}\mathrm{i}^{\delta\sum_{j=1}^{I}u_{j}}(-1)^{\delta\sum_{j=1}^{I}u_{j}\lambda_{m_{j},n_{j}}(h,k)}\\
 & \;\times\omega_{h.k}^{\delta}\varTheta_{h,k}^{\delta}\Pi_{h,k}\left(n+\frac{\delta}{5}\right)^{-1/2}I_{-1}\left(\frac{4\pi\sqrt{\delta}}{\sqrt{5}k}\sqrt{n+\frac{\delta}{5}}\right)\\
 & +\sum_{\substack{5<k\leq N\\
k\equiv5\,\left(\text{mod}\,5\right)
}
}\sum_{\substack{0\leq h\leq k\\
h\equiv3\,\left(\text{mod}\,5\right)
}
}\frac{2\pi\delta^{1/2}}{k}e^{-2\pi\mathrm{i}nh/k}\mathrm{i}^{\delta\sum_{j=1}^{I}u_{j}}(-1)^{\delta\sum_{j=1}^{I}u_{j}\lambda_{m_{j},n_{j}}(h,k)}\\
 & \;\times\omega_{h.k}^{\delta}\varTheta_{h,k}^{\delta}\Pi_{h,k}\left(n+\frac{\delta}{5}\right)^{-1/2}I_{-1}\left(\frac{4\pi\sqrt{\delta}}{\sqrt{5}k}\sqrt{n+\frac{\delta}{5}}\right)+E_{\delta,N}(n).
\end{align*}
For the $E_{\delta,N}(n),$ we have  
\begin{align*}
\left|E_{\delta,N}(n)\right| & \leq\sum_{(\varkappa,\kappa)\in\mathcal{L}_{\leq0}}2^{1+\delta}\exp\left(1+\frac{\delta\pi}{12}\Delta(\varkappa,\kappa)\right)\left|1-e^{2\pi\mathrm{i}/5}\right|^{-\delta}\\
 & \;\times\prod_{i\in\mathcal{I}_{0}(\varkappa,\kappa)}\left(\hat{q}_{i}^{(1)}\hat{q}_{i}^{(2)},\left(\hat{q}_{i}^{(1)}\right)^{-1}\hat{q}_{i}^{(2)};\hat{q}_{i}^{(2)}\right)^{-|u_{i}|\delta}\\
 & \;\times\prod_{i\in\mathcal{I}_{1}(\varkappa,\kappa)}\left(\hat{q}_{i}^{(1)},\left(\hat{q}_{i}^{(1)}\right)^{-1}\hat{q}_{i}^{(2)};\hat{q}_{i}^{(2)}\right)^{-|u_{i}|\delta}\\
 & +\sum_{(\varkappa,\kappa)\in\mathcal{L}_{>0}}2^{\delta}e\left|1-e^{2\pi\mathrm{i}/5}\right|^{-\delta}\exp\left(\frac{\delta\pi}{12}\Delta(\varkappa,\kappa)\right)\\
 & \quad\times\left(2\prod_{i\in\mathcal{I}_{0}(\varkappa,\kappa)}\left(\hat{q}_{i}^{(1)}\hat{q}_{i}^{(2)},\left(\hat{q}_{i}^{(1)}\right)^{-1}\hat{q}_{i}^{(2)};\hat{q}_{i}^{(2)}\right)^{-|u_{i}|\delta}\right.\\
 & \qquad\quad\left.\times\prod_{i\in\mathcal{I}_{1}(\varkappa,\kappa)}\left(\hat{q}_{i}^{(1)},\left(\hat{q}_{i}^{(1)}\right)^{-1}\hat{q}_{i}^{(2)};\hat{q}_{i}^{(2)}\right)^{-|u_{i}|\delta}-2+\sqrt{2}\pi\right).
\end{align*}
Using the software \emph{Mathematica}, we get 
\begin{align*}
\left|E_{\delta,N}(n)\right| & \leq54.366\times10.372^{\delta}+5.437\left(1.702^{\delta}+0.486^{\delta}+0.485^{\delta}\right)\\
 & \;+10.874\times5.978^{\delta}\left(2.443+1.002^{\delta}\right).
\end{align*}
From the above we derive that
\begin{align*}
 & \left|c_{\delta}^{(5)}(n)-\hat{c}_{\delta}^{(5)}(n)\right|\\
 & \leq\left|E_{\delta,N}(n)\right|+\sum_{\substack{5<k\leq N\\
k\equiv5\,\left(\text{mod}\,5\right)
}
}\sum_{\substack{0\leq h\leq k\\
h\equiv2\,\left(\text{mod}\,5\right)
}
}\frac{2\pi\delta^{1/2}}{k}\left|\Pi_{h,k}\right|\left(n+\frac{\delta}{5}\right)^{-1/2}I_{-1}\left(\frac{4\pi\sqrt{\delta}}{\sqrt{5}k}\sqrt{n+\frac{\delta}{5}}\right)\\
 & +\sum_{\substack{5<k\leq N\\
k\equiv5\,\left(\text{mod}\,5\right)
}
}\sum_{\substack{0\leq h\leq k\\
h\equiv3\,\left(\text{mod}\,5\right)
}
}\frac{2\pi\delta^{1/2}}{k}\left|\Pi_{h,k}\right|\left(n+\frac{\delta}{5}\right)^{-1/2}I_{-1}\left(\frac{4\pi\sqrt{\delta}}{\sqrt{5}k}\sqrt{n+\frac{\delta}{5}}\right)\\
 & \leq\sum_{\substack{5<k\leq N\\
k\equiv5\,\left(\text{mod}\,5\right)
}
}8\delta^{1/2}\pi\left|\frac{1}{1-e^{\frac{2\pi\mathrm{i}}{5}}}\right|\left(n+\frac{\delta}{5}\right)^{-1/2}I_{-1}\left(\frac{4\pi\sqrt{\delta}}{\sqrt{5}k}\sqrt{n+\frac{\delta}{5}}\right)\\
 & =8\delta^{1/2}\pi\left|\frac{1}{1-e^{\frac{2\pi\mathrm{i}}{5}}}\right|\left(n+\frac{\delta}{5}\right)^{-1/2}\sum_{1<k^{\prime}\leq N/5}I_{-1}\left(\frac{4\pi\sqrt{\delta}}{5\sqrt{5}k^{\prime}}\sqrt{n+\frac{\delta}{5}}\right).
\end{align*}

In order to bound 
\[
\left|\frac{c_{\delta}^{(5)}(n)-\hat{c}_{\delta}^{(5)}(n)}{\hat{c}_{\delta}^{(5)}(n)}\right|,
\]
 we need the following two lemmas.
\begin{lem}
\label{lem:z} \cite[Lemma 8]{SZ} For any real $x>0$ and integer
$y>2$ we have 
\[
\sum_{2\leq k\leq y}I_{-1}\left(\frac{2x}{k}\right)\leq x\log y+2I_{-1}(x)-\left(2-\gamma-\frac{1}{2y}\right)x.
\]
Here $\gamma=0.577216\cdots$ is the Euler-Mascheroni constant.
\end{lem}
\begin{lem}
\label{lem:z2} Suppose that $s>0,$ then for $t\geq3,$ the function
\[
M\left(s,t\right):=\frac{t\log\left(t\right)+2I_{-1}(t)+st}{I_{-1}(2t)}
\]
is decreasing in $t.$
\end{lem}
This lemma gives an equivalent form of \cite[Lemma 9]{SZ}.

We now turn to bounding 
\[
\left|\frac{c_{\delta}^{(5)}(n)-\hat{c}_{\delta}^{(5)}(n)}{\hat{c}_{\delta}^{(5)}(n)}\right|.
\]

Let
\[
L_{\delta,n}=\frac{2\pi\sqrt{\delta}}{5\sqrt{5}}\sqrt{n+\frac{\delta}{5}}.
\]
Apply Lemma \ref{lem:z}. We get 
\begin{align*}
\left|c_{\delta}^{(5)}(n)-\hat{c}_{\delta}^{(5)}(n)\right| & \leq\left|E_{\delta,N}(n)\right|+8\delta^{1/2}\pi\left|\frac{1}{1-e^{\frac{2\pi i}{5}}}\right|\left(n+\frac{\delta}{5}\right)^{-1/2}\\
 & \;\times\left(L_{\delta,n}\log\left(\frac{N}{5}\right)+2I_{-1}\left(L_{\delta,n}\right)-\left(2-\gamma-\frac{5}{2N}\right)L_{\delta,n}\right).
\end{align*}
For $1\leq\delta\leq\frac{\sqrt{97}-5}{2}\approx2.424428900898\cdots,$
$n\geq176,$ we have
\[
\sqrt{4\pi\left(n+\frac{\delta}{5}\right)}\geq\sqrt{4\pi\left(176+\frac{1}{5}\right)}\geq47,
\]
so that

\[
N=\left\lceil \sqrt{4\pi\left(n+\frac{\delta}{5}\right)}\right\rceil =\left\lceil \frac{5\sqrt{5}L_{\delta,n}}{\sqrt{\pi\delta}}\right\rceil \leq\frac{5\sqrt{5}L_{\delta,n}}{\sqrt{\pi\delta}}+1\leq\frac{48}{47}\times\frac{5\sqrt{5}L_{\delta,n}}{\sqrt{\pi\delta}}.
\]
Then 
\begin{align*}
\left|c_{\delta}^{(5)}(n)-\hat{c}_{\delta}^{(5)}(n)\right| & \leq\left|E_{\delta,N}(n)\right|+8\delta^{1/2}\pi\left|\frac{1}{1-e^{\frac{2\pi\mathrm{i}}{5}}}\right|\left(n+\frac{\delta}{5}\right)^{-1/2}\\
 & \;\times\left(L_{\delta,n}\log\left(\frac{\sqrt{5}L_{\delta,n}}{\sqrt{\pi\delta}}\frac{48}{47}\right)+2I_{-1}\left(L_{\delta,n}\right)-\left(2-\gamma-\frac{5}{94}\right)L_{\delta,n}\right)
\end{align*}
and so
\begin{align*}
 & \left|\frac{c_{\delta}^{(5)}(n)}{\hat{c}_{\delta}^{(5)}(n)}-1\right|\left|\cos\left(\frac{4\pi}{5}\left(n+\frac{3\delta}{20}\right)\right)\right|\\
 & \leq10\sqrt{5}\left|\frac{1}{1-e^{\frac{2\pi\mathrm{i}}{5}}}\right|\frac{L_{\delta,n}\log\left(L_{\delta,n}\right)+2I_{-1}(L_{\delta,n})+\nu(\delta)L_{\delta,n}}{I_{-1}(2L_{\delta,n})}
\end{align*}
with 
\[
\begin{aligned}\nu(\delta) & =\frac{1}{2}\log\left(\frac{1}{\delta}\right)-1.116+\frac{0.914\times5.978^{\delta}\left(2.443+1.002^{\delta}\right)}{\delta}\\
 & \;+\frac{0.084\left(54.366\times10.372^{\delta}+5.437\left(1.702^{\delta}+0.486^{\delta}+0.485^{\delta}\right)\right)}{\delta}.
\end{aligned}
\]

For all $n\geq176,$ $\delta\geq1,$ we have $L_{\delta,n}\geq3,$
$\nu(\delta)>0.$ Apply Lemma \ref{lem:z2}. We obtain that
\[
10\sqrt{5}\left|\frac{1}{1-e^{\frac{2\pi\mathrm{i}}{5}}}\right|\frac{L_{\delta,n}\log\left(L_{\delta,n}\right)+2I_{-1}(L_{\delta,n})+\nu(\delta)L_{\delta,n}}{I_{-1}(2L_{\delta,n})}
\]
is decreasing in $n$ for $n\geq176,$ and $\delta\geq1.$ Employing
the software \emph{Mathematica} we can verify that 
\[
10\sqrt{5}\left|\frac{1}{1-e^{\frac{2\pi\mathrm{i}}{5}}}\right|M\left(\nu(\delta),L_{\delta,176}\right)<\min_{n=0,1,2,3,4}\left|\cos\left(\frac{4\pi}{5}\left(n+\frac{3\delta}{20}\right)\right)\right|
\]
holds for all $\delta\in\left[1,\frac{\sqrt{97}-5}{2}\right]\cup\left[\alpha,4\right].$
Then for $\delta\in\left[1,\frac{\sqrt{97}-5}{2}\right]\cup\left[\alpha,4\right]$
and $n\geq176,$ we have 
\[
\left|\frac{c_{\delta}^{(5)}(n)}{\hat{c}_{\delta}^{(5)}(n)}-1\right|\left|\cos\left(\frac{4\pi}{5}\left(n+\frac{3\delta}{20}\right)\right)\right|<\left|\cos\left(\frac{4\pi}{5}\left(n+\frac{3\delta}{20}\right)\right)\right|,
\]
and so the coefficient $c_{\delta}^{(5)}(n)$ has the same sign as
$\hat{c}_{\delta}^{(5)}(n),$ that is the coefficient $c_{\delta}^{(5)}(n)$
has the same sign as $\cos\left(\frac{4\pi}{5}\left(n+\frac{3\delta}{20}\right)\right).$
Hence, when $n\geq176$ and $\delta\in\left[1,\frac{\sqrt{97}-5}{2}\right],$
the $q$-series coefficients of $Q_{5}^{\delta}(q)$ exhibit the sign
pattern $+-+--,$ and when $n\geq176$ and $\delta\in\left[\alpha,4\right],$
the $q$-series coefficients of $Q_{5}^{\delta}(q)$ exhibit the sign
pattern $+-+-+.$

The rest is to utilize \emph{Mathematica} to confirm that for $\delta\in\left[1,\frac{\sqrt{97}-5}{2}\right],$
the coefficients $\{c_{\delta}^{(5)}(n)\}_{0\leq n<176}$ follow the
sign pattern $+-+--,$ and for $\delta\in\left[\alpha,4\right],$
the coefficients $\{c_{\delta}^{(5)}(n)\}_{0\leq n<176}$ also follow
the sign pattern $+-+-+.$ Computing the values of the coefficients
$\{c_{\delta}^{(5)}(n)\}_{0\leq n<176},$ we find ranges in which
$\{c_{\delta}^{(5)}(n)\}_{0\leq n<176}$ are located for $\delta\in\left[1,\frac{\sqrt{97}-5}{2}\right]$
or $\delta\in\left[\alpha,4\right].$ The ranges are shown in Table
1 and Table 2 of the appendix.

When $\delta=-1$ or $\delta\in[-3,-2],$ we let $\delta^{\prime}=-\delta$
and
\[
Q_{5}^{\delta}(q)=:\sum_{n=1}^{\infty}\bar{c}_{\delta^{\prime}}^{(5)}(n)q^{n}.
\]
Then ${\bf m}=\{1,2\},$ ${\bf n}=\{5,5\},$ and ${\bf u}=\{-1,1\}.$
Hence $L=5,$ and $\Omega=-24/5.$ It can be easily computed that
\[
\mathcal{L}_{>0}=\{(1,5),(4,5)\}.
\]
It is easy to find that 
\[
\max\{\sqrt{\Delta(\varkappa,\kappa)}/k|(\varkappa,\kappa)\in\mathcal{L}_{>0},k\equiv\kappa\;\text{(mod\;}L)\}=\frac{2\sqrt{6}}{5\sqrt{5}}
\]
 and 
\[
\sqrt{\Delta(\varkappa,\kappa)}/k
\]
attains its maximum when $(\varkappa,\kappa,k)=(1,5,5),\;(4,5,5).$
When $(\varkappa,\kappa)=(1,5),(4,5),$ 
\[
\frac{24}{\Delta(\varkappa,\kappa)}\varUpsilon\left(\lambda_{m_{i},n_{i}}^{*}(h,k)\right)\frac{d_{i}^{2}}{n_{i}}\geq\delta^{\prime}
\]
for $0<\delta^{\prime}\leq5.$ Let 
\[
\tilde{c}_{\delta^{\prime}}^{(5)}(n):=\frac{4\pi\sqrt{\delta^{\prime}}}{5\sqrt{5}}\cos\left(\frac{2\pi}{5}\left(n-\frac{2\delta^{\prime}}{5}\right)\right)\left(n-\frac{\delta^{\prime}}{5}\right)^{-1/2}I_{-1}\left(\frac{4\pi\sqrt{\delta^{\prime}}}{5\sqrt{5}}\sqrt{n-\frac{\delta^{\prime}}{5}}\right)
\]
and 
\[
N^{\prime}=\left\lceil \sqrt{4\pi\left(n-\frac{\delta^{\prime}}{5}\right)}\right\rceil .
\]
Then we have 
\begin{align*}
\bar{c}_{\delta^{\prime}}^{(5)}(n)-\tilde{c}_{\delta^{\prime}}^{(5)}(n) & =\sum_{\substack{5<k\leq N\\
k\equiv5\,\left(\text{mod}\,5\right)
}
}\sum_{\substack{0\leq h\leq k\\
h\equiv1\,\left(\text{mod}\,5\right)
}
}\frac{2\pi\sqrt{\delta^{\prime}}}{k}e^{-2\pi\mathrm{i}nh/k}\mathrm{i}^{\delta^{\prime}\sum_{j=1}^{I}u_{j}}(-1)^{\delta^{\prime}\sum_{j=1}^{I}u_{j}\lambda_{m_{j},n_{j}}(h,k)}\\
 & \;\times\omega_{h.k}^{\delta^{\prime}}\varTheta_{h,k}^{\delta^{\prime}}\Pi_{h,k}\left(n-\frac{\delta^{\prime}}{5}\right)^{-1/2}I_{-1}\left(\frac{4\pi\sqrt{\delta^{\prime}}}{\sqrt{5}k}\sqrt{n-\frac{\delta^{\prime}}{5}}\right)\\
 & +\sum_{\substack{5<k\leq N\\
k\equiv5\,\left(\text{mod}\,5\right)
}
}\sum_{\substack{0\leq h\leq k\\
h\equiv4\,\left(\text{mod}\,5\right)
}
}\frac{2\pi\sqrt{\delta^{\prime}}}{k}e^{-2\pi\mathrm{i}nh/k}\mathrm{i}^{\delta^{\prime}\sum_{j=1}^{I}u_{j}}(-1)^{\delta^{\prime}\sum_{j=1}^{I}u_{j}\lambda_{m_{j},n_{j}}(h,k)}\\
 & \;\times\omega_{h.k}^{\delta^{\prime}}\varTheta_{h,k}^{\delta^{\prime}}\Pi_{h,k}\left(n-\frac{\delta^{\prime}}{5}\right)^{-1/2}I_{-1}\left(\frac{4\pi\sqrt{\delta^{\prime}}}{\sqrt{5}k}\sqrt{n-\frac{\delta^{\prime}}{5}}\right)+E_{\delta^{\prime},N}^{\prime}(n)
\end{align*}
Using \emph{Mathematica}, we find that
\begin{align*}
\left|E_{\delta^{\prime},N^{\prime}}(n)\right| & \leq54.366\times10.372^{\delta^{\prime}}+5.437\left(1.702^{\delta^{\prime}}+0.486^{\delta^{\prime}}+0.485^{\delta^{\prime}}\right)\\
 & \;+10.874\times5.978^{\delta^{\prime}}\left(2.443+1.002^{\delta^{\prime}}\right)
\end{align*}
and so
\begin{align*}
 & \left|\bar{c}_{\delta^{\prime}}^{(5)}(n)-\tilde{c}_{\delta^{\prime}}^{(5)}(n)\right|\\
 & \leq\left|E_{\delta^{\prime},N^{\prime}}^{\prime}(n)\right|+\sum_{\substack{5<k\leq N\\
k\equiv5\,\left(\text{mod}\,5\right)
}
}\sum_{\substack{0\leq h\leq k\\
h\equiv1\,\left(\text{mod}\,5\right)
}
}\frac{2\pi\sqrt{\delta^{\prime}}}{k}\left|\Pi_{h,k}\right|\left(n-\frac{\delta^{\prime}}{5}\right)^{-1/2}I_{-1}\left(\frac{4\pi\sqrt{\delta^{\prime}}}{\sqrt{5}k}\sqrt{n-\frac{\delta^{\prime}}{5}}\right)\\
 & +\sum_{\substack{5<k\leq N\\
k\equiv5\,\left(\text{mod}\,5\right)
}
}\sum_{\substack{0\leq h\leq k\\
h\equiv4\,\left(\text{mod}\,5\right)
}
}\frac{2\pi\sqrt{\delta^{\prime}}}{k}\left|\Pi_{h,k}\right|\left(n-\frac{\delta^{\prime}}{5}\right)^{-1/2}I_{-1}\left(\frac{4\pi\sqrt{\delta^{\prime}}}{\sqrt{5}k}\sqrt{n-\frac{\delta^{\prime}}{5}}\right)\\
 & \leq\sum_{\substack{5<k\leq N\\
k\equiv5\,\left(\text{mod}\,5\right)
}
}8\delta^{\prime2}\pi\left|\frac{1}{1-e^{\frac{2\pi\mathrm{i}}{5}}}\right|\left(n-\frac{\delta^{\prime}}{5}\right)^{-1/2}I_{-1}\left(\frac{4\pi\sqrt{\delta^{\prime}}}{\sqrt{5}k}\sqrt{n-\frac{\delta^{\prime}}{5}}\right)\\
 & =8\delta^{\prime2}\pi\left|\frac{1}{1-e^{\frac{2\pi\mathrm{i}}{5}}}\right|\left(n-\frac{\delta^{\prime}}{5}\right)^{-1/2}\sum_{1<k^{\prime}\leq N/5}I_{-1}\left(\frac{4\pi\sqrt{\delta^{\prime}}}{5\sqrt{5}k^{\prime}}\sqrt{n-\frac{\delta^{\prime}}{5}}\right).
\end{align*}

Similarly, take
\[
L_{\delta^{\prime},n}:=\frac{2\pi\sqrt{\delta^{\prime}}}{5\sqrt{5}}\sqrt{n-\frac{\delta^{\prime}}{5}}.
\]
When $\delta=-1$ or $\delta\in[-3,-2],$ i.e. $\delta^{\prime}=1$
or $\delta^{\prime}\in[2,3],$ using Lemma \ref{lem:z}, we get that
for $n\geq143,$ 
\begin{align*}
\left|\bar{c}_{\delta^{\prime}}^{(5)}(n)-\tilde{c}_{\delta^{\prime}}^{(5)}(n)\right| & \leq\left|E_{\delta^{\prime},N^{\prime}}^{\prime}(n)\right|+8\sqrt{\delta^{\prime}}\pi\left|\frac{1}{1-e^{\frac{2\pi\mathrm{i}}{5}}}\right|\left(n-\frac{\delta^{\prime}}{5}\right)^{-1/2}\\
 & \;\times\left(L_{\delta^{\prime},n}\log\left(\frac{\sqrt{5}L_{\delta^{\prime},n}}{\sqrt{\pi\delta^{\prime}}}\frac{43}{42}\right)+2I_{-1}\left(L_{\delta^{\prime},n}\right)-\left(2-\gamma-\frac{5}{84}\right)L_{\delta^{\prime},n}\right).
\end{align*}
and so
\begin{align*}
 & \left|\frac{\bar{c}_{\delta^{\prime}}^{(5)}(n)}{\hat{c}_{\delta^{\prime}}(n)}-1\right|\left|\cos\left(\frac{2\pi}{5}\left(n-\frac{2\delta^{\prime}}{5}\right)\right)\right|\\
 & \leq10\sqrt{5}\left|\frac{1}{1-e^{\frac{2\pi\mathrm{i}}{5}}}\right|\frac{L_{\delta^{\prime},n}\log\left(L_{\delta^{\prime},n}\right)+2I_{-1}(L_{\delta^{\prime},n})+\nu(\delta^{\prime})L_{\delta^{\prime},n}}{I_{-1}(2L_{\delta^{\prime},n})}
\end{align*}
with 
\[
\begin{aligned}\nu^{\prime}(\delta^{\prime}) & =\frac{1}{2}\log\left(\frac{1}{\delta^{\prime}}\right)-1.107+\frac{0.914\times5.978^{\delta^{\prime}}\left(2.443+1.002^{\delta^{\prime}}\right)}{\delta^{\prime}}\\
 & \;+\frac{0.084\left(54.366\times10.372^{\delta^{\prime}}+5.437\left(1.702^{\delta^{\prime}}+0.486^{\delta^{\prime}}+0.485^{\delta^{\prime}}\right)\right)}{\delta^{\prime}}.
\end{aligned}
\]
For $n\geq143,$ and $\delta^{\prime}\in\{1\}\cup[2,3],$ we have
$L_{\delta^{\prime},n}>3,$ $\nu^{\prime}(\delta^{\prime})>0.$ Apply
Lemma \ref{lem:z2}. This establishes that
\[
10\sqrt{5}\left|\frac{1}{1-e^{\frac{2\pi\mathrm{i}}{5}}}\right|M\left(\nu^{\prime}(\delta^{\prime}),L_{\delta^{\prime},n}\right)
\]
is decreasing in $n$ for $n\geq143.$ Using the software \emph{Mathematica}
we can verify that 
\[
10\sqrt{5}\left|\frac{1}{1-e^{\frac{2\pi\mathrm{i}}{5}}}\right|M\left(\nu^{\prime}(\delta^{\prime}),L_{\delta^{\prime},143}\right)<\min_{n=0,1,2,3,4}\left|\cos\left(\frac{2\pi}{5}\left(n-\frac{2\delta^{\prime}}{5}\right)\right)\right|
\]
holds for all $\delta^{\prime}\in\{1\}\cup[2,3],$ so for $n\geq143,$
$\delta\in\{-1\}\cup[-3,-2],$ we have $c_{\delta}^{(5)}(n)=\bar{c}_{\delta^{\prime}}^{(5)}(n)$
has the same sign as $\cos\left(\frac{2\pi}{5}\left(n-\frac{2\delta^{\prime}}{5}\right)\right).$
This indicates that the coefficients $\{c_{-1}^{(5)}(n)\}_{n\geq143}$
exhibit the sign pattern $++---$ and the coefficients $\{c_{\delta}^{(5)}(n)\}_{n\geq143}$
exhibit the sign pattern $+++--$ for $\delta\in\left[-3,-2\right].$
Using \emph{Mathematica}, we can easily confirm that the coefficients
$\{c_{-1}^{(5)}(n)\}_{0\leq n<143}$ follow the sign pattern $++---.$
Similarly, we can obtain ranges in which $\{c_{\delta}^{(5)}(n)\}_{0\leq n<143}$
are located for $\delta\in\left[-3,-2\right].$ The ranges are shown
in Table 3 of the appendix. From this table, we find that the coefficients
$\{c_{\delta}^{(5)}(n)\}_{0\leq n<143}$ exhibit the sign pattern
$+++--$ for $\delta\in\left[-3,-2\right].$ This completes the proof
of Theorem \ref{thm:6}. \qed

\subsection{Proof of Theorem \ref{thm:7}}

It is easy to see that the $q$-series coefficients of $Q_{6}(q)^{\delta}$
exhibit the sign pattern $(+-)^{3}$ is equivalent to that $Q_{6}(-q)^{\delta}$
has non-negative coefficients. For $\delta\geq3,$ we can write $\delta$
as

\[
\delta=\delta_{1}+\delta_{2},
\]
where $\delta_{1}$ is a non-negative integer and $\delta_{2}$ is
a real number such that $3\leq\delta_{2}<4.$ Since 
\[
Q_{6}(-q)^{\delta_{1}}=\left[(-q,-q^{5};q^{6})_{\infty}\right]^{\delta_{1}}
\]
has non-negative coefficients, it suffices to show that $Q_{6}(-q)^{\delta_{2}}$
has non-negative coefficients for $3\leq\delta_{2}<4.$ Namely, we
only need to show that the $q$-series coefficients of $Q_{6}(q)^{\delta_{2}}$
exhibit the sign pattern $(+-)^{3}$ for $3\leq\delta_{2}<4.$

For $Q_{6}(q)^{\delta_{2}},$ ${\bf m}=\{1\},$ ${\bf n}=\{6\},$
and ${\bf u}=\{1\}.$ Hence $L=6,$ and $\Omega=2.$ We compute that
\[
\mathcal{L}_{>0}=\{(1,2),(1,3),(1,4),(2,3),(2,6),(3,4),(3,6),(4,6)\}.
\]
It is easily deduced that when $(\varkappa,\kappa,k)=(3,6,6),$ $\sqrt{\Delta(\varkappa,\kappa)}/k$
attains $\max\{\sqrt{\Delta(\varkappa,\kappa)}/k|(\varkappa,\kappa)\in\mathcal{L}_{>0},k\equiv\kappa\;\text{(mod\;}L)\}.$
When $(\varkappa,\kappa)\in\mathcal{L}_{>0},$ 
\[
\frac{24}{\Delta(\varkappa,\kappa)}\varUpsilon\left(\lambda_{m_{1},n_{1}}^{*}(h,k)\right)\frac{d_{1}^{2}}{n_{1}}=12>\delta_{2}
\]
when $3\leq\delta_{2}<4.$ Let 
\[
Q_{6}^{\delta_{2}}(q)=:\sum_{n=1}^{\infty}c_{\delta_{2}}^{(6)}(n)q^{n}
\]
For $(\varkappa,\kappa,k)=(3,6,6),$ we compute that the $I$-Bessel
term is 
\[
\frac{\sqrt{\delta_{2}}\pi}{6}(-1)^{n}\left(n+\frac{\delta_{2}}{12}\right)^{-\frac{1}{2}}I_{-1}\left(\frac{\sqrt{\delta_{2}}\pi}{3}\sqrt{n+\frac{\delta_{2}}{12}}\right).
\]
Proceeding as in the proof of Theorem \ref{thm:6}, we get that when
$3\leq\delta_{2}<4,$ and $n\geq57,$ the coefficient $c_{\delta_{2}}^{(6)}(n)$
has the same sign as $(-1)^{n},$ i.e. the coefficients $\{c_{\delta_{2}}^{(6)}(n)\}_{n\geq57}$
exhibit the sign pattern $(+-)^{3}$ for $3\leq\delta_{2}<4.$ We
used \emph{Mathematica} to compute ranges in which the coefficients
$\{c_{\delta_{2}}^{(6)}(n)\}_{0\leq n<57}$ are located for $3\leq\delta_{2}\leq4.$
The ranges are shown in Table 4 of the appendix. From the table, we
discover that the coefficients $\{c_{\delta_{2}}^{(6)}(n)\}_{0\leq n<57}$
are alternating when $3\leq\delta_{2}<4.$ This concludes the proof
of Theorem \ref{thm:7}. \qed

\subsection{Proof of Theorem \ref{thm:8}}

Let $Q_{8}(q)$ be defined as in Conjecture \ref{conj:3}. Then ${\bf m}=\{1,3\},$
${\bf n}=\{8,8\},$ and ${\bf u}=\{1,-1\}.$ Hence $L=8,\Omega=12$
and $\mathcal{L}_{>0}=\{(3,8),(5,8)\}.$ After computations, we find
that when $(\varkappa,\kappa,k)=(3,8,8),\;(5,8,8),$ $\sqrt{\Delta(\varkappa,\kappa)}/k$
attains $\max\{\sqrt{\Delta(\varkappa,\kappa)}/k|(\varkappa,\kappa)\in\mathcal{L}_{>0},k\equiv\kappa\;\text{(mod\;}L)\}=\sqrt{3}/4.$

When $(\varkappa,\kappa)=(3,8),(5,8),i=1,2,$ 
\[
\frac{24}{\Delta(\varkappa,\kappa)}\varUpsilon\left(\lambda_{m_{i},n_{i}}^{*}(h,k)\right)\frac{d_{i}^{2}}{n_{i}}\geq\delta
\]
when $0<\delta\leq2.$ When $\beta<\delta\leq4,$ the inequality \eqref{eq:thm1}
is not satisfied and we can not use Theorem \ref{thm:1} to handle
this case.

For $\delta=2,$ we let 
\[
Q_{8}(q)^{\delta}=:\sum_{n=1}^{\infty}c_{\delta}^{(8)}(n)q^{n}
\]
and
\[
\hat{c}_{\delta}^{(8)}(n):=\frac{\pi\sqrt{2\delta}}{4}\cos\left(\frac{3\pi}{4}n\right)\left(n+\frac{\delta}{2}\right)^{-1/2}I_{-1}\left(\frac{\pi\sqrt{2\delta}}{4}\sqrt{n+\frac{\delta}{2}}\right).
\]
It is easily seen that $\cos\left(\frac{3\pi}{4}n\right)>0$ if $n\equiv0,3,5\;(\bmod8),$
and $\cos\left(\frac{3\pi}{4}n\right)<0$ if $n\equiv1,4,7\;\left(\bmod8\right).$
Proceeding as in the proof of Theorem \ref{thm:6}, we can get that
when $n\geq138,$ and $n\equiv0,1,3,4,5,7\;\left(\bmod8\right),$
the coefficients $c_{2}^{(8)}(n)$ have the same sign as $\cos\left(\frac{3\pi}{4}n\right).$
For $n\equiv2,6\;\left(\bmod8\right),$ we have $\cos\left(\frac{3\pi}{4}n\right)=0.$
Let 
\[
\tilde{c}_{2}^{(8)}(n):=\frac{\pi\sqrt{2\delta}}{4}\left(\cos\left(\frac{3\pi}{8}n+\frac{\delta\pi}{2}\right)+\cos\left(\frac{5\pi}{8}n+\frac{\delta\pi}{2}\right)\right)\left(n+\frac{\delta}{2}\right)^{-1/2}I_{-1}\left(\frac{\pi\sqrt{2\delta}}{8}\sqrt{n+\frac{\delta}{2}}\right)
\]
and 
\[
N=\left\lceil 2\sqrt{\pi\left(n+\frac{\delta}{2}\right)}\right\rceil .
\]

Using \emph{Mathematica}, we get 
\[
\left|E_{2,N}(n)\right|\leq932477.
\]
For $n\equiv2,6\left(\bmod8\right),$ we have 
\begin{align*}
 & \left|c_{2}^{(8)}(n)-\hat{c}_{2}^{(8)}(n)-\tilde{c}_{2}^{(8)}(n)\right|\\
 & \leq\left|E_{2,N}(n)\right|+\sum_{\substack{16<k\leq N\\
k\equiv8\,\left(\text{mod}\,8\right)
}
}\sum_{\substack{0\leq h\leq k\\
h\equiv3\,\left(\text{mod}\,8\right)
}
}\frac{2\pi\sqrt{2}}{k}\left|\Pi_{h,k}\right|\left(n+1\right)^{-1/2}I_{-1}\left(\frac{4\pi}{k}\sqrt{n+1}\right)\\
 & +\sum_{\substack{16<k\leq N\\
k\equiv8\,\left(\text{mod}\,8\right)
}
}\sum_{\substack{0\leq h\leq k\\
h\equiv5\,\left(\text{mod}\,8\right)
}
}\frac{2\pi\sqrt{2}}{k}\left|\Pi_{h,k}\right|\left(n+1\right)^{-1/2}I_{-1}\left(\frac{4\pi}{k}\sqrt{n+1}\right)\\
 & \leq\sum_{\substack{24<k\leq N\\
k\equiv12\,\left(\text{mod}\,12\right)
}
}8\sqrt{2}\pi\left|\frac{1}{1-e^{\frac{\pi\mathrm{i}}{4}}}\right|\left(n+1\right)^{-1/2}I_{-1}\left(\frac{4\pi}{k}\sqrt{n+1}\right)\\
 & =8\sqrt{2}\pi\left|\frac{1}{1-e^{\frac{\pi\mathrm{i}}{4}}}\right|\left(n+1\right)^{-1/2}\sum_{3\leq k^{\prime}\leq N/8}I_{-1}\left(\frac{\pi}{2k^{\prime}}\sqrt{n+1}\right).
\end{align*}

In order to bound 
\[
\left|\frac{c_{2}^{(8)}(n)-\hat{c}_{2}^{(8)}(n)-\tilde{c}_{2}^{(8)}(n)}{\tilde{c}_{2}^{(8)}(n)}\right|,
\]
we need the following two lemmas.
\begin{lem}
\label{lem:zplus1} For any real $x>0$ and integer $y>3$ we have
\[
\sum_{3\leq k\leq y}I_{-1}\left(\frac{2x}{k}\right)\leq x\log y+3I_{-1}\left(\frac{2x}{3}\right)-\left(\frac{5}{2}-\gamma-\frac{1}{2y}\right)x.
\]
Here $\gamma=0.577216\cdots$ is the Euler-Mascheroni constant.
\end{lem}
\begin{proof}
It is easily seen that for $n\geq1,$
\[
\sum_{k=3}^{y}\frac{1}{k^{2n+1}}<\sum_{k=3}^{\infty}\frac{1}{k^{2n+1}}<\frac{1}{3^{2n+1}}+\int_{3}^{\infty}\frac{1}{x^{2n+1}}dx=\frac{1}{3^{2n+1}}+\frac{1}{2n3^{2n}}<\frac{1}{3^{2n}}.
\]
Then 
\begin{align*}
\sum_{3\leq k\leq y}I_{-1}\left(\frac{2x}{k}\right) & =\sum_{n\geq0}\frac{x^{2n+1}}{n!(n+1)!}\sum_{3\leq k\leq y}\frac{1}{k^{2n+1}}\\
 & =x\sum_{3\leq k\leq y}\frac{1}{k}+\sum_{n\geq1}\frac{x^{2n+1}}{n!(n+1)!}\sum_{3\leq k\leq y}\frac{1}{k^{2n+1}}\\
 & \leq x\left(\log y+\gamma-\frac{3}{2}+\frac{1}{2y}\right)+3\left(\sum_{n\geq0}\frac{(x/3)^{2n+1}}{n!(n+1)!}-\frac{x}{3}\right)\\
 & \leq x\log y+3I_{-1}\left(\frac{2x}{3}\right)-\left(\frac{5}{2}-\gamma-\frac{1}{2y}\right)x,
\end{align*}
where, in the third step we have used the well-known result:
\[
\sum_{k=1}^{y}\frac{1}{k}\leq\log y+\gamma+\frac{1}{2y}.
\]
\end{proof}
\begin{lem}
\label{lem:zplus2} Suppose that $s>0,$ then the function
\[
\widehat{M}\left(s,t\right):=\frac{t\log\left(t\right)+3I_{-1}\left(\frac{2t}{3}\right)+st}{I_{-1}(t)}
\]
 is decreasing in $t$ for $t\geq5.$
\end{lem}
\begin{proof}
By the definition of $I_{-1}(x),$ we know that $st/I_{-1}(t)$ is
decreasing in $t$ for $t>0.$ For $\frac{t\log\left(t\right)}{I_{-1}(t)},$
we have 
\[
\left(\frac{t\log\left(t\right)}{I_{-1}(t)}\right)^{\prime}=\frac{t^{-1}\left(t^{-1}I_{-1}(t)\right)-\left(t^{-1}I_{-1}(t)\right)^{\prime}\log t}{\text{\ensuremath{\left(t^{-1}I_{-1}(t)\right)}}^{2}},
\]
and 
\begin{align*}
 & t^{-1}\left(t^{-1}I_{-1}(t)\right)-\left(t^{-1}I_{-1}(t)\right)^{\prime}\log t\\
 & =\frac{1}{4}\sum_{l\geq0}\frac{\left(\frac{t}{2}\right)^{2l-1}}{l!(l+1)!}-\frac{\log t}{2}\sum_{l\geq0}\frac{l\left(\frac{t}{2}\right)^{2l-1}}{l!(l+1)!}\\
 & =\frac{1}{4}\left(\frac{2}{t}-\sum_{l\geq0}\frac{(2l\log t-1)\left(\frac{t}{2}\right)^{2l-1}}{l!(l+1)!}\right)\\
 & \leq\frac{1}{4}\left(\frac{2}{t}-\frac{t}{4}(2\log t-1)\right).
\end{align*}
For all $t\geq\sqrt{e},$ the function $\frac{1}{2t}-\frac{t}{16}(2\log t-1)$
is decreasing in $t.$ This implies that $\frac{1}{2t}-\frac{t}{16}(2\log t-1)\leq\frac{1}{10}-\frac{5}{16}(2\log5-1)<0$
for $t\geq5.$ So $t\log t/I_{-1}(t)$ is decreasing in $t$ for $t\geq5.$

We now prove that the function $I_{-1}(2x)/I(3x)$ is decreasing for
$x\geq5/3.$ It suffices to prove 
\begin{equation}
\left(\frac{I_{-1}(2x)}{I_{-1}(3x)}\right)^{\prime}=\frac{2I_{-1}^{\prime}(2x)I_{-1}(3x)-3I_{-1}(2x)I_{-1}^{\prime}(3x)}{I_{-1}(3x)^{2}}<0\label{eq:8-2}
\end{equation}
for all $x>0.$ Utilizing the functional relation for the modified
Bessel function $I_{-1}:$
\[
I_{-1}^{\prime}(x)=I_{0}(x)-x^{-1}I_{-1}(x),
\]
we get 
\begin{equation}
\begin{aligned} & 2I_{-1}^{\prime}(2x)I_{-1}(3x)-3I_{-1}(2x)I_{-1}^{\prime}(3x)\\
 & =2I_{0}(2x)I_{-1}(3x)-3I_{0}(3x)I_{-1}(2x)\\
 & =\frac{1}{x}\left(\frac{2xI_{0}(2x)}{I_{-1}(2x)}-\frac{3xI_{0}(3x)}{I_{-1}(3x)}\right).
\end{aligned}
\label{eq:8-1}
\end{equation}
Recall a result of Simpson and Spector \cite{SS} that for all $\alpha>0,$
the function $xI_{\alpha}(x)/I_{\alpha+1}(x)$ is strictly monotone
decreasing on $(0,+\infty).$ This, together the fact $I_{-1}(x)=I_{1}(x)$
and \eqref{eq:8-1}, completes the proof of \eqref{eq:8-2}.
\end{proof}
Let
\[
L_{\delta,n}=\frac{\pi\sqrt{2\delta}}{8}\sqrt{n+\frac{\delta}{2}}.
\]
Then 
\begin{align*}
\left|c_{2}^{(8)}(n)-\hat{c}_{2}^{(8)}(n)-\tilde{c}_{2}^{(8)}(n)\right| & \leq\left|E_{2,N}(n)\right|+8\sqrt{2}\pi\left|\frac{1}{1-e^{\frac{\pi\mathrm{i}}{4}}}\right|\left(n+1\right)^{-1/2}\\
 & \;\times\left(L_{2,n}\log\left(\frac{N}{8}\right)+3I_{-1}\left(\frac{2L_{2,n}}{3}\right)-\left(\frac{5}{2}-\gamma-\frac{4}{N}\right)L_{2,n}\right).
\end{align*}
When $n\geq565,$ we have
\begin{align*}
 & L_{2,n}=\frac{\pi}{4}\sqrt{n+1},\\
 & 2\sqrt{\pi\left(n+1\right)}\geq2\sqrt{566\pi}\geq84
\end{align*}
and so
\[
N=\left\lceil 2\sqrt{\pi\left(n+1\right)}\right\rceil \leq2\sqrt{\pi\left(n+1\right)}+1\leq\frac{85}{84}\times2\sqrt{\pi\left(n+1\right)}=\frac{85}{84}\times\frac{8}{\sqrt{\pi}}L_{2,n}.
\]
Then
\begin{align*}
\left|c_{2}^{(8)}(n)-\hat{c}_{2}^{(8)}(n)-\tilde{c}_{2}^{(8)}(n)\right| & \leq\left|E_{2,N}(n)\right|+8\sqrt{2}\pi\left|\frac{1}{1-e^{\frac{\pi\mathrm{i}}{4}}}\right|\left(n+1\right)^{-1/2}\\
 & \;\times\left(L_{2,n}\log\left(\sqrt{\frac{1}{\pi}}L_{2,n}\frac{85}{84}\right)+3I_{-1}\left(\frac{2L_{2,n}}{3}\right)-\left(\frac{5}{2}-\gamma-\frac{4}{84}\right)L_{2,n}\right).
\end{align*}
This implies that
\begin{align*}
 & \left|\frac{c_{2}^{(8)}(n)-\hat{c}_{2}^{(8)}(n)-\tilde{c}_{2}^{(8)}(n)}{\tilde{c}_{2}^{(8)}(n)}\right|\left|\cos\left(\frac{3\pi}{8}n+\pi\right)+\cos\left(\frac{5\pi}{8}n+\pi\right)\right|\\
 & \leq16\sqrt{2}\left|\frac{1}{1-e^{\frac{\pi\mathrm{i}}{4}}}\right|\frac{L_{2,n}\log\left(L_{2,n}\right)+3I_{-1}\left(\frac{2L_{2,n}}{3}\right)+\mu(2)L_{2,n}}{I_{-1}(L_{2,n})}
\end{align*}
with $\mu(2)=25564.$ It is easy to verify that $L_{2,n}\geq5$ for
$n\geq565.$ Apply Lemma \ref{lem:zplus2}. We obtain that the sequence
\[
\left\{ \frac{L_{2,n}\log\left(L_{2,n}\right)+3I_{-1}\left(\frac{2L_{2,n}}{3}\right)+\mu(2)L_{2,n}}{I_{-1}(L_{2,n})}\right\} _{n\geq565}
\]
is decreasing monotonically with respect to $n.$ Employing the software
\emph{Mathematica} we can verify that 
\[
16\sqrt{2}\left|\frac{1}{1-e^{\frac{\pi\mathrm{i}}{4}}}\right|\widehat{M}\left(\mu(2),L_{2,565}\right)<\min_{n=2,6,10,14}\left|\cos\left(\frac{3\pi}{8}n+\pi\right)+\cos\left(\frac{5\pi}{8}n+\pi\right)\right|.
\]
This indicates that when $n\geq565,$ and $n\equiv2,6,10,14\left(\bmod16\right),$
the coefficient $c_{2}^{(8)}(n)$ has the same sign as $\cos\left(\frac{3\pi}{8}n+\pi\right)+\cos\left(\frac{5\pi}{8}n+\pi\right).$
Combining the above results we get that the coefficients $\{c_{2}^{(8)}(n)\}_{n\geq565}$
exhibit the sign pattern $+-++-+--+--+-++-.$ We employ \emph{Mathematica}
to compute directly the coefficients $\{c_{2}^{(8)}(n)\}_{0\leq n<565}$
and find that they also exhibit the same sign pattern.

If $\delta<0,$ then we let $\delta=-\delta'$ and
\[
\left(Q_{8}(q)\right)^{\delta}=\left(Q_{8}^{-1}(q)\right)^{\delta^{\prime}}=\sum_{n=1}^{\infty}\bar{c}_{\delta^{\prime}}^{(8)}(n)q^{n}.
\]
We first consider the case 
\[
-0.99\leq\delta\leq\frac{7-\sqrt{73}}{2}.
\]
Proceeding as in the proof of Theorem \ref{thm:6}, we get that when
$n\geq479,$ the coefficient $\bar{c}_{\delta^{\prime}}^{(8)}(n)$
has the same sign as 
\[
\grave{c}_{\delta^{\prime}}^{(8)}(n):=\frac{\pi\sqrt{2\delta^{\prime}}}{4}\cos\left(\frac{\pi}{4}\left(n-\delta^{\prime}\right)\right)\left(n-\frac{\delta^{\prime}}{5}\right)^{-1/2}I_{-1}\left(\frac{\pi\sqrt{2\delta^{\prime}}}{4}\sqrt{n+\frac{\delta^{\prime}}{2}}\right).
\]
Then the coefficients $\{\bar{c}_{\delta^{\prime}}^{(8)}(n)\}_{n\geq479}$
exhibit the sign pattern $+++----+.$ We employ \emph{Mathematica}
to calculate the first $479$ coefficients of $Q_{8}^{\delta}(q)$
for $-0.99\leq\delta\leq\frac{7-\sqrt{73}}{2},$ and find ranges,
in which the coefficients $\{\bar{c}_{\delta^{\prime}}^{(8)}(n)\}_{0\leq n<479}$
are located for $-0.99\leq\delta\leq\frac{7-\sqrt{73}}{2}.$ The ranges
are displayed in Table 5 of the appendix. From the table, we obtain
that the coefficients $\{\bar{c}_{\delta^{\prime}}^{(8)}(n)\}_{0\leq n<479}$
also exhibit the sign pattern $+++----+.$

When $\delta=-2,$ the coefficient $\bar{c}_{2}^{(8)}(n)$ has the
same sign as $\tilde{c}_{2}(n)$ for $n\geq140,$  and $n\equiv1,2,3,5,6,7\left(\bmod8\right).$
Proceeding as in the proof of Theorem \ref{thm:6}, we obtain that
for $n\geq567$ and $n\equiv0,4\left(\bmod8\right),$ the coefficient
$\bar{c}_{2}^{(8)}(n)$ has the same sign as
\[
\frac{\pi}{2}\left(\cos\left(\frac{\pi}{8}n-\frac{\pi}{4}\right)+\cos\left(\frac{7\pi}{8}n-\frac{7\pi}{4}\right)\right)\left(n-1\right)^{-1/2}I_{-1}\left(\frac{\pi}{4}\sqrt{n-1}\right)
\]
Then when $n\geq567,$ the coefficients of $Q_{8}^{-2}(q)=\left(Q_{8}^{-1}(q)\right)^{2}$
exhibit the length $16$ sign pattern $+++++----+++----.$ Employing
\emph{Mathematica} to directly compute the coefficients of $Q_{8}^{-2}(q)$
when $n\leq567,$ we easily find that the coefficients exhibit the
same sign pattern. This finishes the proof of Theorem \ref{thm:8}.
\qed

\subsection{Proof of Theorem \ref{thm:10}}

Let $Q_{10}(q)$ be defined as in Conjecture \ref{conj:4}. Then ${\bf m}=\{1,3\},$
${\bf n}=\{10,10\},$ and ${\bf u}=\{1,-1\}.$ Hence $L=10,$ and
$\Omega=\frac{72}{5}.$ We compute that 
\[
\mathcal{L}_{>0}=\{(3,10),(7,10)\}.
\]
After computations, we find that the maximum of $\sqrt{\Delta(\varkappa,\kappa)}/k$
with $(\varkappa,\kappa)\in\mathcal{L}_{>0}$ and $k\equiv\kappa\;\text{(mod\;}L)$
is $\frac{3\sqrt{2}}{5\sqrt{5}}$ and $\sqrt{\Delta(\varkappa,\kappa)}/k$
attains this maximum when
\[
(\varkappa,\kappa,k)=(3,10,10),\;(7,10,10).
\]
When $(\varkappa,\kappa)=(3,10),(7,10),$ we have
\[
\frac{24}{\Delta(\varkappa,\kappa)}\min_{1\leq i\leq I}\left(\varUpsilon\left(\lambda_{m_{i},n_{i}}^{*}(h,k)\right)\frac{d_{i}^{2}}{n_{i}}\right)=\frac{5}{3}>\delta
\]
when $0<\delta\leq1.$ We let 
\[
Q_{10}^{\delta}(q)=:\sum_{n=1}^{\infty}c_{\delta}^{(10)}(n)q^{n}
\]
and
\[
\hat{c}_{\delta}^{(10)}(n):=\frac{\pi\sqrt{3\delta}}{2\sqrt{5}}\cos\left(\frac{3\pi}{5}\left(n+\frac{8\delta}{15}\right)\right)\left(n+\frac{3\delta}{5}\right)^{-1/2}I_{-1}\left(\frac{\pi\sqrt{3\delta}}{2\sqrt{5}}\sqrt{n+\frac{3\delta}{5}}\right),
\]
and set
\[
N=\left\lceil \sqrt{4\pi\left(n+\frac{3\delta}{5}\right)}\right\rceil .
\]
Then 
\begin{align*}
c_{\delta}^{(10)}(n)-\hat{c}_{\delta}^{(10)}(n) & =\sum_{\substack{10<k\leq N\\
k\equiv10\,\left(\text{mod}\,10\right)
}
}\sum_{\substack{0\leq h\leq k\\
h\equiv3\,\left(\text{mod}\,10\right)
}
}\frac{2\pi\delta^{1/2}}{k}e^{-2\pi\mathrm{i}nh/k}\mathrm{i}^{\delta\sum_{j=1}^{I}u_{j}}(-1)^{\delta\sum_{j=1}^{I}u_{j}\lambda_{m_{j},n_{j}}(h,k)}\\
 & \;\times\omega_{h.k}^{\delta}\varTheta_{h,k}^{\delta}\Pi_{h,k}\left(n+\frac{3\delta}{5}\right)^{-1/2}I_{-1}\left(\frac{5\pi\sqrt{3\delta}}{\sqrt{5}k}\sqrt{n+\frac{3\delta}{5}}\right)\\
 & +\sum_{\substack{10<k\leq N\\
k\equiv10\,\left(\text{mod}\,10\right)
}
}\sum_{\substack{0\leq h\leq k\\
h\equiv7\,\left(\text{mod}\,10\right)
}
}\frac{2\pi\delta^{1/2}}{k}e^{-2\pi\mathrm{i}nh/k}\mathrm{i}^{\delta\sum_{j=1}^{I}u_{j}}(-1)^{\delta\sum_{j=1}^{I}u_{j}\lambda_{m_{j},n_{j}}(h,k)}\\
 & \;\times\omega_{h.k}^{\delta}\varTheta_{h,k}^{\delta}\Pi_{h,k}\left(n+\frac{3\delta}{5}\right)^{-1/2}I_{-1}\left(\frac{5\pi\sqrt{3\delta}}{\sqrt{5}k}\sqrt{n+\frac{3\delta}{5}}\right)+E_{\delta,N}(n).
\end{align*}
and so 
\[
\begin{aligned} & \left|c_{\delta}^{(10)}(n)-\hat{c}_{\delta}^{(10)}(n)\right|\\
 & \leq\left|E_{\delta,N}(n)\right|+\sum_{\substack{10<k\leq N\\
k\equiv10\,\left(\text{mod}\,10\right)
}
}\sum_{\substack{0\leq h\leq k\\
h\equiv3\,\left(\text{mod}\,10\right)
}
}\frac{2\pi\delta^{1/2}}{k}\left|\Pi_{h,k}\right|\left(n+\frac{3\delta}{5}\right)^{-1/2}I_{-1}\left(\frac{5\pi\sqrt{3\delta}}{\sqrt{5}k}\sqrt{n+\frac{3\delta}{5}}\right)\\
 & +\sum_{\substack{10<k\leq N\\
k\equiv10\,\left(\text{mod}\,10\right)
}
}\sum_{\substack{0\leq h\leq k\\
h\equiv7\,\left(\text{mod}\,10\right)
}
}\frac{2\pi\delta^{1/2}}{k}\left|\Pi_{h,k}\right|\left(n+\frac{3\delta}{5}\right)^{-1/2}I_{-1}\left(\frac{5\pi\sqrt{3\delta}}{\sqrt{5}k}\sqrt{n+\frac{3\delta}{5}}\right)\\
 & \leq\sum_{\substack{10<k\leq N\\
k\equiv10\,\left(\text{mod}\,10\right)
}
}8\delta^{1/2}\pi\left|\frac{1}{1-e^{\frac{\pi\mathrm{i}}{5}}}\right|\left(n+\frac{3\delta}{5}\right)^{-1/2}I_{-1}\left(\frac{5\pi\sqrt{3\delta}}{\sqrt{5}k}\sqrt{n+\frac{3\delta}{5}}\right)\\
 & =8\delta^{1/2}\pi\left|\frac{1}{1-e^{\frac{\pi\mathrm{i}}{5}}}\right|\left(n+\frac{3\delta}{5}\right)^{-1/2}\sum_{1<k^{\prime}\leq N/10}I_{-1}\left(\frac{\pi\sqrt{3\delta}}{2\sqrt{5}k^{\prime}}\sqrt{n+\frac{3\delta}{5}}\right).
\end{aligned}
\]
Let 
\[
L_{\delta,n}=\frac{\pi\sqrt{3\delta}}{4\sqrt{5}}\sqrt{n+\frac{3\delta}{5}}.
\]

When $\delta=1,$ we use the software \emph{Mathematica} to get
\begin{align*}
\left|E_{1,N}(n)\right| & \leq61597.1.
\end{align*}
Applying Lemma \ref{lem:z} weget 
\begin{align*}
\left|c_{1}^{(10)}(n)-\hat{c}_{1}^{(10)}(n)\right| & \leq\left|E_{1,N}(n)\right|+8\pi\left|\frac{1}{1-e^{\frac{\pi i}{5}}}\right|\left(n+\frac{3}{5}\right)^{-1/2}\\
 & \;\times\left(L_{1,n}\log\left(\frac{N}{10}\right)+2I_{-1}\left(L_{1,n}\right)-\left(2-\gamma-\frac{5}{N}\right)L_{1,n}\right),
\end{align*}
with 
\[
L_{1,n}=\frac{\pi\sqrt{3}}{4\sqrt{5}}\sqrt{n+\frac{3}{5}}.
\]
 For $n\geq241,$ we have
\[
\sqrt{4\pi\left(n+\frac{3\delta}{5}\right)}\geq\sqrt{4\pi\left(241+\frac{3}{5}\right)}\geq55,
\]
so that

\[
N=\left\lceil \sqrt{4\pi\left(n+\frac{3}{5}\right)}\right\rceil \leq\sqrt{4\pi\left(n+\frac{3}{5}\right)}+1\leq\frac{56}{55}\sqrt{4\pi\left(n+\frac{3}{5}\right)}\leq\frac{56}{55}\times\frac{8\sqrt{5}L_{1,n}}{\sqrt{3\pi}}.
\]
Then 
\begin{align*}
\left|c_{1}^{(10)}(n)-\hat{c}_{1}^{(10)}(n)\right| & \leq\left|E_{1,N}(n)\right|+8\pi\left|\frac{1}{1-e^{\frac{\pi\mathrm{i}}{5}}}\right|\left(n+\frac{3}{5}\right)^{-1/2}\\
 & \;\times\left(L_{1,n}\log\left(\frac{56}{55}\times\frac{8\sqrt{5}L_{1,n}}{\sqrt{3\pi}}\right)+2I_{-1}\left(L_{1,n}\right)-\left(2-\gamma-\frac{5}{55}\right)L_{1,n}\right).
\end{align*}
From this we deduce that 
\begin{align*}
 & \left|\frac{c_{1}^{(10)}(n)}{\hat{c}_{1}^{(10)}(n)}-1\right|\left|\cos\left(\frac{3\pi}{5}\left(n+\frac{8}{15}\right)\right)\right|\\
 & \leq16\sqrt{\frac{5}{3}}\left|\frac{1}{1-e^{\frac{\pi\mathrm{i}}{5}}}\right|\frac{L_{1,n}\log\left(L_{1,n}\right)+2I_{-1}(L_{1,n})+\nu(1)L_{1,n}}{I_{-1}(2L_{1,n})}
\end{align*}
with 
\begin{align*}
\nu(1) & =2490.26.
\end{align*}
For $n\geq241,$ we know $L_{1,n}\geq3.$ Apply Lemma \ref{lem:z2}.
This gives that the sequence
\[
\left\{ \frac{L_{1,n}\log\left(L_{1,n}\right)+2I_{-1}(L_{1,n})+\nu(1)L_{1,n}}{I_{-1}(2L_{1,n})}\right\} _{n\geq241}
\]
is decreasing monotonically with respect to $n.$ Employing the software
\emph{Mathematica} we can verify that 
\[
16\sqrt{\frac{5}{3}}\left|\frac{1}{1-e^{\frac{\pi\mathrm{i}}{5}}}\right|M\left(\nu(1),L_{1,241}\right)<\min_{n=0,1,\cdots,9}\left|\cos\left(\frac{3\pi}{5}\left(n+\frac{8}{15}\right)\right)\right|.
\]
This tells us that when $n\geq241,$ the coefficient $c_{1}^{(10)}(n)$
has the same sign as $\cos\left(\frac{3\pi}{5}\left(n+\frac{8}{15}\right)\right).$
The coefficients of $Q_{10}(q)$ exhibit the sign pattern $+-++--+--+.$
We utilize \emph{Mathematica} to compute the first $242$ coefficients
of $Q_{10}(q)$ and find that they also follow the same sign pattern.

When $\delta=-1,$ we let 
\[
\left(Q_{10}(q)\right)^{-1}=:\sum_{n=1}^{\infty}\bar{c}^{(10)}(n)q^{n}.
\]
Then ${\bf m}=\{1,3\},$ ${\bf n}=\{10,10\},$ and ${\bf u}=\{-1,1\}.$
Hence $L=10,$ and $\Omega=-72/5.$ Proceeding as in the proof of
Theorem \ref{thm:6}, we can get when $n\geq211,$ the coefficient
$\bar{c}^{(10)}(n)$ has the same sign as $\cos\left(\frac{\pi}{5}\left(n-\frac{6}{5}\right)\right).$
This implies that the coefficients of $Q_{10}(q)^{-1}$ exhibit the
sign pattern $++++-----+$ when $n\geq211.$ Similarly, we compute
the first $212$ coefficients of $Q_{10}^{-1}(q)$ and find that they
also follow the same sign pattern. This ends the proof of Theorem
\ref{thm:10}. \qed

\subsection{Proof of Theorem \ref{thm:12}}

Let $Q_{12}(q)$ be defined as in Conjecture \ref{conj:5}. Then ${\bf m}=\{1,5\},$
${\bf n}=\{12,12\},$ and ${\bf u}=\{1,-1\}.$ Hence $L=12,$ and
$\Omega=24.$ We compute that 
\[
\mathcal{L}_{>0}=\{(5,12),(7,12)\}.
\]
It is easy to find that the maximum of $\sqrt{\Delta(\varkappa,\kappa)}/k$
with $(\varkappa,\kappa)\in\mathcal{L}_{>0}$ and $k\equiv\kappa\;\text{(mod\;}L)$
is $\frac{\sqrt{6}}{6}$ and $\sqrt{\Delta(\varkappa,\kappa)}/k$
attains the value $\frac{\sqrt{6}}{6}$ when
\[
(\varkappa,\kappa,k)=(5,12,12),\;(7,12,12).
\]
When $(\varkappa,\kappa)=(5,12),$ 
\[
\frac{24}{\Delta(\varkappa,\kappa)}\min_{1\leq i\leq I}\left(\varUpsilon\left(\lambda_{m_{i},n_{i}}^{*}(h,k)\right)\frac{d_{i}^{2}}{n_{i}}\right)=1
\]
takes its minimum value, so when $\delta=1,$ the inequality \eqref{eq:thm1}
is satisfied and when $2\leq\delta\leq3,$ the inequality \eqref{eq:thm1}
is not satisfied.

For $\delta>0,$ we let 
\[
Q_{12}^{\delta}(q)=\sum_{n=1}^{\infty}c_{\delta}^{(12)}(n)q^{n}.
\]
 Proceeding as in the proof of Theorem \ref{thm:6}, we can get when
$n\geq326,$ and $n\equiv0,1,2,4,5,6,7,8,10,11\left(\bmod12\right),$
the coefficient $c_{1}^{(12)}(n)$ has the same sign as $\cos\left(\frac{5\pi}{6}n\right),$
that is 
\[
\begin{cases}
c_{1}^{(12)}(n)>0, & \mathrm{if}\;n\equiv0,2,5,7,10\left(\bmod12\right),\\
c_{1}^{(12)}(n)<0, & \mathrm{if}\;n\equiv1,4,6,8,11\left(\bmod12\right).
\end{cases}
\]
Recall from \cite[Theorem 2.]{AB} that if
\[
\frac{(q^{r},q^{2k-r};q^{2k})}{(q^{k-r},q^{k+r};q^{2k})}=\sum_{n=0}^{\infty}a_{n}q^{n},
\]
then $a_{kn+r(k-r+1)/2}$ is always zero. Choose $k=6,$ $r=1,$ we
can get $c_{1}^{(12)}(6n+3)=0,$ so the coefficients of $Q_{12}(q)$
exhibit the sign pattern $+-+\,0-+-+-0+-.$

When $\delta\in[-1,0),$ we let $\delta=-\delta'$ and
\[
Q_{12}(q)^{\delta}=\left(Q_{12}(q)^{-1}\right)^{\delta^{\prime}}=:\sum_{n=1}^{\infty}\bar{c}_{\delta^{\prime}}^{(12)}(n)q^{n}.
\]
Then ${\bf m}=\{1,5\},$ ${\bf n}=\{12,12\},$ and ${\bf u}=\{-1,1\}.$
Hence $L=12,$ and $\Omega=-24.$ We compute that 
\[
\mathcal{L}_{>0}=\{(1,12),(11,12)\}.
\]
After some computations, we find taht the maximum of $\sqrt{\Delta(\varkappa,\kappa)}/k$
with $(\varkappa,\kappa)\in\mathcal{L}_{>0}$ and $k\equiv\kappa\;\text{(mod\;}L)$
is $\frac{\sqrt{6}}{6}$ and $\sqrt{\Delta(\varkappa,\kappa)}/k$
attains the maximum when
\[
(\varkappa,\kappa,k)=(1,12,12),\;(11,12,12).
\]
When $(\varkappa,\kappa)=(1,12),$ we have
\[
\frac{24}{\Delta(\varkappa,\kappa)}\min_{1\leq i\leq I}\left(\varUpsilon\left(\lambda_{m_{i},n_{i}}^{*}(h,k)\right)\frac{d_{i}^{2}}{n_{i}}\right)=1
\]
takes its minimum value. So when $0<\delta^{\prime}\leq1,$ the inequality
\eqref{eq:thm1} is satisfied. Proceeding as in the proof of Theorem
\ref{thm:6}, we can get when $n\geq328,$ and $n\equiv0,1,2,3,4,6,7,8,9,10\left(\bmod12\right),$
the coefficient $\bar{c}_{1}^{(12)}(n)$ has the same sign as $\cos\left(\frac{\pi}{6}n\right),$
that is 
\[
\begin{cases}
\bar{c}_{1}^{(12)}(n)>0, & \mathrm{if}\;n\equiv0,1,2,3,4\left(\bmod12\right),\\
\bar{c}_{1}^{(12)}(n)<0, & \mathrm{if}\;n\equiv6,7,8,9,10\left(\bmod12\right).
\end{cases}
\]
Utilizing \cite[Theorem 2.]{AAR}, we find that $\bar{c}_{1}^{(12)}(6n+5)=0,$
so the coefficients of $Q_{12}^{-1}(q)$ exhibit the sign pattern
$+++++\,0-----0.$

For $0.001\leq\delta^{\prime}\leq0.499,$ proceeding as in the proof
of Theorem \ref{thm:6}, we can get that when $n\geq1283,$ the coefficient
$\bar{c}_{\delta^{\prime}}^{(12)}(n)$ has the same sign as 
\[
\tilde{c}_{\delta^{\prime}}^{(12)}(n):=\frac{\pi\sqrt{\delta}}{3}\cos\left(\frac{\pi}{6}\left(n-2\delta^{\prime}\right)\right)\left(n-\delta^{\prime}\right)^{-1/2}I_{-1}\left(\frac{\pi\sqrt{\delta}}{3}\sqrt{n-\delta^{\prime}}\right).
\]
So $Q_{12}^{\delta}(q)=\left(Q_{12}^{-1}(q)\right)^{\delta^{\prime}}$
exhibit the sign pattern $++++------++.$ Similarly, for $0.501\leq\delta^{\prime}\leq0.999,$
we can obtain that when $n\geq1277,$ the coefficient $\bar{c}_{\delta^{\prime}}^{(12)}(n)$
has the same sign as $\tilde{c}_{\delta^{\prime}}^{(12)}(n).$ Thus,
the coefficients of $Q_{12}^{\delta}(q)=\left(Q_{12}^{-1}(q)\right)^{\delta^{\prime}}$
exhibit the sign pattern $+++++------+.$

For $\delta^{\prime}=\frac{1}{2},$ we know that when $n\geq439,$
and $n\equiv0,1,2,3,5,6,7,8,9,11\left(\bmod12\right),$ the coefficient
$\bar{c}_{1/2}^{(12)}(n)$ has the same sign as $\tilde{c}_{1/2}^{(12)}(n).$
Similarly, we can deduce that  for $n\geq1859,$
\[
48\left|\frac{1}{1-e^{\frac{\pi\mathrm{i}}{6}}}\right|\widehat{M}\left(\mu(1/2),L_{\frac{1}{2},1859}\right)<\min_{n=4,10,16,22}\left|\cos\left(\frac{\pi}{12}n-\frac{1}{12}\right)+\cos\left(\frac{11\pi}{12}n+\frac{1}{12}\right)\right|.
\]
Then when $n\geq1859,$ and $n\equiv4,10,16,22\left(\bmod24\right),$
the coefficient $\bar{c}_{1/2}^{(12)}(n)$ has the same sign as $\cos\left(\frac{\pi}{12}n-\frac{1}{12}\right)+\cos\left(\frac{11\pi}{12}n+\frac{1}{12}\right).$
Combining the above results we obtain that the coefficients of $Q_{12}^{-1/2}(q)$
exhibit the sign pattern $+++++------+++++------++.$ For $n<1859,$
we utilized \emph{Mathematica} to confirm that the coefficients of
$Q_{12}^{-1/2}(q)$ exhibit the same sign pattern. The proof of Theorem
\ref{thm:12} is complete. \qed

\section*{Appendix}
\begin{center}
\vspace{-35pt}
\par\end{center}

\noindent \begin{center}
\begin{longtable}[c]{cccccc}
\caption{Ranges in which $\{c_{\delta}^{(5)}(n)\}_{0\protect\leq n\protect\leq175}$
are located for $1\protect\leq\delta\protect\leq(\sqrt{97}-5)/2$}
\tabularnewline
\hline 
$n$ & Range & $n$ & Range & $n$ & Range\tabularnewline
\hline 
\endhead
\hline 
 &  &  &  &  & \tabularnewline
\endfoot
$0$ & $[1,1]$ & $59$ & $[-1005.4,-57]$ & $118$ & $[-1524803.7,-148]$\tabularnewline
$1$ & $[-2.4,-1]$ & $60$ & $[107,11752.2]$ & $119$ & $[-133054.2,-1296]$\tabularnewline
$2$ & $[1,4.2]$ & $61$ & $[-20423.3,-119]$ & $120$ & $[2360,1911942]$\tabularnewline
$3$ & $[-3.7,0]$ & $62$ & $[83,22161.4]$ & $121$ & $[-3233455.6,-2574]$\tabularnewline
$4$ & $[-1.3,0]$ & $63$ & $[-14667.6,-8]$ & $122$ & $[1771,3415223.4]$\tabularnewline
$5$ & $[1,5.3]$ & $64$ & $[-1616.4,-79]$ & $123$ & $[-2201715.1,-180]$\tabularnewline
$6$ & $[-10,-1]$ & $65$ & $[143,19277.3]$ & $124$ & $[-189402.8,-1630]$\tabularnewline
$7$ & $[1,11.4]$ & $66$ & $[-33343.4,-157]$ & $125$ & $[2949,2753054.5]$\tabularnewline
$8$ & $[-7.6,0]$ & $67$ & $[110,36004.7]$ & $126$ & $[-4649222.4,-3208]$\tabularnewline
$9$ & $[-2,-1]$ & $68$ & $[-23724.5,-12]$ & $127$ & $[2208,4903530.7]$\tabularnewline
$10$ & $[2,13.9]$ & $69$ & $[-2531.4,-103]$ & $128$ & $[-3156784.9,-230]$\tabularnewline
$11$ & $[-28,-3]$ & $70$ & $[191,30997.3]$ & $129$ & $[-267764.6,-2024]$\tabularnewline
$12$ & $[2,33.9]$ & $71$ & $[-53509.5,-212]$ & $130$ & $[3676,3937093.6]$\tabularnewline
$13$ & $[-25.4,0]$ & $72$ & $[147,57667.2]$ & $131$ & $[-6640782.8,-4004]$\tabularnewline
$14$ & $[-5.1,-1.1]$ & $73$ & $[-37917.9,-14]$ & $132$ & $[2750,6995709]$\tabularnewline
$15$ & $[4,39.9]$ & $74$ & $[-3949.1,-139]$ & $133$ & $[-4498306.3,-279]$\tabularnewline
$16$ & $[-72.1,-4]$ & $75$ & $[253,49254.2]$ & $134$ & $[-376621.7,-2523]$\tabularnewline
$17$ & $[3,82.1]$ & $76$ & $[-84705.3,-277]$ & $135$ & $[4563,5596381.3]$\tabularnewline
$18$ & $[-58.1,-1]$ & $77$ & $[193,90960.4]$ & $136$ & $[-9427466.2,-4957]$\tabularnewline
$19$ & $[-10,-3]$ & $78$ & $[-59616,-22]$ & $137$ & $[3406,9918781]$\tabularnewline
$20$ & $[6,85.4]$ & $79$ & $[-6042.4,-180]$ & $138$ & $[-6370023.7,-355]$\tabularnewline
$21$ & $[-159.6,-7]$ & $80$ & $[332,77025.7]$ & $139$ & $[-526550.3,-3114]$\tabularnewline
$22$ & $[5,186.5]$ & $81$ & $[-132228.8,-366]$ & $140$ & $[5646,7906498]$\tabularnewline
$23$ & $[-130,0]$ & $82$ & $[254,141739]$ & $141$ & $[-13304421.4,-6139]$\tabularnewline
$24$ & $[-22,-5]$ & $83$ & $[-92702.1,-25]$ & $142$ & $[4213,13982635.5]$\tabularnewline
$25$ & $[9,188.6]$ & $84$ & $[-9197.1,-238]$ & $143$ & $[-8970096.9,-428]$\tabularnewline
$26$ & $[-340.5,-10]$ & $85$ & $[432,119201.8]$ & $144$ & $[-732670.4,-3853]$\tabularnewline
$27$ & $[7,380.5]$ & $86$ & $[-204045.2,-473]$ & $145$ & $[6959,11108984.7]$\tabularnewline
$28$ & $[-260.6,-1]$ & $87$ & $[328,218087.7]$ & $146$ & $[-18672059.6,-7553]$\tabularnewline
$29$ & $[-38.5,-7]$ & $88$ & $[-142264.8,-35]$ & $147$ & $[5181,19601702]$\tabularnewline
$30$ & $[14,372.5]$ & $89$ & $[-13792.8,-305]$ & $148$ & $[-12560953.9,-536]$\tabularnewline
$31$ & $[-671.8,-16]$ & $90$ & $[561,182144.7]$ & $149$ & $[-1014043.8,-4727]$\tabularnewline
$32$ & $[11,756.5]$ & $91$ & $[-311230.5,-616]$ & $150$ & $[8558,15523302]$\tabularnewline
$33$ & $[-519.7,-1]$ & $92$ & $[426,332090.3]$ & $151$ & $[-26065393.3,-9292]$\tabularnewline
$34$ & $[-74.8,-11]$ & $93$ & $[-216249,-43]$ & $152$ & $[6368,27335974.8]$\tabularnewline
$35$ & $[20,722.9]$ & $94$ & $[-20565.9,-397]$ & $153$ & $[-17499723.6,-649]$\tabularnewline
$36$ & $[-1281.8,-22]$ & $95$ & $[720,275698.1]$ & $154$ & $[-1397269.3,-5809]$\tabularnewline
$37$ & $[16,1421.1]$ & $96$ & $[-469996.8,-787]$ & $155$ & $[10483,21583095.3]$\tabularnewline
$38$ & $[-959.6,-2]$ & $97$ & $[545,500345.9]$ & $156$ & $[-36203388.4,-11364]$\tabularnewline
$39$ & $[-127.3,-15]$ & $98$ & $[-325094.4,-58]$ & $157$ & $[7788,37929563.3]$\tabularnewline
$40$ & $[29,1322.2]$ & $99$ & $[-30308.1,-504]$ & $158$ & $[-24257181,-804]$\tabularnewline
$41$ & $[-2349,-33]$ & $100$ & $[923,412923.5]$ & $159$ & $[-1916118.7,-7088]$\tabularnewline
$42$ & $[23,2599.4]$ & $101$ & $[-702816.7,-1012]$ & $160$ & $[12813,29859903.9]$\tabularnewline
$43$ & $[-1753.2,-2]$ & $102$ & $[698,747019.8]$ & $161$ & $[-50040614.6,-13896]$\tabularnewline
$44$ & $[-225.8,-23]$ & $103$ & $[-484583.7,-70]$ & $162$ & $[9512,52378591.6]$\tabularnewline
$45$ & $[41,2387.2]$ & $104$ & $[-44401.6,-648]$ & $163$ & $[-33467189.9,-968]$\tabularnewline
$46$ & $[-4189.7,-45]$ & $105$ & $[1175,613315]$ & $164$ & $[-2616822.8,-8658]$\tabularnewline
$47$ & $[32,4590.3]$ & $106$ & $[-1041841.4,-1281]$ & $165$ & $[15612,41121289.2]$\tabularnewline
$48$ & $[-3071,-4]$ & $107$ & $[885,1105230]$ & $166$ & $[-68848821.8,-16906]$\tabularnewline
$49$ & $[-373.3,-30]$ & $108$ & $[-715639.4,-94]$ & $167$ & $[11572,71999014.2]$\tabularnewline
$50$ & $[57,4135.5]$ & $109$ & $[-64432.1,-816]$ & $168$ & $[-45961846.9,-1193]$\tabularnewline
$51$ & $[-7251.2,-64]$ & $110$ & $[1489,902751.4]$ & $169$ & $[-3558280.9,-10511]$\tabularnewline
$52$ & $[45,7937.3]$ & $111$ & $[-1531293.9,-1628]$ & $170$ & $[18975,56374058]$\tabularnewline
$53$ & $[-5295,-4]$ & $112$ & $[1122,1622160.6]$ & $171$ & $[-94305940.4,-20555]$\tabularnewline
$54$ & $[-624.6,-43]$ & $113$ & $[-1048806.2,-113]$ & $172$ & $[14058,98537759.4]$\tabularnewline
$55$ & $[78,7065.5]$ & $114$ & $[-92960.2,-1036]$ & $173$ & $[-62850291,-1431]$\tabularnewline
$56$ & $[-12306,-86]$ & $115$ & $[1877,1318912.7]$ & $174$ & $[-4819914.1,-12766]$\tabularnewline
$57$ & $[60,13374.2]$ & $116$ & $[-2233496.2,-2045]$ & $175$ & $[23000,76958092.2]$\tabularnewline
$58$ & $[-8872.4,-7]$ & $117$ & $[1409,2362106.4]$ & $176$ & $[-128630751.1,-24886]$\tabularnewline
\end{longtable}
\par\end{center}

\noindent \begin{center}
\vspace{-40pt}%
\begin{longtable}[c]{cccccc}
\caption{Ranges in which $\{c_{\delta}^{(5)}(n)\}_{0\protect\leq n\protect\leq175}$
are located for $\alpha\protect\leq\delta\protect\leq4$}
\tabularnewline
\hline 
$n$ & Range & $n$ & Range & $n$ & Range\tabularnewline
\hline 
\endhead
\hline 
 &  &  &  &  & \tabularnewline
\endfoot
$0$ & $[1,1]$ & $59$ & $[-1005.4,-57]$ & $118$ & $[-1524803.7,-148]$\tabularnewline
$1$ & $[-2.4,-1]$ & $60$ & $[107,11752.2]$ & $119$ & $[-133054.2,-1296]$\tabularnewline
$2$ & $[1,4.2]$ & $61$ & $[-20423.3,-119]$ & $120$ & $[2360,1911942]$\tabularnewline
$3$ & $[-3.7,0]$ & $62$ & $[83,22161.4]$ & $121$ & $[-3233455.6,-2574]$\tabularnewline
$4$ & $[-1.3,0]$ & $63$ & $[-14667.6,-8]$ & $122$ & $[1771,3415223.4]$\tabularnewline
$5$ & $[1,5.3]$ & $64$ & $[-1616.4,-79]$ & $123$ & $[-2201715.1,-180]$\tabularnewline
$6$ & $[-10,-1]$ & $65$ & $[143,19277.3]$ & $124$ & $[-189402.8,-1630]$\tabularnewline
$7$ & $[1,11.4]$ & $66$ & $[-33343.4,-157]$ & $125$ & $[2949,2753054.5]$\tabularnewline
$8$ & $[-7.6,0]$ & $67$ & $[110,36004.7]$ & $126$ & $[-4649222.4,-3208]$\tabularnewline
$9$ & $[-2,-1]$ & $68$ & $[-23724.5,-12]$ & $127$ & $[2208,4903530.7]$\tabularnewline
$10$ & $[2,13.9]$ & $69$ & $[-2531.4,-103]$ & $128$ & $[-3156784.9,-230]$\tabularnewline
$11$ & $[-28,-3]$ & $70$ & $[191,30997.3]$ & $129$ & $[-267764.6,-2024]$\tabularnewline
$12$ & $[2,33.9]$ & $71$ & $[-53509.5,-212]$ & $130$ & $[3676,3937093.6]$\tabularnewline
$13$ & $[-25.4,0]$ & $72$ & $[147,57667.2]$ & $131$ & $[-6640782.8,-4004]$\tabularnewline
$14$ & $[-5.1,-1.1]$ & $73$ & $[-37917.9,-14]$ & $132$ & $[2750,6995709]$\tabularnewline
$15$ & $[4,39.9]$ & $74$ & $[-3949.1,-139]$ & $133$ & $[-4498306.3,-279]$\tabularnewline
$16$ & $[-72.1,-4]$ & $75$ & $[253,49254.2]$ & $134$ & $[-376621.7,-2523]$\tabularnewline
$17$ & $[3,82.1]$ & $76$ & $[-84705.3,-277]$ & $135$ & $[4563,5596381.3]$\tabularnewline
$18$ & $[-58.1,-1]$ & $77$ & $[193,90960.4]$ & $136$ & $[-9427466.2,-4957]$\tabularnewline
$19$ & $[-10,-3]$ & $78$ & $[-59616,-22]$ & $137$ & $[3406,9918781]$\tabularnewline
$20$ & $[6,85.4]$ & $79$ & $[-6042.4,-180]$ & $138$ & $[-6370023.7,-355]$\tabularnewline
$21$ & $[-159.6,-7]$ & $80$ & $[332,77025.7]$ & $139$ & $[-526550.3,-3114]$\tabularnewline
$22$ & $[5,186.5]$ & $81$ & $[-132228.8,-366]$ & $140$ & $[5646,7906498]$\tabularnewline
$23$ & $[-130,0]$ & $82$ & $[254,141739]$ & $141$ & $[-13304421.4,-6139]$\tabularnewline
$24$ & $[-22,-5]$ & $83$ & $[-92702.1,-25]$ & $142$ & $[4213,13982635.5]$\tabularnewline
$25$ & $[9,188.6]$ & $84$ & $[-9197.1,-238]$ & $143$ & $[-8970096.9,-428]$\tabularnewline
$26$ & $[-340.5,-10]$ & $85$ & $[432,119201.8]$ & $144$ & $[-732670.4,-3853]$\tabularnewline
$27$ & $[7,380.5]$ & $86$ & $[-204045.2,-473]$ & $145$ & $[6959,11108984.7]$\tabularnewline
$28$ & $[-260.6,-1]$ & $87$ & $[328,218087.7]$ & $146$ & $[-18672059.6,-7553]$\tabularnewline
$29$ & $[-38.5,-7]$ & $88$ & $[-142264.8,-35]$ & $147$ & $[5181,19601702]$\tabularnewline
$30$ & $[14,372.5]$ & $89$ & $[-13792.8,-305]$ & $148$ & $[-12560953.9,-536]$\tabularnewline
$31$ & $[-671.8,-16]$ & $90$ & $[561,182144.7]$ & $149$ & $[-1014043.8,-4727]$\tabularnewline
$32$ & $[11,756.5]$ & $91$ & $[-311230.5,-616]$ & $150$ & $[8558,15523302]$\tabularnewline
$33$ & $[-519.7,-1]$ & $92$ & $[426,332090.3]$ & $151$ & $[-26065393.3,-9292]$\tabularnewline
$34$ & $[-74.8,-11]$ & $93$ & $[-216249,-43]$ & $152$ & $[6368,27335974.8]$\tabularnewline
$35$ & $[20,722.9]$ & $94$ & $[-20565.9,-397]$ & $153$ & $[-17499723.6,-649]$\tabularnewline
$36$ & $[-1281.8,-22]$ & $95$ & $[720,275698.1]$ & $154$ & $[-1397269.3,-5809]$\tabularnewline
$37$ & $[16,1421.1]$ & $96$ & $[-469996.8,-787]$ & $155$ & $[10483,21583095.3]$\tabularnewline
$38$ & $[-959.6,-2]$ & $97$ & $[545,500345.9]$ & $156$ & $[-36203388.4,-11364]$\tabularnewline
$39$ & $[-127.3,-15]$ & $98$ & $[-325094.4,-58]$ & $157$ & $[7788,37929563.3]$\tabularnewline
$40$ & $[29,1322.2]$ & $99$ & $[-30308.1,-504]$ & $158$ & $[-24257181,-804]$\tabularnewline
$41$ & $[-2349,-33]$ & $100$ & $[923,412923.5]$ & $159$ & $[-1916118.7,-7088]$\tabularnewline
$42$ & $[23,2599.4]$ & $101$ & $[-702816.7,-1012]$ & $160$ & $[12813,29859903.9]$\tabularnewline
$43$ & $[-1753.2,-2]$ & $102$ & $[698,747019.8]$ & $161$ & $[-50040614.6,-13896]$\tabularnewline
$44$ & $[-225.8,-23]$ & $103$ & $[-484583.7,-70]$ & $162$ & $[9512,52378591.6]$\tabularnewline
$45$ & $[41,2387.2]$ & $104$ & $[-44401.6,-648]$ & $163$ & $[-33467189.9,-968]$\tabularnewline
$46$ & $[-4189.7,-45]$ & $105$ & $[1175,613315]$ & $164$ & $[-2616822.8,-8658]$\tabularnewline
$47$ & $[32,4590.3]$ & $106$ & $[-1041841.4,-1281]$ & $165$ & $[15612,41121289.2]$\tabularnewline
$48$ & $[-3071,-4]$ & $107$ & $[885,1105230]$ & $166$ & $[-68848821.8,-16906]$\tabularnewline
$49$ & $[-373.3,-30]$ & $108$ & $[-715639.4,-94]$ & $167$ & $[11572,71999014.2]$\tabularnewline
$50$ & $[57,4135.5]$ & $109$ & $[-64432.1,-816]$ & $168$ & $[-45961846.9,-1193]$\tabularnewline
$51$ & $[-7251.2,-64]$ & $110$ & $[1489,902751.4]$ & $169$ & $[-3558280.9,-10511]$\tabularnewline
$52$ & $[45,7937.3]$ & $111$ & $[-1531293.9,-1628]$ & $170$ & $[18975,56374058]$\tabularnewline
$53$ & $[-5295,-4]$ & $112$ & $[1122,1622160.6]$ & $171$ & $[-94305940.4,-20555]$\tabularnewline
$54$ & $[-624.6,-43]$ & $113$ & $[-1048806.2,-113]$ & $172$ & $[14058,98537759.4]$\tabularnewline
$55$ & $[78,7065.5]$ & $114$ & $[-92960.2,-1036]$ & $173$ & $[-62850291,-1431]$\tabularnewline
$56$ & $[-12306,-86]$ & $115$ & $[1877,1318912.7]$ & $174$ & $[-4819914.1,-12766]$\tabularnewline
$57$ & $[60,13374.2]$ & $116$ & $[-2233496.2,-2045]$ & $175$ & $[23000,76958092.2]$\tabularnewline
$58$ & $[-8872.4,-7]$ & $117$ & $[1409,2362106.4]$ & $176$ & $[-128630751.1,-24886]$\tabularnewline
\end{longtable}
\par\end{center}

\noindent \begin{center}
\vspace{-40pt}%
\begin{longtable}[c]{cccccc}
\caption{Ranges in which $\{c_{\delta}^{(5)}(n)\}_{0\protect\leq n\protect\leq142}$
are located for $-3\protect\leq\delta\protect\leq-2$}
\tabularnewline
\hline 
$n$ & Range & $n$ & Range & $n$ & Range\tabularnewline
\hline 
\endhead
\hline 
 &  &  &  &  & \tabularnewline
\endfoot
$0$ & $[1,1]$ & $48$ & $[-12636,-1442]$ & $96$ & $[87308,3134741]$\tabularnewline
$1$ & $[2,3]$ & $49$ & $[-20877,-1090]$ & $97$ & $[6072,1900626]$\tabularnewline
$2$ & $[1,3]$ & $50$ & $[1016,2629.5]$ & $98$ & $[-2477301,-97066]$\tabularnewline
$3$ & $[-2.2,-2]$ & $51$ & $[2024,27874]$ & $99$ & $[-3956904,-71548]$\tabularnewline
$4$ & $[-6,-2]$ & $52$ & $[147,17427]$ & $100$ & $[64697,389287.7]$\tabularnewline
$5$ & $[0,2]$ & $53$ & $[-23364,-2350]$ & $101$ & $[125666,4937112]$\tabularnewline
$6$ & $[5,12]$ & $54$ & $[-38444,-1775]$ & $102$ & $[8773,2987237]$\tabularnewline
$7$ & $[0,9]$ & $55$ & $[1632,4639.8]$ & $103$ & $[-3884727,-139160]$\tabularnewline
$8$ & $[-15,-8]$ & $56$ & $[3244,50718]$ & $104$ & $[-6192612,-102453]$\tabularnewline
$9$ & $[-28,-6]$ & $57$ & $[226,31532]$ & $105$ & $[92402,599630.3]$\tabularnewline
$10$ & $[3,7.8]$ & $58$ & $[-42174,-3756]$ & $106$ & $[179256,7694952]$\tabularnewline
$11$ & $[14,48]$ & $59$ & $[-69000,-2818]$ & $107$ & $[12444,4646544]$\tabularnewline
$12$ & $[1,33]$ & $60$ & $[2605,8111.2]$ & $108$ & $[-6031739,-197999]$\tabularnewline
$13$ & $[-48,-18]$ & $61$ & $[5148,90312]$ & $109$ & $[-9596208,-145508]$\tabularnewline
$14$ & $[-87,-15]$ & $62$ & $[372,55929]$ & $110$ & $[131145,916256.7]$\tabularnewline
$15$ & $[7,17.7]$ & $63$ & $[-74344,-5898]$ & $111$ & $[253984,11880316]$\tabularnewline
$16$ & $[30,135]$ & $64$ & $[-121209,-4423]$ & $112$ & $[17663,7160853]$\tabularnewline
$17$ & $[2,90]$ & $65$ & $[4048,13777.2]$ & $113$ & $[-9277578,-279622]$\tabularnewline
$18$ & $[-134,-40]$ & $66$ & $[7990,157242]$ & $114$ & $[-14734337,-205256]$\tabularnewline
$19$ & $[-234,-32]$ & $67$ & $[558,96966]$ & $115$ & $[184608,1387431.9]$\tabularnewline
$20$ & $[21,44.3]$ & $68$ & $[-128580,-9130]$ & $116$ & $[357096,18176196]$\tabularnewline
$21$ & $[66,356]$ & $69$ & $[-208748,-6812]$ & $117$ & $[24728,10936480]$\tabularnewline
$22$ & $[6,237]$ & $70$ & $[6246,23175.8]$ & $118$ & $[-14146467,-392217]$\tabularnewline
$23$ & $[-330,-82]$ & $71$ & $[12278,268992]$ & $119$ & $[-22428240,-287458]$\tabularnewline
$24$ & $[-575,-65]$ & $72$ & $[872,165328]$ & $120$ & $[258363,2085718.1]$\tabularnewline
$25$ & $[42,92.8]$ & $73$ & $[-218262,-13918]$ & $121$ & $[499038,27577116]$\tabularnewline
$26$ & $[125,831]$ & $74$ & $[-353265,-10373]$ & $122$ & $[34608,16566303]$\tabularnewline
$27$ & $[8,540]$ & $75$ & $[9454,38203.7]$ & $123$ & $[-21392580,-546562]$\tabularnewline
$28$ & $[-762,-157]$ & $76$ & $[18558,452052]$ & $124$ & $[-33863613,-400161]$\tabularnewline
$29$ & $[-1296,-120]$ & $77$ & $[1298,276912]$ & $125$ & $[359028,3110699.3]$\tabularnewline
$30$ & $[107,203.1]$ & $78$ & $[-364679,-20971]$ & $126$ & $[692733,41506438]$\tabularnewline
$31$ & $[238,1848]$ & $79$ & $[-588297,-15576]$ & $127$ & $[47872,24895416]$\tabularnewline
$32$ & $[19,1191]$ & $80$ & $[14203,62325.5]$ & $128$ & $[-32101560,-757083]$\tabularnewline
$33$ & $[-1633,-286]$ & $81$ & $[27792,748552]$ & $129$ & $[-50738224,-553542]$\tabularnewline
$34$ & $[-2769,-222]$ & $82$ & $[1960,457272]$ & $130$ & $[496271,4608858.7]$\tabularnewline
$35$ & $[206,395.5]$ & $83$ & $[-600162,-31224]$ & $131$ & $[956318,62008230]$\tabularnewline
$36$ & $[419,3842]$ & $84$ & $[-965643,-23158]$ & $132$ & $[66127,37138891]$\tabularnewline
$37$ & $[28,2448]$ & $85$ & $[21022,100091.3]$ & $133$ & $[-47817984,-1042560]$\tabularnewline
$38$ & $[-3366,-507]$ & $86$ & $[41073,1221666]$ & $134$ & $[-75473118,-761523]$\tabularnewline
$39$ & $[-5634,-386]$ & $87$ & $[2862,744208]$ & $135$ & $[681750,6781172.8]$\tabularnewline
$40$ & $[366,778.2]$ & $88$ & $[-974622,-46008]$ & $136$ & $[1312430,91978176]$\tabularnewline
$41$ & $[732,7722]$ & $89$ & $[-1563840,-34030]$ & $137$ & $[90530,55012824]$\tabularnewline
$42$ & $[55,4889]$ & $90$ & $[30891,159226.6]$ & $138$ & $[-70738170,-1427957]$\tabularnewline
$43$ & $[-6624,-864]$ & $91$ & $[60208,1968912]$ & $139$ & $[-111497472,-1041816]$\tabularnewline
$44$ & $[-11028,-659]$ & $92$ & $[4225,1196562]$ & $140$ & $[931969,9917198.9]$\tabularnewline
$45$ & $[610,1435.5]$ & $93$ & $[-1562712,-67118]$ & $141$ & $[1792120,135524112]$\tabularnewline
$46$ & $[1224,14871]$ & $94$ & $[-2501817,-49587]$ & $142$ & $[123636,80953146]$\tabularnewline
$47$ & $[86,9342]$ & $95$ & $[44856,250048.8]$ & $143$ & $[-103955994,-1945616]$\tabularnewline
\end{longtable}
\par\end{center}

\noindent \begin{center}
\vspace{-40pt}%
\begin{longtable}[c]{cccccc}
\caption{Ranges in which $\{c_{\delta_{2}}^{(6)}(n)\}_{0\protect\leq n<57}$
are located for $3\protect\leq\delta_{2}\protect\leq4$}
\tabularnewline
\hline 
$n$ & Range & $n$ & Range & $n$ & Range\tabularnewline
\hline 
\endhead
\hline 
 &  &  &  &  & \tabularnewline
\endfoot
$0$ & $[1,1]$ & $20$ & $[111,449]$ & $40$ & $[1260,8916]$\tabularnewline
$1$ & $[3,4]$ & $21$ & $[91,420]$ & $41$ & $[1641,10708]$\tabularnewline
$2$ & $[3,6]$ & $22$ & $[84,396]$ & $42$ & $[2287,14378]$\tabularnewline
$3$ & $[1,4]$ & $23$ & $[123,508]$ & $43$ & $[2799,18320]$\tabularnewline
$4$ & $[0,1]$ & $24$ & $[208,809]$ & $44$ & $[2886,20557]$\tabularnewline
$5$ & $[3,4]$ & $25$ & $[279,1160]$ & $45$ & $[2691,20896]$\tabularnewline
$6$ & $[9,16]$ & $26$ & $[282,1332]$ & $46$ & $[2724,21576]$\tabularnewline
$7$ & $[12,28]$ & $27$ & $[234,1272]$ & $47$ & $[3405,25440]$\tabularnewline
$8$ & $[12,32]$ & $28$ & $[222,1225]$ & $48$ & $[4582,32970]$\tabularnewline
$9$ & $[9,28]$ & $29$ & $[321,1548]$ & $49$ & $[5556,41160]$\tabularnewline
$10$ & $[6,22]$ & $30$ & $[495,2300]$ & $50$ & $[5754,46086]$\tabularnewline
$11$ & $[12,32]$ & $31$ & $[630,3112]$ & $51$ & $[5429,47268]$\tabularnewline
$12$ & $[27,68]$ & $32$ & $[642,3525]$ & $52$ & $[5550,49194]$\tabularnewline
$13$ & $[42,116]$ & $33$ & $[568,3472]$ & $53$ & $[6834,57428]$\tabularnewline
$14$ & $[42,140]$ & $34$ & $[558,3476]$ & $54$ & $[8922,72612]$\tabularnewline
$15$ & $[28,120]$ & $35$ & $[741,4236]$ & $55$ & $[10611,88828]$\tabularnewline
$16$ & $[24,100]$ & $36$ & $[1082,5908]$ & $56$ & $[11031,98973]$\tabularnewline
$17$ & $[45,144]$ & $37$ & $[1368,7772]$ & $57$ & $[10632,102520]$\tabularnewline
$18$ & $[82,262]$ & $38$ & $[1404,8792]$ & ~ & ~\tabularnewline
$19$ & $[111,392]$ & $39$ & $[1262,8784]$ &  & \tabularnewline
\end{longtable}
\par\end{center}

\noindent \begin{center}
\vspace{-40pt}%
\begin{longtable}[c]{cccccc}
\caption{Ranges in which $\{\bar{c}_{\delta^{\prime}}^{(8)}(n)\}_{0\protect\leq n<479}$
are located for $\delta'\in[(\sqrt{73}-7)/2,0.99]$}
\tabularnewline
\hline 
$n$ & Range & $n$ & Range & $n$ & Range\tabularnewline
\hline 
\endhead
\hline 
 &  &  &  &  & \tabularnewline
\endfoot
$0$ & $[1,1]$ & $160$ & $[1586.2,7476.4]$ & $320$ & $[158350.8,1471930.2]$\tabularnewline
$1$ & $[0.8,1]$ & $161$ & $[1958.6,10890.9]$ & $321$ & $[194359.1,2125074.3]$\tabularnewline
$2$ & $[0.7,1]$ & $162$ & $[1178.9,7958]$ & $322$ & $[115436.2,1534146.1]$\tabularnewline
$3$ & $[-0.1,0]$ & $163$ & $[-481.6,-92.5]$ & $323$ & $[-68661.5,-17670.2]$\tabularnewline
$4$ & $[0,0]$ & $164$ & $[-8716.5,-1808.5]$ & $324$ & $[-1649560.2,-174883.5]$\tabularnewline
$5$ & $[-1,-0.7]$ & $165$ & $[-12720.3,-2241.3]$ & $325$ & $[-2381452.2,-214697.9]$\tabularnewline
$6$ & $[-1,-0.6]$ & $166$ & $[-9280.2,-1345]$ & $326$ & $[-1718858.3,-127475.8]$\tabularnewline
$7$ & $[0,0.2]$ & $167$ & $[107.9,556.1]$ & $327$ & $[19795.5,76518.4]$\tabularnewline
$8$ & $[0.7,1]$ & $168$ & $[2064.1,10160.5]$ & $328$ & $[193068.5,1847580.9]$\tabularnewline
$9$ & $[1.3,2]$ & $169$ & $[2562.9,14834.5]$ & $329$ & $[237027.4,2667026.1]$\tabularnewline
$10$ & $[0.4,1]$ & $170$ & $[1528.2,10796.6]$ & $330$ & $[140649.9,1924452.3]$\tabularnewline
$11$ & $[-0.3,0]$ & $171$ & $[-641.3,-125.7]$ & $331$ & $[-85226.1,-22162.4]$\tabularnewline
$12$ & $[-2,-1.3]$ & $172$ & $[-11855.2,-2364.9]$ & $332$ & $[-2068395.7,-213130.2]$\tabularnewline
$13$ & $[-2.9,-1.8]$ & $173$ & $[-17267.6,-2924.2]$ & $333$ & $[-2984961.3,-261538]$\tabularnewline
$14$ & $[-1.9,-0.8]$ & $174$ & $[-12574.7,-1747.3]$ & $334$ & $[-2153627,-155195.4]$\tabularnewline
$15$ & $[0,0.3]$ & $175$ & $[146.3,738.8]$ & $335$ & $[24796.6,94870.2]$\tabularnewline
$16$ & $[1.9,2.9]$ & $176$ & $[2700.9,13794.5]$ & $336$ & $[235108,2313977]$\tabularnewline
$17$ & $[2.1,3.9]$ & $177$ & $[3332.8,20071.1]$ & $337$ & $[288432.6,3338725.2]$\tabularnewline
$18$ & $[1.9,3.9]$ & $178$ & $[2000.2,14633.6]$ & $338$ & $[171204.7,2408742.4]$\tabularnewline
$19$ & $[-0.5,0]$ & $179$ & $[-849.7,-169.9]$ & $339$ & $[-105547.6,-27726.8]$\tabularnewline
$20$ & $[-3.9,-2.5]$ & $180$ & $[-15991.1,-3066.8]$ & $340$ & $[-2586760.4,-259106.5]$\tabularnewline
$21$ & $[-5.8,-2.9]$ & $181$ & $[-23296.9,-3795.4]$ & $341$ & $[-3732113.8,-317918.3]$\tabularnewline
$22$ & $[-4.8,-2.2]$ & $182$ & $[-16967,-2273.1]$ & $342$ & $[-2692036.4,-188657.7]$\tabularnewline
$23$ & $[0.1,0.7]$ & $183$ & $[197,976.1]$ & $343$ & $[30984.5,117363]$\tabularnewline
$24$ & $[2.8,4.9]$ & $184$ & $[3481.5,18526.2]$ & $344$ & $[285450.3,2890144.9]$\tabularnewline
$25$ & $[4.3,8.7]$ & $185$ & $[4314.1,26996.5]$ & $345$ & $[350238,4169350.8]$\tabularnewline
$26$ & $[2.1,5.8]$ & $186$ & $[2572.5,19627.9]$ & $346$ & $[207728.5,3006694.5]$\tabularnewline
$27$ & $[-0.9,-0.1]$ & $187$ & $[-1120.1,-228.2]$ & $347$ & $[-130430,-34604.1]$\tabularnewline
$28$ & $[-7.8,-4.2]$ & $188$ & $[-21477.9,-3966.1]$ & $348$ & $[-3227591,-314427]$\tabularnewline
$29$ & $[-11.6,-5.4]$ & $189$ & $[-31242,-4898.9]$ & $349$ & $[-4655063.6,-385651]$\tabularnewline
$30$ & $[-8.6,-3.2]$ & $190$ & $[-22718.8,-2923.9]$ & $350$ & $[-3356599,-228731.2]$\tabularnewline
$31$ & $[0.1,1.2]$ & $191$ & $[263.8,1283.4]$ & $351$ & $[38624,144875.6]$\tabularnewline
$32$ & $[6.3,11.7]$ & $192$ & $[4506.4,24842.7]$ & $352$ & $[346122.4,3602121.2]$\tabularnewline
$33$ & $[6.6,15.4]$ & $193$ & $[5557.4,36106.9]$ & $353$ & $[424423.4,5194314.1]$\tabularnewline
$34$ & $[4.6,12.5]$ & $194$ & $[3328.4,26279.6]$ & $354$ & $[251785.2,3745202.8]$\tabularnewline
$35$ & $[-1.4,-0.1]$ & $195$ & $[-1469,-304.7]$ & $355$ & $[-160837.7,-43086]$\tabularnewline
$36$ & $[-13.6,-6.8]$ & $196$ & $[-28643.3,-5092.4]$ & $356$ & $[-4017358.3,-380698.6]$\tabularnewline
$37$ & $[-21.2,-8.9]$ & $197$ & $[-41669.1,-6295.3]$ & $357$ & $[-5792720.5,-466865.1]$\tabularnewline
$38$ & $[-16.3,-5.7]$ & $198$ & $[-30304.3,-3765.5]$ & $358$ & $[-4175925.8,-276897.2]$\tabularnewline
$39$ & $[0.2,1.7]$ & $199$ & $[351.4,1679.9]$ & $359$ & $[48036,178466.1]$\tabularnewline
$40$ & $[8.2,17.5]$ & $200$ & $[5759.9,33014.9]$ & $360$ & $[418574.1,4478148.9]$\tabularnewline
$41$ & $[11.3,27.9]$ & $201$ & $[7124,48028.5]$ & $361$ & $[513306.7,6456433.3]$\tabularnewline
$42$ & $[6.4,20]$ & $202$ & $[4245.6,34880.4]$ & $362$ & $[304309,4653404.7]$\tabularnewline
$43$ & $[-2.3,-0.3]$ & $203$ & $[-1918.4,-404.8]$ & $363$ & $[-197927.1,-53524.8]$\tabularnewline
$44$ & $[-25.2,-11.7]$ & $204$ & $[-38055.5,-6527]$ & $364$ & $[-4989496.8,-460130.6]$\tabularnewline
$45$ & $[-36.5,-14.1]$ & $205$ & $[-55290.4,-8054.5]$ & $365$ & $[-7192176,-564088.8]$\tabularnewline
$46$ & $[-26.7,-8.2]$ & $206$ & $[-40158.9,-4803.6]$ & $366$ & $[-5183133.9,-334411.1]$\tabularnewline
$47$ & $[0.3,2.9]$ & $207$ & $[465.7,2188.6]$ & $367$ & $[59608,219401.5]$\tabularnewline
$48$ & $[14.8,32.9]$ & $208$ & $[7377.5,43779.2]$ & $368$ & $[505523.8,5555946]$\tabularnewline
$49$ & $[18.1,48]$ & $209$ & $[9094.4,63566.5]$ & $369$ & $[619604.8,8007393]$\tabularnewline
$50$ & $[11.8,37.2]$ & $210$ & $[5438.5,46198.8]$ & $370$ & $[367381.2,5770213]$\tabularnewline
$51$ & $[-3.5,-0.4]$ & $211$ & $[-2494.4,-535.1]$ & $371$ & $[-243085.4,-66346.4]$\tabularnewline
$52$ & $[-40.6,-17.4]$ & $212$ & $[-50249.4,-8311.4]$ & $372$ & $[-6182815,-554985.5]$\tabularnewline
$53$ & $[-61.3,-22.2]$ & $213$ & $[-72996.6,-10258.7]$ & $373$ & $[-8910225.9,-680273.4]$\tabularnewline
$54$ & $[-46.6,-14.1]$ & $214$ & $[-53015.4,-6127.2]$ & $374$ & $[-6419775.9,-403269.6]$\tabularnewline
$55$ & $[0.5,4.4]$ & $215$ & $[614.1,2839.6]$ & $375$ & $[73807.2,269196.4]$\tabularnewline
$56$ & $[20.8,51.1]$ & $216$ & $[9360.8,57634.8]$ & $376$ & $[609056.9,6877012.3]$\tabularnewline
$57$ & $[27.3,78.4]$ & $217$ & $[11562.6,83730.9]$ & $377$ & $[746531.8,9909577.2]$\tabularnewline
$58$ & $[15.9,56.9]$ & $218$ & $[6886.9,60747.8]$ & $378$ & $[442388.3,7138466.1]$\tabularnewline
$59$ & $[-5.6,-0.7]$ & $219$ & $[-3229.5,-704]$ & $379$ & $[-297970.9,-82064]$\tabularnewline
$60$ & $[-67.4,-26.9]$ & $220$ & $[-66098.3,-10557.1]$ & $380$ & $[-7645766.1,-668267]$\tabularnewline
$61$ & $[-100.3,-34.1]$ & $221$ & $[-95928.7,-13016.8]$ & $381$ & $[-11015274.8,-818877.5]$\tabularnewline
$62$ & $[-73.9,-20.2]$ & $222$ & $[-69602.1,-7758.4]$ & $382$ & $[-7934123,-485243.1]$\tabularnewline
$63$ & $[0.9,6.9]$ & $223$ & $[806.2,3669.7]$ & $383$ & $[91196.9,329665.3]$\tabularnewline
$64$ & $[33.8,86.6]$ & $224$ & $[11887.3,75688.9]$ & $384$ & $[732842.4,8495763.1]$\tabularnewline
$65$ & $[41.2,125.9]$ & $225$ & $[14640.8,109784.5]$ & $385$ & $[897840.2,12238059.4]$\tabularnewline
$66$ & $[26.1,95.6]$ & $226$ & $[8742.4,79683.3]$ & $386$ & $[532097.2,8814202.7]$\tabularnewline
$67$ & $[-8.5,-1.1]$ & $227$ & $[-4165.4,-922.1]$ & $387$ & $[-364562.7,-101294.7]$\tabularnewline
$68$ & $[-105.6,-39.4]$ & $228$ & $[-86494.3,-13342.2]$ & $388$ & $[-9434777,-803115.5]$\tabularnewline
$69$ & $[-158.1,-50]$ & $229$ & $[-125508.7,-16454.1]$ & $389$ & $[-13589672.5,-983977.3]$\tabularnewline
$70$ & $[-118.1,-30.9]$ & $230$ & $[-91046.8,-9815.6]$ & $390$ & $[-9786243.6,-583041.4]$\tabularnewline
$71$ & $[1.4,10.3]$ & $231$ & $[1053.5,4723.8]$ & $391$ & $[112454,402971.1]$\tabularnewline
$72$ & $[47.4,131.4]$ & $232$ & $[14972.6,98778.3]$ & $392$ & $[879826.6,10472686]$\tabularnewline
$73$ & $[60.9,197.9]$ & $233$ & $[18472,143326.9]$ & $393$ & $[1077915.2,15083042.6]$\tabularnewline
$74$ & $[35.7,144.3]$ & $234$ & $[10996.9,103889.1]$ & $394$ & $[638499.6,10859806.9]$\tabularnewline
$75$ & $[-12.6,-1.7]$ & $235$ & $[-5352.4,-1202.5]$ & $395$ & $[-445224.4,-124780.7]$\tabularnewline
$76$ & $[-166.7,-58.7]$ & $236$ & $[-112781.4,-16820.7]$ & $396$ & $[-11619752.9,-963644.4]$\tabularnewline
$77$ & $[-245.1,-73.1]$ & $237$ & $[-163509.9,-20720.2]$ & $397$ & $[-16732296.7,-1180325.3]$\tabularnewline
$78$ & $[-180.9,-44]$ & $238$ & $[-118514.3,-12339.5]$ & $398$ & $[-12046034.2,-699134.9]$\tabularnewline
$79$ & $[2.2,15.3]$ & $239$ & $[1371.2,6058.9]$ & $399$ & $[138391.1,491690.9]$\tabularnewline
$80$ & $[71.5,207.6]$ & $240$ & $[18866.2,128589.4]$ & $400$ & $[1054921.2,12885807.3]$\tabularnewline
$81$ & $[88,303.7]$ & $241$ & $[23224.3,186351.2]$ & $401$ & $[1291913.7,18552863.7]$\tabularnewline
$82$ & $[54.5,226.9]$ & $242$ & $[13851.3,135103.1]$ & $402$ & $[765299.3,13355659.7]$\tabularnewline
$83$ & $[-18.5,-2.7]$ & $243$ & $[-6852.9,-1561.9]$ & $403$ & $[-542770,-153412.2]$\tabularnewline
$84$ & $[-252.3,-83.9]$ & $244$ & $[-146370.2,-21107.4]$ & $404$ & $[-14282175.1,-1154163.9]$\tabularnewline
$85$ & $[-374.4,-105.7]$ & $245$ & $[-212166.7,-26005.4]$ & $405$ & $[-20561698.6,-1413469.6]$\tabularnewline
$86$ & $[-277.4,-64.5]$ & $246$ & $[-153752.4,-15499]$ & $406$ & $[-14799744.6,-837166.6]$\tabularnewline
$87$ & $[3.3,22.2]$ & $247$ & $[1777.5,7744.5]$ & $407$ & $[169982.8,598895.6]$\tabularnewline
$88$ & $[98.6,307.2]$ & $248$ & $[23610.9,166503.7]$ & $408$ & $[1262301.4,15822757.2]$\tabularnewline
$89$ & $[125.5,458.2]$ & $249$ & $[29096.9,241331]$ & $409$ & $[1545839.2,22777317.3]$\tabularnewline
$90$ & $[74.5,335.1]$ & $250$ & $[17311.3,174768.7]$ & $410$ & $[915315.8,16391969.9]$\tabularnewline
$91$ & $[-26.7,-4]$ & $251$ & $[-8744.3,-2020.8]$ & $411$ & $[-660544.3,-188255.2]$\tabularnewline
$92$ & $[-381.3,-120]$ & $252$ & $[-189335.6,-26427.4]$ & $412$ & $[-17522220.7,-1380239.1]$\tabularnewline
$93$ & $[-559.7,-149.1]$ & $253$ & $[-274250.3,-32531.9]$ & $413$ & $[-25219925.3,-1689914.4]$\tabularnewline
$94$ & $[-410.7,-89]$ & $254$ & $[-198598.1,-19360.3]$ & $414$ & $[-18148020.3,-1000584.5]$\tabularnewline
$95$ & $[4.8,31.8]$ & $255$ & $[2295.4,9865.4]$ & $415$ & $[208395.7,728233.8]$\tabularnewline
$96$ & $[143.8,466.6]$ & $256$ & $[29540,215041]$ & $416$ & $[1508517.7,19394877.6]$\tabularnewline
$97$ & $[177.3,681.9]$ & $257$ & $[36339.8,311365.3]$ & $417$ & $[1846688.6,27911693.9]$\tabularnewline
$98$ & $[108.8,506]$ & $258$ & $[21652.2,225510.7]$ & $418$ & $[1093469.5,20083360.3]$\tabularnewline
$99$ & $[-38,-5.9]$ & $259$ & $[-11121.6,-2604.9]$ & $419$ & $[-802525.7,-230585.7]$\tabularnewline
$100$ & $[-559.2,-166.5]$ & $260$ & $[-243894.9,-32952.8]$ & $420$ & $[-21456945,-1647815.2]$\tabularnewline
$101$ & $[-826.4,-209.1]$ & $261$ & $[-353191.8,-40564]$ & $421$ & $[-30876812.5,-2017225.6]$\tabularnewline
$102$ & $[-610.2,-127.4]$ & $262$ & $[-255704.3,-24153.1]$ & $422$ & $[-22214062.5,-1194271.5]$\tabularnewline
$103$ & $[7.1,45.3]$ & $263$ & $[2953.5,12527.5]$ & $423$ & $[255024.1,884035.8]$\tabularnewline
$104$ & $[195.6,675.4]$ & $264$ & $[36749.2,276440.1]$ & $424$ & $[1799369.7,23728109.5]$\tabularnewline
$105$ & $[246.1,998.8]$ & $265$ & $[45247.1,400288.2]$ & $425$ & $[2202648.2,34141657.9]$\tabularnewline
$106$ & $[146.4,729.8]$ & $266$ & $[26904.3,289644.2]$ & $426$ & $[1303742.5,24559497.8]$\tabularnewline
$107$ & $[-53.7,-8.6]$ & $267$ & $[-14100.3,-3345.8]$ & $427$ & $[-973432.1,-281927.8]$\tabularnewline
$108$ & $[-821.8,-233.5]$ & $268$ & $[-313192.6,-40996.5]$ & $428$ & $[-26229176.7,-1964392.8]$\tabularnewline
$109$ & $[-1204.9,-290]$ & $269$ & $[-453265.4,-50429.6]$ & $429$ & $[-37735131.2,-2404203.1]$\tabularnewline
$110$ & $[-883.2,-173.6]$ & $270$ & $[-327956.7,-29993.4]$ & $430$ & $[-27141794,-1422970]$\tabularnewline
$111$ & $[10.4,63.6]$ & $271$ & $[3787.1,15859.1]$ & $431$ & $[311533.3,1071439.9]$\tabularnewline
$112$ & $[274.7,989.7]$ & $272$ & $[45685,354463.2]$ & $432$ & $[2143632.9,28980695.9]$\tabularnewline
$113$ & $[339.7,1447.1]$ & $273$ & $[56166.2,512827.2]$ & $433$ & $[2623222.9,41688735.9]$\tabularnewline
$114$ & $[207.1,1068.1]$ & $274$ & $[33433.2,371076.7]$ & $434$ & $[1552655.7,29983077]$\tabularnewline
$115$ & $[-75.3,-12.5]$ & $275$ & $[-17823.4,-4283]$ & $435$ & $[-1178852.1,-344099.3]$\tabularnewline
$116$ & $[-1179.9,-318.6]$ & $276$ & $[-400684.4,-50822.8]$ & $436$ & $[-32005923.4,-2338060.5]$\tabularnewline
$117$ & $[-1734,-397.1]$ & $277$ & $[-579742.1,-62514]$ & $437$ & $[-46036807.1,-2861121.1]$\tabularnewline
$118$ & $[-1274.7,-240.5]$ & $278$ & $[-419360.6,-37192.4]$ & $438$ & $[-33106393.5,-1693242.8]$\tabularnewline
$119$ & $[14.9,88.8]$ & $279$ & $[4839.8,20016.9]$ & $439$ & $[379908,1296529.6]$\tabularnewline
$120$ & $[370.1,1407.5]$ & $280$ & $[56515.4,452638.9]$ & $440$ & $[2549305,35332552.2]$\tabularnewline
$121$ & $[464.3,2073.8]$ & $281$ & $[69526.9,654849.5]$ & $441$ & $[3119456.7,50816903.6]$\tabularnewline
$122$ & $[276.5,1514.1]$ & $282$ & $[41320.5,473481.6]$ & $442$ & $[1845741,36539182]$\tabularnewline
$123$ & $[-104.5,-17.8]$ & $283$ & $[-22464.8,-5464.7]$ & $443$ & $[-1425405.5,-419266.7]$\tabularnewline
$124$ & $[-1687.8,-434.8]$ & $284$ & $[-511105.4,-62862.9]$ & $444$ & $[-38989884.7,-2778942.2]$\tabularnewline
$125$ & $[-2471.1,-540.1]$ & $285$ & $[-739106.2,-77274.6]$ & $445$ & $[-56069935.7,-3399896.4]$\tabularnewline
$126$ & $[-1808.4,-323.5]$ & $286$ & $[-534349.7,-45929.9]$ & $446$ & $[-40312530.3,-2011568.1]$\tabularnewline
$127$ & $[21.2,122.8]$ & $287$ & $[6165.4,25193.7]$ & $447$ & $[462511.3,1566498.9]$\tabularnewline
$128$ & $[508.7,2013.9]$ & $288$ & $[69854.4,576582.9]$ & $448$ & $[3028055.7,43007373.3]$\tabularnewline
$129$ & $[627.5,2938.7]$ & $289$ & $[85831.9,833562.6]$ & $449$ & $[3704212.7,61840462.3]$\tabularnewline
$130$ & $[380.2,2159]$ & $290$ & $[51049.4,602654.1]$ & $450$ & $[2191663.7,44457785.1]$\tabularnewline
$131$ & $[-143.8,-25.1]$ & $291$ & $[-28235.1,-6950.6]$ & $451$ & $[-1720915.1,-510007.6]$\tabularnewline
$132$ & $[-2378.6,-585.1]$ & $292$ & $[-649748.7,-77504.8]$ & $452$ & $[-47418079,-3297997]$\tabularnewline
$133$ & $[-3487.9,-728.6]$ & $293$ & $[-939363,-95269]$ & $453$ & $[-68177152.3,-4034349.5]$\tabularnewline
$134$ & $[-2555.9,-439.4]$ & $294$ & $[-678960.8,-56639.5]$ & $454$ & $[-49007903,-2386700.7]$\tabularnewline
$135$ & $[29.8,168.2]$ & $295$ & $[7830,31622.8]$ & $455$ & $[562153.8,1889852.3]$\tabularnewline
$136$ & $[674.9,2812.4]$ & $296$ & $[85968,731761.6]$ & $456$ & $[3590878.9,52260758.5]$\tabularnewline
$137$ & $[842.9,4127.6]$ & $297$ & $[105677.4,1057808.8]$ & $457$ & $[4392390.7,75132988.8]$\tabularnewline
$138$ & $[503.1,3012]$ & $298$ & $[62770.5,764281.6]$ & $458$ & $[2598036.8,54001551.8]$\tabularnewline
$139$ & $[-196.6,-35.3]$ & $299$ & $[-35393.1,-8814]$ & $459$ & $[-2074616.7,-619384.7]$\tabularnewline
$140$ & $[-3338.1,-786.4]$ & $300$ & $[-823731.4,-95361.7]$ & $460$ & $[-57576618.5,-3908786.8]$\tabularnewline
$141$ & $[-4876.3,-974.2]$ & $301$ & $[-1190324.3,-117153.1]$ & $461$ & $[-82765628.9,-4780550.8]$\tabularnewline
$142$ & $[-3561.2,-582.8]$ & $302$ & $[-859939.1,-69592.2]$ & $462$ & $[-59482126.5,-2827480.8]$\tabularnewline
$143$ & $[41.6,229.1]$ & $303$ & $[9914.5,39587.6]$ & $463$ & $[682173.3,2276624.3]$\tabularnewline
$144$ & $[910.1,3941.8]$ & $304$ & $[105684.5,926473]$ & $464$ & $[4253259.4,63407543]$\tabularnewline
$145$ & $[1124.6,5749.7]$ & $305$ & $[129788.5,1338464.6]$ & $465$ & $[5201251.1,91137970.4]$\tabularnewline
$146$ & $[678.4,4210.9]$ & $306$ & $[77138.4,966956.6]$ & $466$ & $[3076317.8,65493940.3]$\tabularnewline
$147$ & $[-266.7,-49]$ & $307$ & $[-44251.3,-11144.6]$ & $467$ & $[-2497409.1,-751034.2]$\tabularnewline
$148$ & $[-4624.7,-1042.3]$ & $308$ & $[-1041065.2,-116985.8]$ & $468$ & $[-69800199.5,-4626133]$\tabularnewline
$149$ & $[-6763.2,-1294.5]$ & $309$ & $[-1503988.2,-143701.6]$ & $469$ & $[-100318436.6,-5657087.8]$\tabularnewline
$150$ & $[-4944.7,-778.7]$ & $310$ & $[-1086266.6,-85375]$ & $470$ & $[-72083798.5,-3345549.9]$\tabularnewline
$151$ & $[57.6,309.9]$ & $311$ & $[12518.4,49432.9]$ & $471$ & $[826528.4,2738641.3]$\tabularnewline
$152$ & $[1195.9,5427.6]$ & $312$ & $[129443.3,1169130.9]$ & $472$ & $[5030162.9,76808707.8]$\tabularnewline
$153$ & $[1488.1,7941.5]$ & $313$ & $[159014.4,1688817.3]$ & $473$ & $[6150820.5,110381427.8]$\tabularnewline
$154$ & $[887.6,5786.6]$ & $314$ & $[94402.7,1219371.2]$ & $474$ & $[3636948.3,79305801.7]$\tabularnewline
$155$ & $[-359.3,-67.5]$ & $315$ & $[-55187.6,-14052]$ & $475$ & $[-3002130.3,-909266]$\tabularnewline
$156$ & $[-6381.7,-1380.4]$ & $316$ & $[-1312317.8,-143222.3]$ & $476$ & $[-84490797,-5468131.9]$\tabularnewline
$157$ & $[-9309.6,-1709.2]$ & $317$ & $[-1895058.2,-175850.3]$ & $477$ & $[-121408020.7,-6685462.1]$\tabularnewline
$158$ & $[-6788.5,-1021.7]$ & $318$ & $[-1368142.9,-104404.6]$ & $478$ & $[-87220473.3,-3952858.6]$\tabularnewline
$159$ & $[79.1,416.3]$ & $319$ & $[15762.8,61575.6]$ & $479$ & $[999910.2,3289835.9]$\tabularnewline
\end{longtable}
\par\end{center}

\end{document}